\documentclass[12pt]{amsart}

\allowdisplaybreaks

\usepackage[english]{babel}
\usepackage{amsthm}
\usepackage[pdftex,paper=a4paper,portrait=true,textwidth=450pt,textheight=675pt,tmargin=3cm,marginratio=1:1]{geometry}
\usepackage{graphicx}
\usepackage[colorinlistoftodos]{todonotes}
\usepackage{amsfonts}
\usepackage{amsmath}
\usepackage{amsthm}
\usepackage{amssymb}
\usepackage{bbm}
\usepackage{eufrak}
\usepackage{cancel}
\usepackage{color}
\usepackage[curve]{xypic}
\usepackage[all,color]{xy}

\newtheorem{innercustomthm}{Theorem}
\newenvironment{customthm}[1]
  {\renewcommand\theinnercustomthm{#1}\innercustomthm}
  {\endinnercustomthm}

\newtheorem{innercustomconj}{Conjecture}
\newenvironment{customconj}[1]
  {\renewcommand\theinnercustomconj{#1}\innercustomconj}
  {\endinnercustomconj}

\newtheorem*{theorem*}{Theorem}
\newtheorem{theorem}{Theorem}[section]
\newtheorem{corollary}[theorem]{Corollary}
\newtheorem{proposition}[theorem]{Proposition}
\newtheorem{propositiondefinition}[theorem]{Proposition-Definition}
\newtheorem{lemma}[theorem]{Lemma}
\newtheorem{conjecture}[theorem]{Conjecture}
\newtheorem{mainconjecture}[theorem]{Main Conjecture}

\theoremstyle{definition}
\newtheorem{definition}[theorem]{Definition}
\newtheorem{example}[theorem]{Example}
\newtheorem{remark}[theorem]{Remark}
\newtheorem{assumption}[theorem]{Assumption}

\theoremstyle{property}

\DeclareFontFamily{OT1}{rsfs}{}
\DeclareFontShape{OT1}{rsfs}{n}{it}{<-> rsfs10}{}
\DeclareMathAlphabet{\curly}{OT1}{rsfs}{n}{it}

\newcommand\C{\mathbb C}

\newcommand\FF{\mathbb F}
\newcommand\GG{\mathbb G}

\newcommand\LL{\mathbb L}
\newcommand\PP{\mathbb P}
\newcommand\Q{\mathbb Q}
\newcommand\R{\mathbb R}
\newcommand\Z{\mathbb Z}

\newcommand\cC{\mathcal C}

\newcommand\cE{\mathcal E}
\newcommand\F{\mathcal F}
\newcommand\G{\mathcal G}

\renewcommand\L{\mathcal L}
\newcommand\M{\mathcal M}
\newcommand\N{\mathcal N}
\renewcommand\O{\mathcal O}
\newcommand\cQ{\mathcal Q}
\newcommand\ccR{\mathcal R}

\newcommand\T{\mathcal T}
\newcommand\X{\mathcal X}

\newcommand\sfE{\mathsf E}
\newcommand\sfG{\mathsf G}

\newcommand\sfQ{\mathsf Q}

\newcommand\sfV{\mathsf V}
\newcommand\sfW{\mathsf W}
\newcommand\sfZ{\mathsf Z}

\newcommand\vir{\mathrm{vir}}
\newcommand\vd{\mathrm{vd}}
\newcommand\comb{\mathrm{comb}}
\newcommand\ttop{\mathrm{top}}
\newcommand\slice{\mathrm{slice}}
\newcommand\DT{\mathrm{DT}}
\newcommand\SL{\mathrm{SL}}

\newcommand\Quot{\mathrm{Quot}}

\newcommand\bsdelta{\boldsymbol \delta}
\newcommand\bschi{\boldsymbol \chi}
\newcommand\bspi{\boldsymbol \pi}
\newcommand\bsvpi{\boldsymbol \varpi}
\newcommand\bslambda{\boldsymbol \lambda}
\newcommand\bsmu{\boldsymbol \mu}

\newcommand\mdot{{\scriptscriptstyle\bullet}}

\newcommand\Into{\ar@{^(->}[r]<-.3ex>}

\newcommand\rk{\operatorname{rk}}

\newcommand\tr{\operatorname{tr}}

\newcommand\ch{\operatorname{ch}}

\newcommand\Hom{\operatorname{Hom}}
\renewcommand\hom{\curly H\!om}
\newcommand\Ext{\operatorname{Ext}}

\newcommand\ext{\curly Ext}

\newcommand\Pic{\operatorname{Pic}}

\newcommand\beq[1]{\begin{equation}\label{#1}}
\newcommand\eeq{\end{equation}}
\newcommand\beqa{\begin{eqnarray*}}
\newcommand\eeqa{\end{eqnarray*}}
\begin{document}
\title[Rank 2 sheaves on toric 3-folds]{Rank 2 sheaves on toric 3-folds: classical and virtual counts}
\author[A. Gholampour, M. Kool, B. Young]{Amin Gholampour, Martijn Kool, and Benjamin Young \vspace{-5mm}}
\maketitle

\begin{abstract}
Let $\M$ be the moduli space of rank 2 stable torsion free sheaves with Chern classes $c_i$ on a smooth 3-fold $X$. When $X$ is toric with torus $T$, we describe the $T$-fixed locus of the moduli space. Connected components of $\M^T$ with constant reflexive hulls are isomorphic to products of $\PP^1$. We mainly consider such connected components, which typically arise for any $c_1$, ``low values'' of $c_2$, and arbitrary $c_3$. 

In the \emph{classical part} of the paper, we introduce a new type of combinatorics called double box configurations, which can be used to compute the generating function $\sfZ(q)$ of topological Euler characteristics of $\M$ (summing over all $c_3$). The combinatorics is solved using the double dimer model in a companion paper. This leads to explicit formulae for $\sfZ(q)$ involving the MacMahon function. 

In the \emph{virtual part} of the paper, we define Donaldson-Thomas type invariants of toric Calabi-Yau 3-folds by virtual localization. The contribution to the invariant of an individual connected component of the $T$-fixed locus is in general not equal to its signed Euler characteristic due to $T$-fixed obstructions. Nevertheless, the generating function of all invariants is given by $\sfZ(q)$ up to signs.



\end{abstract}

\tableofcontents

\section{Introduction}

Let $X$ be an $n$-dimensional smooth projective variety over $\C$ with polarization $H$. Let $\M_{X}^{H}(r,c_1, \ldots, c_n)$ denote the moduli space of rank $r$ $\mu$-stable torsion free sheaves on $X$ with Chern classes $c_1, \ldots, c_n$.\footnote{A coherent sheaf $\F$ is called torsion free if it has no non-zero subsheaves of dimension $< \dim X$.} See \cite{HL} for the construction of the moduli space and its history. In this paper, we do \emph{not} consider strictly semistable sheaves so this moduli space may be non-compact.\footnote{One way to ensure the absence of strictly semistable sheaves is by taking $\gcd(r,c_1 H^{n-1})=1$. In this case, the moduli space is compact.} Fixing $c_1, \ldots, c_{n-1}$, one can consider the generating function
\begin{equation} \label{tocalc}
\sfZ_{X,H,r,c_1, \ldots, c_{n-1}}(q) = \sum_{c_n \in \Z} e(\M_{X}^{H}(r,c_1, \ldots, c_n)) q^{c_n},
\end{equation} 
where $e(\cdot)$ denotes topological Euler characteristic. This is a Laurent series in $q$ when $n$ is even and $q^{-1}$ when $n$ is odd \cite[Prop.~3.6]{GK1}. These generating functions are fascinating objects with links to combinatorics, number theory, and representation theory. We mention a few among many references. For $n=1$, generating functions of Poincar\'e polynomials were calculated by G.~Harder and M.~S.~Narasimhan \cite{HN}, U.V.~Desale and S.~Ramanan \cite{DR}, and M.~F.~Atiyah and R.~Bott \cite{AB}. For $n=2$ and $r=1$, generating functions for Poincar\'e polynomials were computed by L.~G\"ottsche \cite{Got1}. For $n=2$ and $r>1$, these generating functions are related to (quasi-)modular forms and have been computed in numerous examples \cite{Kly2, Yos, Man1, Man2, Man3, Moz}. This case is related to the S-duality conjecture of physics \cite{VW}. For any $n$ and $r=1$, these generating functions lead to counting $n$-dimensional boxes as shown by J.~Cheah \cite{Che}.

For $n=3$ and $X$ satisfying $H^0(K_{X}^{-1}) \neq 0$, R.~P.~Thomas constructed a perfect obstruction theory on $\M_{X}^{H}(r,c_1, c_2, c_3)$ \cite{Tho}. When the moduli space is compact, this leads to a virtual cycle, which can be used to define deformation invariants. When $X$ is in addition Calabi-Yau, the perfect obstruction theory is symmetric so the virtual cycle is 0-dimensional. Its degree is known as a Donaldson-Thomas invariant. K.~Behrend \cite{Beh} showed the DT invariant is also equal to the Euler characteristic of $\M_{X}^{H}(r,c_1, c_2, c_3)$ weighted by a constructible function $\nu$ 
$$
\deg([\M_{X}^{H}(r,c_1, c_2, c_3)]^{\vir}) = e(\M_{X}^{H}(r,c_1, c_2, c_3), \nu) = \sum_{k \in \Z} k \ e(\nu^{-1}(\{k\})).
$$ 
Therefore, the generating function \eqref{tocalc} is related to DT theory. When $X$ is a toric Calabi-Yau 3-fold and $r=1$, D.~Maulik, N.~Nekrasov, A.~Okounkov, and R.~Pandharipande \cite{MNOP1} showed that the generating function of DT invariants (defined by $T$-localization) is given by counting the $T$-fixed points of $\M_{X}^{H}(1,c_1, c_2, c_3)$ with signs.\footnote{A toric Calabi-Yau 3-fold cannot be projective. In this case the DT invariants are defined by the localization formula of \cite{GP}.} The signs can be seen as a consequence of weighing by the Behrend function \cite{BF}. 

This paper is split up in two parts, which can be read more or less independently. The \emph{classical part} is about \eqref{tocalc} in the case $r=2$ and $X$ is a toric 3-fold. The \emph{virtual part} is concerned with rank 2 DT type invariants of a toric Calabi-Yau 3-fold and how these invariants relate to the Euler characteristics discussed in the classical part. \\

\noindent {\Large{\textit{Part I: Classical}}} \\

Little is known about the generating function \eqref{tocalc} for $n>2$ and $r>1$. We study the case $n=3$, $r=2$, and $X$ is \emph{toric}. In this case, $X$ contains a dense open torus $T$ acting on $X$ and this action lifts to the moduli space. Elements of the fixed locus of the moduli space admit a $T$-equivariant structure (unique up to tensoring by a character of $T$). Equivariant sheaves on toric varieties were studied in depth by A.~Klyachko \cite{Kly1, Kly2}. He used his methods to calculate \eqref{tocalc} for $X = \PP^2$ and $r=2$ obtaining the holomorphic part of a quasi-modular form of weight 3/2 \cite{Kly2, VW}. Other localization calculations on (stacky/ordinary) toric surfaces appear in \cite{Per2, Cho, GJK, Koo2}. 

Let $(\cdot)^* = \hom(\cdot, \O_X)$ denote the dual. Any torsion free sheaf $\F$ can be naturally embedded in its double dual $\ccR = \F^{**}$, known as its reflexive hull. This leads to a short exact sequence
$$
0 \longrightarrow \F \longrightarrow \ccR \longrightarrow \cQ \longrightarrow 0,
$$
where $\cQ$ is a sheaf of dimension $\leq 1$. Reflexive hulls are \emph{reflexive sheaves}.\footnote{A coherent sheaf $\F$ is called reflexive if the natural map $\F \rightarrow \F^{**}$ is an isomorphism.}  A reflexive sheaf is torsion free, but much easier to describe than a general torsion free sheaf, because it is determined by its restriction to the complement of any closed subset of codimension $\geq 2$. Reflexive sheaves on a 3-fold are locally free outside a 0-dimensional subscheme. For general theory on reflexive sheaves, we refer to \cite{Har2}. Klyachko showed that $T$-equivariant rank 2 reflexive sheaves have a straightforward description: they are described by attaching a flag of $\C^2$ to each ray of the fan of $X$ (his description works for any rank on any toric $n$-fold). Denote by $\N_{X}^{H}(2,c_1,c_2, c_3)$ the moduli space of rank 2 $\mu$-stable reflexive sheaves on $X$ with Chern classes $c_1$, $c_2$, $c_3$. Taking double dual gives a map to the scheme theoretic disjoint union
\begin{equation} \label{ddmap}
(\cdot)^{**} : \M_{X}^{H}(2,c_1,c_2,c_3) \longrightarrow \bigsqcup_{c_2',c_3'} \N_{X}^{H}(2,c_1,c_2',c_3').
\end{equation}
Reflexive hulls of members of a flat family do not need to form a flat family (already on surfaces), so this map is \emph{not} a morphism. Nevertheless, by a result of J.~Koll\'ar \cite{Kol}, this map is \emph{constructible}.\footnote{More precisely, the domain can be written as a union of locally closed subschemes $C_i$ such that on each $C_i$ the map $(\cdot)^{**}$ is a morphism.} This is very useful for calculating Euler characteristics by cut-paste methods. On a surface, the cokernel $\cQ$ is 0-dimensional and the double dual map leads to a product formula expressing \eqref{tocalc} as a ``point contribution'' times the generating function of reflexive (i.e.~locally free) sheaves \cite{Got2}. Two interesting things happen when $n=3$ and $r=2$: the cokernel $\cQ$ need not be 0-dimensional and $\ccR$ need not be locally free. Since we are interested in Euler characteristics, we only need to understand \eqref{ddmap} at the level of $T$-fixed points, where we can use Klyachko's description of $T$-equivariant torsion free and reflexive sheaves. As already mentioned, $T$-equivariant reflexive sheaves are very well-behaved. Generating functions of $\N_{X}^{H}(2,c_1,c_2',c_3')$ were studied in \cite{GK1}.

Equivariant torsion free sheaves are much harder to enumerate than equivariant reflexive sheaves. At the level of closed points, the fibre of \eqref{ddmap} over a reflexive sheaf $\ccR$ is the Quot scheme
$$
\Quot_X(\ccR,c_2'',c_3''),
$$
of cokernels $\ccR \twoheadrightarrow \cQ$ where $\cQ$ has dimension $\leq 1$ and Chern classes $c_2''$, $c_3''$.\footnote{The Chern classes $c_i$, $c_i'$, $c_i''$ are related by equation \eqref{Chernrel}.} Suppose $\ccR$ is $T$-equivariant, then $T$ acts on $\Quot_X(\ccR,c_2'',c_3'')$. The connected components of the $T$-fixed locus of $\Quot_X(\ccR,c_2'',c_3'')$ are indexed by a new type of combinatorial objects, which we call double box configurations (Definition \ref{doubleboxconf}). Moreover, each connected component is isomorphic to a product of $\PP^1$'s (Proposition \ref{prodP1s}). When $\cQ$ is 0-dimensional, i.e.~$c_2'' = 0$, we solve the combinatorics completely in a companion paper \cite{GKY}. 

One basic building block for our computations is $\Quot_{\C^3}(\ccR,n)$: the Quot scheme of 0-dimensional quotients $\ccR \twoheadrightarrow \cQ$ of length $n$, where $\ccR$ is a $T$-equivariant rank 2 reflexive sheaf on $\C^3$. If $\ccR$ is locally free, it decomposes\footnote{Due to Proposition \ref{rank2singular}.} as a sum of two line bundles $\ccR \cong L_1 \oplus L_2$, which gives an additional $\C^*$-action by scaling the factors. Localization with respect to this action gives a well-known result
\begin{equation} \label{locfree}
\sum_{n = 0}^{\infty} e(\Quot_{\C^3}(\ccR,n)) q^n = M(q)^2,
\end{equation}
where $M(q)$ denotes the MacMahon function counting 3D partitions. When $\ccR$ is singular, i.e.~not locally free hence indecomposable$^{7}$, this generating function is much harder to calculate. Then there are integers $v_1, v_2, v_3>0$, which roughly speaking encode the weights of the three homogeneous generators of $H^0(\ccR)$. In terms of these weights, we have: 
\begin{customthm}{A} \cite{GKY, GK2} \label{A}
$$
\sum_{n = 0}^{\infty} e(\Quot_{\C^3}(\ccR,n)) q^n = M(q)^2 \prod_{i=1}^{v_1} \prod_{j=1}^{v_2} \prod_{k=1}^{v_3} \frac{1-q^{i+j+k-1}}{1-q^{i+j+k-2}}.
$$
\end{customthm}
The triple product factor is the generating function of 3D partitions in a box of lengths $v_1$ by $v_2$ by $v_3$ and is therefore a polynomial in $q$ \cite[(7.109)]{Sta}. The length of the singularity of $\ccR$ is $v_1v_2v_3$ (Proposition \ref{rank2singular}). Therefore, in the singular case we get the ``locally free answer'' $M(q)^2$ times a polynomial counting partitions ``in the singularity of $\ccR$''. There are two proofs of this theorem:
\begin{itemize}
\item Using $\C^{*3}$-localization on $\Quot_{\C^3}(\ccR,n)$ and the double dimer model. The fixed locus of $\Quot_{\C^3}(\ccR,n)$ is described in this paper. The link to double dimers and how it leads to a proof of Theorem \ref{A} form the topic of the (purely combinatorial) companion paper \cite{GKY}. 
\item Using Hall algebra techniques, a non-toric version of Theorem \ref{A} was recently proved by the first two authors in \cite{GK2}. This proof provides a geometric explanation why counting inside the singularity of $\ccR$ and the MacMahon function occur. 
\end{itemize}
The combinatorial proof \cite{GKY} was found before the geometric proof \cite{GK2}. We briefly discuss these two proofs in Section \ref{combsection}.

What does this result tell us about generating function \eqref{tocalc} for $n=3$ and $r=2$? Recall that we fix $c_1,c_2$ and sum over all $c_3$. If $c_2$ is chosen such that $c_2 H$ is minimal with the property that there exists a rank 2 $\mu$-stable torsion free sheaf on $X$ with Chern classes $c_1, c_2$, then $\cQ$ is automatically 0-dimensional (Proposition \ref{minc2a}). In this case, we get a structure theorem for \eqref{tocalc} (Proposition \ref{minc2b}). If in addition the polyhedron $\Delta(X)$ of the toric 3-fold $X$ is small in the sense of Definition \ref{small} (e.g.~when $\rk \Pic(X) \leq 3$), then all $T$-equivariant reflexive hulls in the moduli space are isolated $T$-fixed points and the formula simplifies (Theorem \ref{Eulermin}). A special case of this setting is $X = \PP^3$ and $c_2=1$. In principle, our approach also works for more general $c_2$, but then the combinatorics becomes increasingly complex. We illustrate this for $X = \PP^3$ and $c_2=2$.
\begin{customthm}{B}\label{B}
For $X = \PP^3$ and $c_1 = -1$, we have 
\begin{align*}
\sfZ_{c_2=1}(q) &= 4(q^{-1}+q)M(q^{-2})^8, \\
\sfZ_{c_2=2}(q) &= 12 \Bigg( \frac{2q^{-4}-q^{-2}+1-4q^2+3q^4+5q^8}{(1-q^2)^2} \Bigg) M(q^{-2})^{8}.
\end{align*}
\end{customthm}

\noindent {\Large{\textit{{Part II: Virtual}}}} \\

Let $Y$ be a smooth toric Calabi-Yau 3-fold. Since $Y$ is non-compact, we choose a toric compactification $Y \subset X$ with $H^0(K_{X}^{-1}) \neq 0$ and a polarization $H$ on $X$. Define the $T$-invariant open subset
$$
\M_{Y \subset X}^{H}(2,c_1, c_2, c_3) \subset \M_{X}^{H}(2,c_1, c_2, c_3)
$$
of torsion free sheaves $\F$ on $X$ such that $\F^{**} / \F \cong \cQ$ is supported on $Y$. Openness is proved in Proposition \ref{openness}. Unlike the rank 1 case, this moduli space depends on the choice of $X$ and $H$. This reflects dependence of stability on choice of polarization for higher rank. When $\cC$ is a compact connected component of $\M_{X}^{H}(2,c_1, c_2, c_3)^T$, it has an induced virtual cycle by \cite{GP} and one can consider the expression
\begin{equation} \label{Cvir}
\int_{[\cC]^{\vir}} \frac{1}{e(N^{\vir})},
\end{equation}
where $e(\cdot)$ denotes $T$-equivariant Euler class and $N^{\vir}$ denotes the virtual normal bundle of $\cC$. These integrals depend on equivariant parameters $s_1,s_2,s_3$. We are interested in connected components $\cC$, which also satisfy $\cC \subset \M_{Y \subset X}$. Unfortunately, the DT complex restricted to such a component $\cC$ is \emph{not} automatically symmetric. However, if we suppose in addition that all closed points of $\cC$ have the same (i.e.~isomorphic) reflexive hull $\ccR$, then we can write \eqref{Cvir} as
$$
e(R\Hom(\ccR,\ccR)_0 \otimes \O_\cC) \times \int_{[\cC]^{\vir}} \frac{1}{e(N^{\vir}) \ e(R\Hom(\ccR,\ccR)_0 \otimes \O_\cC)},
$$ 
where the second factor of this product is a \emph{number} for $s_1+s_2+s_3=0$ (Proposition-Definition \ref{propdef}). We denote this number by $\DT(\cC)$. The reason for this is that after removing $-R\Hom(\ccR,\ccR)_0$ from the DT complex on $\cC$ in the $K$-group, it becomes symmetric (Proposition \ref{SD}). We think of $\DT(\cC)$ as the contribution of $\cC$ to rank 2 DT type invariants of $Y$.\footnote{Another approach to DT invariants of $Y$ is by using ``framing at infinity'' as in the work of D.~Oprea \cite{Opr}. In this case, the generating function becomes a product of rank 1 generating functions by using an extra torus action on the framing. In our theory, such a splitting does not occur due to the presence of singularities in the reflexive hull.} We have a precise conjecture for generating functions of $\DT(\cC)$ in any Chern classes on any toric Calabi-Yau 3-fold (Main Conjecture \ref{mainconj}).

For $Y = \C^3$ and fixed $T$-equivariant rank 2 reflexive hull $\ccR$ on $\C^3$, consider all connected components $\cC$ for which all its elements have constant reflexive hull $\ccR$ and 0-dimensional cokernel. Then the generating function of all $\DT(\cC)$ is the degree 0 rank 2 equivariant vertex, which we denote by $\sfW_{\ccR,  \varnothing, \varnothing, \varnothing}(q)$. We conjecture an explicit formula for it (Conjecture \ref{edgelessconj}):
\begin{customconj}{C}\label{C}
\begin{equation*}
\sfW_{\ccR,  \varnothing, \varnothing, \varnothing}(q) \Big|_{s_1+s_2+s_3=0} = \left\{ \begin{array}{cc} M(q)^2 & \mathrm{if \ } \ccR \mathrm{ \ is \ locally \ free } \\ M(q)^2 \prod_{i=1}^{v_1} \prod_{j=1}^{v_2} \prod_{k=1}^{v_3} \frac{1-q^{i+j+k-1}}{1-q^{i+j+k-2}} & \mathrm{if \ } \ccR \mathrm{ \ is \ singular}. \end{array} \right.
\end{equation*}
\end{customconj}
This can be viewed as the degree 0 part of rank 2 DT theory on smooth toric 3-folds. This formula is in accordance with the expectation that higher rank DT invariants can be expressed in terms of rank 1 DT invariants and ``point contributions''.\footnote{Mentioned to us by D.~Maulik and J.~Manschot.} Note that unlike the rank 1 case, there are no signs in this formula (essentially because rank 2 is even). Comparing to \eqref{locfree} and Theorem A, Conjecture C suggest $\DT(\cC) = e(\cC)$, but this is \emph{not} the case. In fact, $\DT(\cC)$ can be zero due to obstructions (Example \ref{abnormal}). The following is strong evidence for Conjecture C. Let $T_0 \subset T$ denote the ``Calabi-Yau torus'' defined by $t_1t_2t_3=1$. Let $\cC \subset \M_{Y \subset X}$ be a compact connected component of the $T_0$-fixed locus with constant reflexive hulls $\ccR$ and cokernels of dimension 0. Using $T_0$-localization and assuming analogs of two conjectures\footnote{One conjecture asserts smoothness of $\cC$.} used in the $T_0$-localization of stable pair theory of Pandharipande-Thomas \cite{PT2}, we prove 
\begin{equation*} 
\sum_{\cC_i \subset \cC^{\C^*}} \DT(\cC) = \sum_{\cC_i \subset \cC^{\C^*}} e(\cC_i) = e(\cC).
\end{equation*}
Here the sum is over all connected components of the $T/T_0 \cong \C^*$-fixed locus of $\cC$. See Remark \ref{finalrem} for the precise statement. Combinatorially, this means that whenever $\DT(\cC_i) = 0$ for some $\cC_i$, other ``buddy components'' $\cC_j$ in the same $\C^*$-fixed locus will compensate to recover the Euler characteristic.

Main Conjecture \ref{mainconj} is a generalization of Conjecture \ref{C} to arbitrary smooth toric Calabi-Yau 3-folds and Chern classes. It expresses generating functions of rank 2 DT type invariants in terms of Euler characteristics and signs. The signs only occur in the presence of legs, i.e.~when the cokernel is 1-dimensional. We provide the following evidence for Main Conjecture \ref{mainconj}:
\begin{itemize}
\item Explicit examples (Section \ref{nolegs}, \ref{O(-1,-1)}).
\item A proof in the case of ``expected obstructions'' (Theorem \ref{Tthm}).
\item A proof using $T_0$-localization based on the analogs of two conjectures of \cite{PT2} (Theorem \ref{T0thm}).
\end{itemize}

\noindent \textbf{Notation.} For any sets $A,B$, we denote projections to the factors by $p_A : A \times B \rightarrow A$ and $p_B : A \times B \rightarrow B$. If $A,B$ are schemes with coherent sheaves $\cE,\F$, we write $\cE \boxtimes \F := p_{A}^{*} \cE \otimes p_{B}^{*}\F$. We denote the closed points of a finite type $\C$-scheme $B$ by $B_{\mathrm{cl}}$. \\

\noindent \textbf{Acknowledgements.} We would like to thank K.~Behrend, J.~Bryan, J.~Manschot, D.~Maulik, and R.~P.~Thomas for useful discussions. We would also like to thank the anonymous referees whose comments led to an immense improvement of the exposition, e.g.~by splitting the paper in two parts.

A.G.~was partially supported by NSF grant DMS-1406788. M.K.~was supported by EP/G06170X/1, ``Applied derived categories'' (while at Imperial College), a PIMS postdoctoral fellowship (CRG Geometry and Physics, while at UBC), and NWO-GQT and Marie Sk{\l}odowska-Curie Project 656898 (INVLOCCY) (while at Utrecht).

\hfill
\newline

\noindent {\Large{\textit{{Part I: Classical}}}} \\
\addcontentsline{toc}{section}{\textit{Part I: Classical}}

\section{Fixed loci} \label{fixedloci}

In this section we recall Klyachko's description of $T$-equivariant torsion free and reflexive sheaves on a smooth projective toric 3-fold $X$. In the rank 2 case, we use it to describe (scheme-theoretically) the $T$-fixed locus of the moduli space of $\mu$-stable torsion free sheaves and deduce it is smooth (Theorem \ref{fixedlocithm}). 
We also give a combinatorial description of the fixed locus of the Quot scheme of a $T$-equivariant rank 2 reflexive hull in terms of new combinatorial objects called double box configurations (Definition \ref{doubleboxconf}).

\subsection{Toric 3-folds} \label{toric3folds}

Let $X$ be a smooth projective toric 3-fold with dense open torus $T$. We recall some basic facts and notation from toric geometry.
For a general introduction to toric varieties see \cite{Ful}. Let $\Delta(X)$ be the Newton polyhedron of $X$ determined by a polarization (e.g.~\cite[Sect.~4.1]{MNOP1}). Note that we are not actually fixing a polarization at this stage, so $\Delta(X)$ does not come with an embedding in the character lattice.
Denote the collection of vertices of $\Delta(X)$ by $V(X)$ and the collection of edges of $\Delta(X)$ by $E(X)$. The vertices $\alpha \in V(X)$ correspond bijectively to the cones of dimension 3 of the fan of $X$. We denote the edge between two adjacent vertices $\alpha, \beta \in V(X)$ by $\alpha\beta \in E(X)$. The edges $\alpha\beta \in E(X)$ correspond bijectively to the cones of dimension 2 of the fan of $X$. We denote by $U_\alpha \cong \C^3$ the $T$-invariant affine open subset corresponding to $\alpha$. Furthermore, we write $X_\alpha$ for the fixed point corresponding to $\alpha$ and $C_{\alpha \beta} \cong \PP^1$ for the $T$-invariant line joining $X_\alpha$ and $X_\beta$.  The toric divisors are labelled by the faces of $\Delta(X)$. We denote the collection of faces of $\Delta(X)$ by $F(X)$ and label its elements by $\rho \in F(X)$.\footnote{We identify $F(X)$ with the collection of rays of the fan of $X$ by sending a face to the ray it is transverse to.} The toric divisor corresponding to $\rho \in F(X)$ is denoted by $D_\rho$.

For any $\alpha \in V(X)$, there are coordinates $(x_1,x_2,x_3)$ on $U_\alpha \cong \C^3$ and $t=(t_1, t_2, t_3)$ on $T \cong \C^{*3}$ such that
\[
t \cdot (x_1, x_3, x_3) = (t_1 x_1, t_2 x_2, t_3 x_3).
\]
In these coordinates, the fixed point $X_\alpha$ has cotangent representation
\[
\langle dx_1 \rangle_\C \otimes t_1 + \langle dx_2 \rangle_\C \otimes t_2 + \langle dx_3 \rangle_\C \otimes t_3,
\]
where $\langle \cdot \rangle_{\C}$ denotes $\C$-span. We therefore have a $T$-equivariant isomorphism
\begin{equation} \label{Klocal}
K_X |_{U_\alpha} \cong \O_{U_\alpha} \otimes t_1 t_2 t_3.
\end{equation}
For each $\alpha\beta \in E(X)$, there exist integers $m_{\alpha\beta}$, $m_{\alpha\beta}'$ such that the normal bundle $N_{C_{\alpha\beta}/X}$ is given by
\[
N_{C_{\alpha\beta}/X} \cong \O_{\PP^1}(m_{\alpha\beta}) \oplus \O_{\PP^1}(m_{\alpha \beta}^{\prime}).
\]
The transition function between the charts $U_\alpha$, $U_\beta$ is given by
\begin{equation} \label{transf}
(x_1, x_2, x_3) \mapsto (x_{1}^{-1}, x_2 x_{1}^{-m_{\alpha\beta}}, x_3 x_{1}^{-m_{\alpha \beta}^{\prime}}),
\end{equation}
where $C_{\alpha\beta}$ is defined by $x_2=x_3=0$. Here we follow the notation of \cite{MNOP1}.

\subsection{Equivariant torsion free sheaves} \label{fixedloci1}

We give Klyacho's description of $T$-equivariant torsion free sheaves on a smooth projective toric 3-fold $X$ \cite{Kly1, Kly2}. Klyachko's work was further developed by M.~Perling \cite{Per1} and we will mostly follow his notation. Most constructions of Sections \ref{fixedloci1}--\ref{fixedlocigeneral} work for $X$ of any dimension, but we stick to dimension 3 for notational convenience.

Let $\F$ be a $T$-equivariant quasi-coherent sheaf on $X$ and $\alpha \in V(X)$. Let $(x_1, x_2, x_3)$ be coordinates on $U_\alpha \cong \C^3$ as in Section \ref{toric3folds}. We first study the restriction $\F|_{U_\alpha}$. The global section functor $H^0(\cdot)$ gives an equivalence between the categories of quasi-coherent sheaves on $U_\alpha$ and $\C[x_1,x_2,x_3]$-modules. It is not hard to extend this to an equivalence between the categories of $T$-equivariant quasi-coherent sheaves on $U_\alpha$ and $\C[x_1,x_2,x_3]$-modules with regular $T$-action \cite{Per1}. 

Given a $T$-equivariant quasi-coherent sheaf $\F_\alpha$ on $U_\alpha$, we have a decomposition into weight spaces according to the character group $X(T) = \Z^3$
\[
H^0(\F_\alpha) = \bigoplus_{(k_1,k_2,k_3) \in \Z^3} F_\alpha(k_1,k_2,k_3).
\]
The sheaf $\F_\alpha$ is coherent if and only if $H^0(\F_\alpha)$ has a finite number of homogeneous generators. Moreover $\F_\alpha$ is torsion free if and only if the multiplication maps $x_1, x_2, x_3$ are injective on all weight spaces $F_\alpha(k_1,k_2,k_3)$ \cite{Per1}. Therefore, the category of $T$-equivariant rank $r$ torsion free sheaves $\F_\alpha$ on $U_\alpha$ is equivalent to the category of collections of complex vector spaces $\{F_\alpha(k_1,k_2,k_3)\}_{(k_1,k_2,k_3) \in \Z^3}$ satisfying:
\begin{itemize}
\item The vector spaces $F_\alpha(k_1,k_2,k_3)$ form a triple filtration
\begin{align*}
&F_\alpha(k_1,k_2,k_3) \subset F_\alpha(k_1+1,k_2,k_3), \\
&F_\alpha(k_1,k_2,k_3) \subset F_\alpha(k_1,k_2+1,k_3), \\
&F_\alpha(k_1,k_2,k_3) \subset F_\alpha(k_1,k_2,k_3+1).
\end{align*}
\item For all $k_1, k_2, k_3 \gg 0$ we have $F_\alpha(k_1,k_2,k_3) = \C^r$.
\item Whenever $k_1 \ll 0$ or $k_2 \ll 0$ or $k_3 \ll 0$ we have $F_\alpha(k_1,k_2,k_3) = 0$.
\end{itemize}
We refer to $\{F_\alpha(k_1,k_2,k_3)\}_{(k_1,k_2,k_3) \in \Z^3}$ as a rank $r$ $\sigma$-family.\footnote{This terminology is due to Perling \cite{Per1}. Here $\sigma$ denotes the cone of maximal dimension corresponding to vertex $\alpha$. In our notation $\alpha$-family would perhaps be better terminology.} The maps in the category of $T$-equivariant rank $r$ torsion free sheaves on $U_\alpha$ are $T$-equivariant morphisms and the maps in the category of rank $r$ $\sigma$-families are endomorphisms of $\C^r$ compatible with the triple filtrations.

A $T$-equivariant torsion free sheaf $\F$ on $X$ gives rise to $T$-equivariant torsion free sheaves $\F_\alpha := \F|_{U_\alpha}$ and corresponding $\sigma$-families
\[
{\bf{F}} := \{F_\alpha(k_1,k_2,k_3)\}_{(k_1,k_2,k_3) \in \Z^3, \alpha \in V(X)}.
\]
Conversely, given such a collection of $\sigma$-families, when do the corresponding sheaves $\F_\alpha$ glue to a sheaf $\F$ on $X$? Note that for fixed $k_2, k_3$, the vector space $F_\alpha(k_1,k_2,k_3) \subset \C^r$ does not change for $k_1 \gg 0$. We denote this ``limiting vector space'' by $F_\alpha(\infty,k_2,k_3)$. It is not hard to show that the gluing conditions are given by the following equations \cite{Kly2, Per1}
\begin{equation} \label{glue}
F_\alpha(\infty,k_2,k_3) = F_\beta(\infty,k_2,k_3),
\end{equation}
for all $\alpha, \beta \in V(X)$ and $k_2, k_3 \in \Z$. Here the coordinates are chosen such that $C_{\alpha\beta}$ is defined by equations $x_2=x_3=0$ in $U_\alpha$ and $x_{2}' = x_{3}' = 0$ in $U_\beta$. This describes an equivalence between the category of $T$-equivariant rank $r$ torsion free sheaves on $X$ and collections of rank $r$ $\sigma$-families satisfying the gluing conditions (\ref{glue}) \cite{Per1}. The maps in both categories are the obvious ones.

\begin{definition}
Let $\F$ be a $T$-equivariant torsion free sheaf on a smooth projective toric 3-fold $X$ with corresponding collection of $\sigma$-families ${\bf{F}}$. We define the \emph{characteristic function} ${\bschi}_\F$ of $\F$ as follows
\begin{align*}
&{\bschi}_\F := \{\chi_{\F}^{\alpha}\}_{\alpha \in V(X)}, \\
&\chi_{\F}^{\alpha} : \Z^3 \longrightarrow \Z, \ \chi_{\F}^{\alpha}(k_1,k_2,k_3) := \dim(F_\alpha(k_1,k_2,k_3)). 
\end{align*}
We denote the collection of characteristic functions of $T$-equivariant torsion free sheaves on $X$ by $\X$. 
\end{definition}

Characteristic functions of members of a flat family of $T$-equivariant coherent sheaves are locally constant over the base \cite{Koo1}. One can construct projective coarse moduli spaces of  $T$-equivariant torsion free sheaves with fixed characteristic function using GIT. Moreover, one can choose a GIT stability condition for which GIT stability and Gieseker stability coincide \cite{Koo1}. 

Let $\F$ be a $T$-equivariant torsion free sheaf on $X$. Then there is an explicit formula for the Chern character $\ch(\F)$ in terms of the characteristic function ${\bschi}_\F$. This formula is due to Klyachko \cite{Kly1, Kly2}. An alternative way for obtaining such a formula is by $T$-equivariant d\'evissage. Concretely, this works as follows. The cohomology ring $H^{2*}(X,\Z)$ is explicitly determined by the fan (or polyhedron) of $X$. It is the free $\Z$-module generated by the $T$-invariant closed subsets
\[
D_\rho, \ C_{\alpha\beta}, \ X_\alpha
\]
modulo certain relations \cite{Ful}. For example $X_\alpha = X_\beta$ for all $\alpha, \beta \in V(X)$ is such a relation. There are three linear relations among the toric prime divisors $D_\rho$ and certain relations among the $C_{\alpha\beta}$. The $T$-equivariant Picard group $\Pic^T(X)$ is generated by $D_\rho$ \emph{without} relations. Inside the $T$-equivariant $K$-group $K_{0}^{T}(X)$, there is an explicit way of writing $\F$ as a finite sum of elements of the form
\[
\O_{D_\rho} \otimes L, \ \O_{C_{\alpha\beta}} \otimes L, \ \O_{X_{\alpha}} \otimes L,
\]
for $L \in \Pic^T(X)$. The algorithm for getting such a decomposition is explicitly described in \cite{GK1} for the case $X = \PP^3$, $r=2$, and $\F$ reflexive. See Example \ref{P3reflch} below for the equations. It is straightforward to generalize the algorithm to any $X$ and $\F$ a $T$-equivariant rank $r$ torsion free sheaf on $X$. For any ${\bschi} \in \X$, we write $\ch({\bschi})$ for the Chern character determined by a characteristic function ${\bschi}$. Characteristic function determines Chern character and hence Hilbert polynomial (by Hirzebruch-Riemann-Roch). For a fixed Chern character $\ch$, we denote by $$\X_{\ch} \subset \X$$ the collection of characteristic functions ${\bschi}$ satisfying $\ch({\bschi}) = \ch$.

\subsection{Equivariant reflexive sheaves} \label{equivsh}

Let $X$ be a smooth projective toric 3-fold with polarization $H$. Recall that a coherent sheaf $\F$ on $X$ is $\mu$-semistable (resp.~$\mu$-stable) if and only if $\F$ is torsion free and we have $\mu_\cE \leq \mu_\F$ (resp.~$\mu_\cE < \mu_\F$) for all subsheaves $\cE \subset \F$ with $0< \rk(\cE) < \rk(\F)$ \cite[Def.~1.2.12]{HL}. Here the slope $\mu_\F$ is defined by
$$
\mu_\F := \frac{c_1 H^2}{\rk(\F)}.
$$
Recall from the introduction that we always work on the open subset of $\mu$-stable sheaves and disregard strictly $\mu$-semistable sheaves completely. A torsion free sheaf $\F$ on $X$ is $\mu$-stable if and only if its reflexive hull $\F^{**}$ is $\mu$-stable.\footnote{The reason is that for all $\cE \subset \F$ the sheaves $\cE$ and $\cE^{**}$ are equal on the complement of a closed subset of codimension $\geq 2$ so $\mu_\cE = \mu_{\cE^{**}}$.} 
Klyachko \cite{Kly2} found out that $T$-equivariant reflexive sheaves on $X$ have a particularly straightforward description. Moreover, their stability is easily described in terms of their characteristic function 
\cite{Koo1}. We will now describe this.

We denote by $l$ the number of faces of the polyhedron $\Delta(X)$. Note that $l$ equals the rank of $\Pic^T(X)$. Klyachko \cite{Kly2} gives an equivalence between the category of $T$-equivariant rank $r$ reflexive sheaves on $X$ and the category of collections of flags
\begin{align*}
{\bf{R}} &:= \{R_\rho(k)\}_{k \in \Z, \rho \in F(X)} \\
\cdots \subset R_\rho(k-1) &\subset R_\rho(k) \subset R_\rho(k+1) \subset \cdots
\end{align*}
satisfying
\begin{itemize}
\item $R_\rho(k) = 0$ for all $k \ll 0$,
\item $R_\rho(k) = \C^r$ for all $k \gg 0$.
\end{itemize}
The morphisms of both categories are the obvious ones. 

We now describe how a family of flags ${\bf{R}}$ gives rise to a $T$-equivariant reflexive sheaf $\ccR$ on $X$. Let $\alpha \in V(X)$ and let $U_\alpha \cong \C^3$ be the corresponding $T$-invariant affine open subset. Write the coordinates on $U_\alpha$ by $(x_1, x_2, x_3)$ as in Section \ref{toric3folds}. Suppose the toric divisor $x_i=0$ corresponds to ray $\rho_i \in F(X)$. Define vector spaces
\[
R_\alpha(k_1,k_2,k_3) := R_{\rho_1}(k_1) \cap R_{\rho_2}(k_2) \cap R_{\rho_3}(k_3).
\]
Then 
\[
\{R_\alpha(k_1,k_2,k_3)\}_{(k_1,k_2,k_3) \in \Z^3, \alpha \in V(X)}
\]
is a collection of rank $r$ $\sigma$-families satisfying the gluing conditions \eqref{glue}. The corresponding sheaf $\ccR$ on $X$ is reflexive \cite{Kly2, Per2}. Roughly speaking the reason is that a reflexive sheaf is determined by its restriction to the complement of any closed subset of codimension $\geq 2$ \cite[Prop.~1.6]{Har2} and the flags $R_{\rho_1}(k), R_{\rho_2}(k), R_{\rho_3}(k)$ precisely correspond to the restriction of the reflexive sheaf to $\C \times \C^* \times \C^*, \C^* \times \C \times \C^*, \C^* \times \C^* \times \C$. 

In the case $r=2$, a family of filtrations ${\bf{R}}$ is entirely determined by the integers where the dimensions jump together with a choice of 1-dimensional subspaces for the jumps. More precisely, for each $\rho \in F(X)$, there are unique $u_\rho \in \Z$, $v_\rho \in \Z_{\geq 0}$ and $p_\rho \in \PP^1$ such that
\[
R_\rho(k) = \left\{\begin{array}{cc} 0 & \mathrm{for \ all \ } k < u_\rho \\ p_\rho & \mathrm{for \ all \ } u_\rho \leq k < u_\rho+v_\rho \\ \C^2 & \mathrm{for \ all \ } u_\rho+v_\rho \leq k \end{array}\right.,
\] 
where $p_\rho$ does not occur when $v_\rho = 0$.

\begin{definition} \label{toricdata} 
Let $\ccR$ be a $T$-equivariant rank 2 reflexive sheaf on $X$ corresponding to a collection of filtrations $\{R_\rho(k)\}$. We refer to the data 
\[
\{(u_\rho,v_\rho,p_\rho)\}_{\rho \in F(X)}
\] 
described above as the \emph{toric data} of $\ccR$ and abbreviate it by $({\bf{u}},{\bf{v}},{\bf{p}})$. Given toric data $({\bf{u}},{\bf{v}},{\bf{p}})$, we define
\begin{align*}
\delta_{\rho,\rho'} &:=\dim(p_{\rho} \cap p_{\rho'}) \in \{0,1\}, \\
\delta_{\rho,\rho',\rho''} &:= \dim(p_{\rho} \cap p_{\rho'} \cap p_{\rho''}) \in \{0,1\}.
\end{align*}
We abbreviate the collection of these number by $\bsdelta$. Clearly the the numbers $({\bf{u}},{\bf{v}}, \bsdelta)$ determine the characteristic function $\bschi_{\ccR}$ and vice versa. We treat both notions on equal footing.
\end{definition}

\begin{example} \label{P3reflch}
The toric 3-fold $X = \PP^3$ is described by the lattice $N = \Z^3$ together with the fan consisting of the 3-dimensional cones 
\begin{align*}
\sigma_1 &=\langle e_1, e_2, e_3 \rangle_{\Z_{\geq 0}}, \\
\sigma_2 &=\langle e_2, e_3, -e_1-e_2-e_3 \rangle_{\Z_{\geq 0}}, \\ 
\sigma_3 &=\langle e_1, e_3, -e_1-e_2-e_3 \rangle_{\Z_{\geq 0}}, \\ 
\sigma_4 &=\langle e_1, e_2, -e_1-e_2-e_3 \rangle_{\Z_{\geq 0}}.
\end{align*} 
Here $(e_1,e_2,e_3)$ is the standard basis of $\Z^3$. We denote the rays generated by $e_1$, $e_2$, $e_3$, $-e_1-e_2-e_3$ by $\rho_1$, $\rho_2$, $\rho_3$, $\rho_4$. The corresponding polyhedron is a tetrahedron and $F(X) = \{\rho_1,\rho_2,\rho_3,\rho_4\}$. Let $h$ denote the hyperplane class on $X$ and let $\ccR$ be a $T$-equivariant rank 2 reflexive sheaf on $X$ described by toric data $({\bf{u}},{\bf{v}},{\bf{p}})$. We define $u_i := u_{\rho_i}$, $v_i := v_{\rho_i}$, $\delta_{ij} := \delta_{\rho_i, \rho_j}$, and $\delta_{ijk} := \delta_{\rho_i, \rho_j,\rho_k}$. In \cite{GK1} the Chern classes of $\ccR$ are computed using $T$-equivariant d\'evissage as mentioned in Section \ref{fixedloci1}
\begin{align*}
&c_1(\ccR) = - \big(2 \sum_i u_i + \sum_i v_i \big) h, \\
&c_2(\ccR) =  \frac{1}{4} c_1(\ccR)^2 + \Big( \frac{1}{2} \sum_{i<j} (1-2\delta_{ij}) v_i v_j -\frac{1}{4} \sum_i v_{i}^{2} \Big) h^2, \\
&c_3(\ccR) = \sum_{i<j<k} v_i v_j v_k (1-\delta_{ij}-\delta_{ik}-\delta_{jk} + 2\delta_{ijk}) h^3.
\end{align*}
This expresses $\ch(\ccR)$ in terms of the characteristic function ${\bschi_\ccR}$. By further removing Chern characters of $\O_{C_{\alpha\beta}} \otimes L$, $\O_{X_{\alpha}} \otimes L$ with $L \in \Pic^T(\PP^3)$, we obtain explicit expressions for  $\ch(\F)$ in terms of ${\bschi_\F}$ for any given $T$-equivariant rank 2 torsion free sheaf $\F$ on $\PP^3$. 
\end{example}

Back to any smooth projective toric 3-fold $X$. The following proposition is very useful.
\begin{proposition} \label{rank2singular}
Let $\ccR$ be a $T$-equivariant rank 2 reflexive sheaf on $X$ with toric data $({\bf{u}},{\bf{v}},{\bf{p}})$. Let $\alpha \in V(X)$ and denote by $\rho_{1,\alpha}, \rho_{2,\alpha}, \rho_{3,\alpha} \in F(X)$ the faces sharing vertex $\alpha$. Then $\ccR|_{U_\alpha}$ is singular, i.e.~not locally free, precisely if $v_{\rho_1,\alpha}$, $v_{\rho_2,\alpha}$, $v_{\rho_3,\alpha}$ are all positive and $p_{1,\alpha}$, $p_{2,\alpha}$, $p_{3,\alpha}$ are mutually distinct. If $\ccR|_{U_\alpha}$ is singular, then the length of its singularity equals the product $v_{\rho_1,\alpha}v_{\rho_2,\alpha}v_{\rho_3,\alpha}$. Moreover, the degree of $c_3(\ccR)$ equals the sum of the lengths of the singularities of $\{\ccR|_{U_\alpha}\}_{\alpha \in V(X)}$. 
\end{proposition}
\begin{proof}
For $X = \PP^3$ this is immediate from \cite[Prop.~2.6]{Har2} and the formula for $c_3(\ccR)$ of Example \ref{P3reflch}. For general $X$, the proof is easily adapted as is discussed in detail in the proof of \cite[Prop.~3.6]{GK1}.
\end{proof}

Let $\ccR$ be a $T$-equivariant rank 2 reflexive sheaf on $X$ described by toric data $({\bf{u}},{\bf{v}},{\bf{p}})$. Like in the case of Chern classes, $\mu$-stability of $\ccR$ can be described in terms of the toric data only. More precisely $\ccR$ is $\mu$-stable if and only if for any 1-dimensional subspace $q \subset \C^2$ 
\begin{align*}
\sum_{\rho \in F(X)} \dim(p_{\rho} \cap q) (D_{\rho} H^2) v_{\rho} < \frac{1}{2} \sum_{\rho \in F(X)} (D_{\rho} H^2) v_{\rho}, 
\end{align*}
where we recall that $H$ denotes the polarization and $D_{\rho}$ the toric divisor corresponding to $\rho \in F(X)$ (Section \ref{toric3folds}). A similar description exists for any rank on any polarized smooth projective toric $n$-fold \cite{Koo1}.

Let $\F$ be a $T$-equivariant torsion free sheaf on $X$. Then $\F$ is $\mu$-stable if and only if its reflexive hull $\ccR := \F^{**}$ is $\mu$-stable. So in order to understand $\mu$-stability of such sheaves, we only need to write the family of filtrations ${\bf{R}}$ corresponding to $\ccR$ in terms of the collection of $\sigma$-families ${\bf{F}}$ corresponding to $\F$. This can be done as follows: for any $\alpha \in V(X)$ denote by $\rho_{1}$, $\rho_{2}$, $\rho_{3}$ the faces sharing vertex $\alpha$, then 
\[
R_{\rho_{1}}(k) = F_\alpha(k,\infty,\infty), \ R_{\rho_{2}}(k) = F_\alpha(\infty,k,\infty), \ R_{\rho_{3}}(k) = F_\alpha(\infty,\infty,k),
\]
for all $k \in \Z$ (\cite{Kly2}, see also \cite{GK1}).

\subsection{General structure of fixed loci} \label{fixedlocigeneral}

Let ${\bschi} \in \X$ be the characteristic function of a $T$-equivariant rank 2 torsion free sheaf on $X$. There exists a natural moduli functor of families of $\mu$-stable $T$-equivariant torsion free sheaves on $X$ with characteristic function ${\bschi}$ \cite{Koo1}.\footnote{The paper \cite{Koo1} rather deals with Gieseker stability. The case of $\mu$-stability is easier.} Using the previous section, a fairly straight-forward GIT construction yields a quasi-projective $\C$-scheme corepresenting this functor. In this section we describe the outcome. For details and proofs (in a much more general setting) see \cite{Koo1}.

For a fixed characteristic function ${\bschi} = \{\chi_\alpha\}_{\alpha \in V(X)} \in \X$, the following definitions identify connected components in the domain of each $\chi_\alpha$ where $\chi_\alpha$ takes value 1. To each such connected component, we will then associate a moduli factor $\PP^1$ (up to gluing).
\begin{definition} \label{comp1s}
Denote the Euclidean norm on $\mathbb{R}^3$ by $|\!| \cdot |\!|$. Let ${\bschi} = \{\chi_\alpha\}_{\alpha \in V(X)} \in \X$ and consider $\chi_\alpha$. Let $S \subset X(T) = \Z^3$, then $S$ is called a \emph{connected component of 1's} associated to $\chi_\alpha$ when 
\begin{align*}
&\chi_\alpha(k_1,k_2,k_3) = 1, \ \mathrm{for  \ all \ } (k_1,k_2,k_3) \in S \\
&\mathrm{for \ all \ } (k_1,k_2,k_3) \in S \mathrm{ \ and \ } (l_1,l_2,l_3) \in \Z^3: \\
&\quad \mathrm{if \ } \chi_\alpha(l_1,l_2,l_3)=1 \mathrm{ \ and \ }  |\!|(k_1,k_2,k_3) - (l_1,l_2,l_3)|\!| = 1 \mathrm{ \ then \ } (l_1,l_2,l_3) \in S, 
\end{align*}
and $S$ does not contain a non-empty proper subset with these properties.
\end{definition}

\begin{definition} \label{vefcomps}
Let ${\bschi} = \{\chi_\alpha\}_{\alpha \in V(X)} \in \X$ and fix $\alpha \in V(X)$. Consider a connected component of 1's,  $\kappa$, associated to $\chi_\alpha$. 
\begin{itemize}
\item If $\kappa$ is bounded in $\Z^3$, then we refer to it as a \emph{vertex component} associated to $\chi_\alpha$. 
\item The connected component $\kappa$ is called a \emph{local face component} associated to $\chi_\alpha$ if it satisfies the following property. There exists a face $\rho \in F(X)$ containing the vertex $\alpha$ and with $v_\rho > 0$ such that if we choose coordinates for which the toric divisor $D_\rho$ corresponds to $\{x_1 = 0\}$ and write $u_1 := u_\rho$, $v_1 := v_\rho$, then we have $(k_1,k_2,k_3) \in \kappa$ for all integers $u_1 \leq k_1 < u_1 + v_1$ and $k_2,k_3 \gg 0$.
\item In all other cases $\kappa$ is called a \emph{local edge components} associated to $\chi_\alpha$.
\end{itemize}
\end{definition}

\begin{definition} \label{vefcomps2}
Let ${\bschi} = \{\chi_\alpha\}_{\alpha \in V(X)} \in \X$. The toric data $({\bf{u}},{\bf{v}},{\bf{p}})$ of any reflexive hull of a $T$-equivariant torsion free sheaf on $X$ with characteristic function ${\bschi}$ has the same values ${\bf{u}}$, ${\bf{v}}$. Consider all connected components of 1's of all $\{\chi_\alpha\}_{\alpha \in V(X)}$ at once. It is clear that connected components of 1's associated to $\chi_\alpha$, $\chi_\beta$ glue just like $\sigma$-families \eqref{glue}. Vertex components never glue, but local face or local edge components associated to one $\chi_\alpha$ may glue to local face or local edge components associated to another $\chi_\beta$. Note that local face components may glue to local edge components. After gluing, we refer to a global connected component of 1's containing a local face component as a \emph{face component} associated to $\bschi$. All other global connected components of 1's which are not vertex or face components are called \emph{edge components} associated to $\bschi$. We denote the number of face, edge, vertex components associated to $\bschi$ by $a,b,c$.
\end{definition}

Let the situation be as in the previous definition. To each connected component of 1's of $\bschi$ we associate a choice of $\C \subset \C^2$. This specifies a filtration $\{F_\alpha(k_1,k_2,k_3)\}_{\alpha \in V(X)}$, i.e.~a $T$-equivariant torsion free sheaf on $X$ with characteristic function $\bschi$.  Hence the closed points of 
\[
(\PP^1)^a \times (\PP^1)^b \times (\PP^1)^c
\]
are in bijective correspondence with all possible collections of $\sigma$-families of $T$-equivariant torsion free sheaves on $X$ with characteristic function ${\bschi}$. Moreover, two points correspond to isomorphic objects if and only if they lie in the same $\SL(2,\C)$-orbit. Here $\SL(2,\C)$ acts by matrix multiplication on each factor $\PP^1$. The $\SL(2,\C)$-equivariant line bundles on this space are in 1-1 correspondence with elements of \cite{Dol}
\[
\Z^a \times \Z^b \times \Z^c
\]
via the correspondence
\[
(\{a_i\}_{i=1}^{a}, \{b_j\}_{j=1}^{b}, \{c_k\}_{k=1}^{c}) \leftrightarrow \boxtimes_{i=1}^{a} \O_{\PP^1}(a_i) \boxtimes \boxtimes_{j=1}^{b} \O_{\PP^1}(b_j) \boxtimes \boxtimes_{k=1}^{c} \O_{\PP^1}(c_k).
\]
We want to pick an $\SL(2,\C)$-equivariant line bundle for which the properly GIT stable points exactly correspond to the collections of $\sigma$-families for which the associated sheaf is $\mu$-stable. Using the previous section, one can show that the following choice works \cite{Koo1}
\begin{align*}
a_i :=  \sum_\rho (H^2 D_{\rho}) v_{\rho},  \ b_{j} = 0, \ c_{k} = 0,
\end{align*}
where the sum is over all faces $\rho \in F(X)$ for which $v_{\rho} >0$ and the corresponding local face component is contained in the (global) face component indexed by $i$. The stability does not depend on the edge and vertex components but only on the face components. In particular, we get an open subset $U_{X}^{H}(\bschi) \subset (\PP^1)^a$ such that the closed points of 
\[
U_{X}^{H}(\bschi) \times (\PP^1)^b \times (\PP^1)^c
\]
precisely correspond to $\mu$-stable sheaves. Unfortunately the above $\SL(2,\C)$-equivariant line bundle is not ample. For ampleness, all $a_i, b_i,c_i$ must be positive. Therefore we fix $R \gg 0$ and take
\begin{align*}
a_i :=  R \sum_\rho (H^2 D_{\rho}) v_{\rho},  \ b_{i} = 1, \ c_{i} = 1,
\end{align*}
where the sum over $\rho$ is as above. Then the properly GIT stable locus is still
\[
U_{X}^{H}(\bschi) \times (\PP^1)^b \times (\PP^1)^c
\] 
and the GIT quotient
\[
\M_{X}^{H}(\bschi) := U_{X}^{H}(\bschi) \times (\PP^1)^b \times (\PP^1)^c \ \slash \ \SL(2,\C)
\]
is a quasi-projective variety whose closed points exactly correspond to the $T$-equivariant isomorphism classes of $T$-equivariant rank 2 $\mu$-stable torsion free sheaves on $X$ with characteristic function ${\bschi}$. The space $\M_{X}^{H}(\bschi)$ is projective when $\gcd(2,c_1 H^2) = 1$. Moreover, $\M_{X}^{H}(\bschi)$ is smooth, because $U_{X}^{H}(\bschi) \times (\PP^1)^b \times (\PP^1)^c$ is smooth and $\mathrm{PGL}(2,\C)$ acts without stabilizers. 
One can show that $\M_{X}^{H}(\bschi)$ indeed corepresents a natural moduli functor as follows \cite{Koo1}:
\begin{theorem} \label{equivmod}
Let $X$ be a smooth projective toric 3-fold with polarization $H$. Let $\bschi \in \X$ be the characteristic function of a $T$-equivariant rank 2 torsion free sheaf on $X$. Then 
\[
\M_{X}^{H}(\bschi) := U_{X}^{H}(\bschi) \times (\PP^1)^b \times (\PP^1)^c \ \slash \ \SL(2,\C)
\]
is a smooth quasi-projective variety corepresenting the moduli functor of families of $T$-equivariant $\mu$-stable torsion free sheaves on $X$ with characteristic function $\bschi$. 
\end{theorem}

\begin{remark} \label{a=3}
In the above theorem, the double dual map is given by 
$$
U_{X}^{H}(\bschi) \times (\PP^1)^b \times (\PP^1)^c \ \slash \ \SL(2,\C) \longrightarrow U_{X}^{H}(\bschi)  \ \slash \ \SL(2,\C),
$$
where we denote the latter quotient by $\N_{X}^{H}(\bschi)$. We deduce that if $\M_{X}^{H}(\bschi) \neq \varnothing$, then $a \geq 3$ (or else the connected component of reflexive hulls $\N_{X}^{H}(\bschi) = \varnothing$). We also deduce: all closed points of $\M_{X}^{H}(\bschi)$ have the same reflexive hull if and only if $a=3$. Suppose this is the case. Then $\M_{X}^{H}(\bschi) \cong (\PP^1)^b \times (\PP^1)^c$. Moreover, $\N_{X}^{H}(\bschi)$ is an isolated reduced point $\{[\ccR]\}$ of the $T$-fixed locus of the moduli space of $\mu$-stable reflexive sheaves on $X$. Therefore
$$
\Ext^1(\ccR,\ccR)^T = 0.
$$
We have a tautological line bundle $\O_{\PP^1}(1)$ on each factor $\PP^1$ of $\M_{X}^{H}(\bschi)$. Using the tautological inclusions 
\[
\O_{\PP^1}(-1) \subset \O_{\PP^1} \oplus \O_{\PP^1},
\]
it is easy to construct a universal family $\FF$ on $\M_{X}^{H}(\bschi)$ fitting in a short exact sequence
$$
0 \longrightarrow \FF \longrightarrow p_{X}^{*} \ccR \longrightarrow \Q \longrightarrow 0,
$$
where the cokernel $\Q$ is $\M_{X}^{H}(\bschi)$-flat. 
\end{remark}

Suppose now $X$ is a smooth projective toric 3-fold with polarization $H$ and let $\M_X := \M_{X}^{H}(2,c_1,c_2,c_3)$ denote the moduli space of rank $2$ $\mu$-stable torsion free sheaves on $X$ with indicated Chern classes. We consider its fixed locus $\M_X^T$.
Denote by $\X_{(2,c_1,c_2,c_3)}$ the collection of characteristic functions of $T$-equivariant rank 2 torsion free sheaves on $X$ with Chern classes $c_1, c_2, c_3$. Forgetting the $T$-equivariant structure defines a map
\begin{equation} \label{forget}
\bigsqcup_{{\bschi} \in \X_{(2,c_1,c_2,c_3)}} \M_{X}^{H}(\bschi) \longrightarrow \M_{X}^{T}.
\end{equation}
It is not hard to show this is a morphism \cite{Koo1}. Any element of $\M_{X}^{T}$ admits a $T$-equivariant structure, which is unique up to tensoring by a character of $X(T)$ \cite{Koo1}. The map \eqref{forget} is therefore surjective but not injective. It can be made injective as follows. Choose any vertex $\alpha_0 \in V(X)$ and denote by $\rho_1, \rho_2, \rho_3 \in F(X)$ the faces having $\alpha_0$ as a vertex. Define
\[
\X_{(2,c_1,c_2,c_3)}^{\slice} \subset \X_{(2,c_1,c_2,c_3)}
\]
as the subset of characteristic functions ${\bschi}$ for which the corresponding $({\bf{u}},{\bf{v}})$ satisfy
\begin{equation} \label{slice}
u_{\rho_1} = u_{\rho_2} = u_{\rho_3} = 0.
\end{equation}
Then 
\[
\bigsqcup_{{\bschi} \in \X_{(2,c_1,c_2,c_3)}^{\slice}} \M_{X}^{H}(\bschi) \longrightarrow \M_{X}^{T}
\]
is a bijective morphism. In fact, it is an isomorphism of schemes \cite[Cor.~4.10]{Koo1}.\footnote{This construction works for any rank $r$ and any smooth projective toric $n$-fold.} Combining with Theorem \ref{equivmod} gives:
\begin{theorem} \label{fixedlocithm}
Let $X$ be a smooth projective toric 3-fold with polarization $H$ and let $\M_X := \M_{X}^{H}(2,c_1,c_2,c_3)$. Then the forgetful map described above is an isomorphism of schemes
\[
\M_{X}^{T} \cong \bigsqcup_{{\bschi} \in \X_{(2,c_1,c_2,c_3)}^{\slice}} \M_{X}^{H}(\bschi).
\]
In particular, $\M_{X}^{T}$ is a smooth quasi-projective variety.
\end{theorem}

\subsection{Quot schemes of reflexive sheaves} \label{3Dpartitions}

As discussed in the introduction, double duals of members of a flat family of torsion free sheaves do not need to form a flat family. However, a result of Koll\'ar \cite{Kol} implies the double dual map is constructible 
\[
(\cdot)^{**} : \M_{X}^{H}(2,c_1,c_2,c_3) \longrightarrow \bigsqcup_{c_{2}', c_{3}'} \N_{X}^{H}(2,c_1,c_{2}',c_{3}').
\] 
At the level of closed points, the fibre over $[\ccR] \in  \N_{X}^{H}(2,c_1,c_{2}',c_{3}')$ is the Quot scheme
\[
\Quot_X(\ccR,c_{2}'',c_{3}'')
\]
of quotients of $\ccR$ of dimension $\leq 1$, such that
\begin{align}
\begin{split} \label{Chernrel}
&c_{2}' = c_2 + c_{2}'' \\
&c_{3}' = c_3 + c_{3}'' + c_1 c_{2}''.
\end{split}
\end{align}
In this section, we develop a way of indexing the connected components of $\Quot_X(\ccR,c_{2}'',c_{3}'')^T$ by a new type of combinatorial objects called double box configurations. Each connected component is isomorphic to a product of $\PP^1$'s as we will show first.

Suppose $\ccR$ is any $T$-equivariant rank 2 reflexive sheaf on $X$ described by toric data $({\bf{u}},{\bf{v}},{\bf{p}})$ (Definition \ref{toricdata}). Denote by 
$
\X_{c_{2}'',c_{3}''}(\ccR) \subset \X
$
the collection of characteristic functions of $T$-equivariant subsheaves of $\ccR$ such that the cokernel has Chern classes $c_{2}''$, $c_{3}''$. The closed points of $\Quot_X(\ccR,c_{2}'',c_{3}'')^T$ are $T$-invariant submodules of $$\{H^0(U_\alpha,\ccR)\}_{\alpha \in V(X)}$$ satisfying the gluing conditions (\ref{glue}). This gives a decomposition into connected components
\begin{equation} \label{quotcomp}
\Quot_X(\ccR,c_{2}'',c_{3}'')^T \cong \bigsqcup_{\bschi \in \X_{c_{2}'',c_{3}''}(\ccR)} \cC_{\bschi}.
\end{equation}
\begin{proposition} \label{prodP1s}
Each $\cC_{\bschi}$ of \eqref{quotcomp} is isomorphic to a product of $\PP^1$'s, where the number of $\PP^1$'s is the number of vertex and edge components associated to $\bschi$.
\end{proposition}
\begin{proof}
This can be proved using the methods of \cite{Koo1}. In fact, this case is much easier than Theorem \ref{fixedlocithm}. Step 1: introduce a moduli functor of $T$-equivariant families of quotients of $\ccR$ with Chern classes $c_{2}'',c_{3}''$ as in \cite[Sect.~3.1]{Koo1}. Step 2: describe these families via a family version of Klyachko's filtrations like in \cite[Sect.~3.2]{Koo1}. Step 3: construct the moduli spaces $\cC_{\bschi}$ as GIT quotients and show they corepresent the previous moduli functors \cite[Sect.~3.3]{Koo1}. Step 4: lift the action of $T$ on $X$ to $\Quot_X(\ccR,c_{2}'',c_{3}'')$ (done in the proof of \cite[Prop.~4.1]{Koo1}). Step 4: show the forgetful morphism is an isomorphism. This goes similar to \cite[Sect.~4.3]{Koo1} but much easier: this time elements of $\Quot_X(\ccR,c_{2}'',c_{3}'')^T$ admit a \emph{unique} $T$-equivariant structure. 
\end{proof}

\begin{remark} \label{nofacemod}
Note that we do not assign factors of $\PP^1$ to the face components in Proposition \ref{prodP1s}. This is because the reflexive hull $\ccR$ is kept fixed so does not have moduli.
\end{remark}

Fix $\alpha \in V(X)$ and $\ccR_\alpha := \ccR|_{U_\alpha}$. Our goal is to give a combinatorial description of characteristic functions $\chi$ of rank 2 $T$-equivariant subsheaves of $\ccR_\alpha$. We denote the collection of such characteristic functions by $\X(\ccR_\alpha)$. Let $\rho_1, \rho_2, \rho_3 \in F(X)$ be the faces with vertex $\alpha$. Suppose the labeling is chosen such that the toric divisor $D_{\rho_i}$ in chart $U_\alpha$ corresponds to the coordinate hyperplane $\{x_i=0\}$. Then $\chi$ gives rise to integers (Definition \ref{toricdata})
\begin{align*}
u_i &:= u_{\rho_i} \in \Z, \\ 
v_i &:= v_{\rho_i} \geq 0.
\end{align*}
Moreover, $\ccR_\alpha$ is described by toric data (Definition \ref{toricdata})
\begin{align*}
u_i &\in \Z, \\ 
v_i &\geq 0, \\ 
p_i &:= p_{\rho_i} \in \PP^1.
\end{align*}
This data is represented by Figure 1. The sheaf $\ccR_\alpha$ is singular if and only if $v_1, v_2, v_3 >0$ and $p_1, p_2, p_3$ are mutually distinct (Proposition \ref{rank2singular}). We discuss the case $\ccR_\alpha$ is singular. The (easier) case $\ccR_\alpha$ is locally free is treated in Remark \ref{deco}. In what follows, we let $$\ccR_{\alpha\beta} := \ccR|_{U_{\alpha\beta}},$$ where $U_{\alpha\beta} = U_\alpha \cap U_\beta$ for all $\alpha\beta \in E(X)$. Recall that (Section \ref{equivsh})
$$
R_{\alpha\beta_1}(k_2,k_3) = R_{\alpha}(\infty,k_2,k_3)
$$
for all $k_2, k_3 \in \Z$ (and similarly for $\ccR_{\alpha\beta_2}$, $\ccR_{\alpha\beta_3}$). Note that
\begin{equation*}
R_{\alpha\beta_1}(k_2,k_3) = \left\{\begin{array}{cc} \C^2 & \mathrm{for \ all \ } k_2 \geq u_2+v_2 \ \mathrm{and \ } k_3 \geq u_3 + v_3 \\ p_2 & \mathrm{for \ all \ } u_2 \leq k_2 < u_2+v_2 \ \mathrm{and \ } k_3 \geq u_3+v_3 \\ p_3 & \mathrm{for \ all \ } k_2 \geq u_2+v_2 \ \mathrm{and \ } u_3 \leq k_3 < u_3+v_3 \\  0 & \mathrm{otherwise}. \end{array}\right.
\end{equation*} 
In this way $\ccR_{\alpha\beta_1}$ determines the integers $u_2,u_3,v_2,v_3$ (and similarly for  $\ccR_{\alpha\beta_2}$, $\ccR_{\alpha\beta_3}$). 
 
\begin{figure}
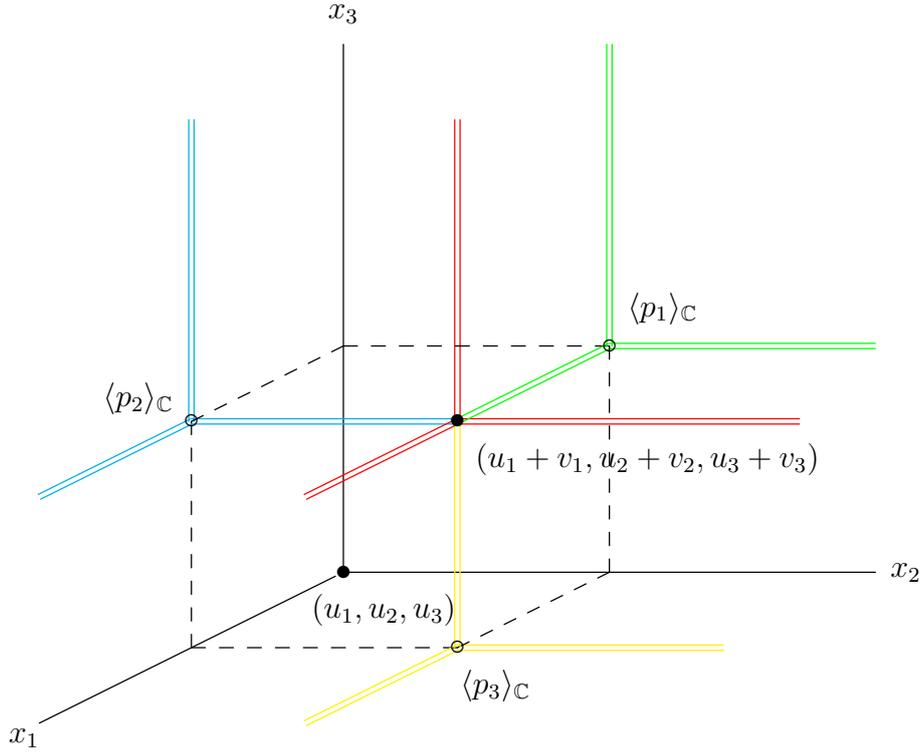
 
\begin{displaymath}
\xy
(0,0)*{\bullet} ; (-40,-20)*{} **\dir{-} ; (-42,-22)*{x_1} ; (0,0)*{} ; (70,0)*{} **\dir{-} ; (74,0)*{x_2} ; (0,0)*{} ; (0,70)*{} **\dir{-} ; (0,74)*{x_3} ; (5.3,-5)*{(u_1,u_2,u_3)}
 ; (-20,-10)*{} ; (15,-10) **\dir{--} ; (15,-10) ; (50,-10)*{} **[yellow]\dir{=} ; (35,0)*{} ; (15,-10)*{} **\dir{--} ; (15,-10)*{} ; (-5,-20)*{} **[yellow]\dir{=} ; (-20,-10)*{} ; (-20,20)*{} **\dir{--} ; (-20,20)*{} ; (-20,60)*{} **[cyan]\dir{=} ; (15,-10)*{} ; (15,20)*{} **[yellow]\dir{=} ; (15,20)*{} ; (15,60)*{} **[red]\dir{=} ; (-20,20)*{} ; (15,20)*{} **[cyan]\dir{=} ; (15,20)*{} ; (60,20)*{} **[red]\dir{=} ; (0,30)*{} ; (-20,20)*{} **\dir{--} ; (-20,20)*{} ; (-40,10)*{} **[cyan]\dir{=} ; (0,30)*{} ; (35,30)*{} **\dir{--} ; (35,30)*{} ; (70,30)*{} **[green]\dir{=} ; (35,30)*{} ; (15,20)*{} **[green]\dir{=} ; (15,20)*{} ;  (-5,10)*{} **[red]\dir{=} ; (35,0)*{} ; (35,30)*{} **\dir{--} ; (35,30)*{} ; (35,70)*{} **[green]\dir{=} ; (15,20)*{\bullet} ; (40,15)*{(u_1+v_1,u_2+v_2,u_3+v_3)} 
; (42,35)*{\langle p_1 \rangle_{\C}} ; (-27,23)*{\langle p_2 \rangle_{\C}} ; (20,-15)*{\langle p_3 \rangle_{\C}} ; (35,30)*{\circ} ; (-20,20)*{\circ} ; (15,-10)*{\circ}
\endxy 
\end{displaymath}
\caption{Toric data $u_i, v_i, p_i$ lying in the character lattice $X(T) = \Z^3$. There are three local face components with weight spaces of dimension 1: the green (generated by $p_1$), blue  (generated by $p_2$), and yellow  (generated by $p_3$) regions. The red region has weight spaces of dimension 2. All other points of the character lattice have weight spaces of dimension 0.}
\end{figure}

We start with some definitions. Given a point ${\bf{k}} = (k_1,k_2)$ in $\R^2$, we write 
\[
C({\bf{k}}) := \{(k_1+l_1,k_2+l_2) \ | \ l_i \in \R_{\geq 0} \}.
\]
\begin{definition}
A \emph{2D partition} based at $(m_1,m_2) \in \Z^2$ is a subset $\lambda \subset \Z^2$ satisfying the following condition. If $(k_1,k_2) \in \lambda$ then $k_1 \geq m_1$, $k_2 \geq m_2$, and $(l_1,k_2) \in \lambda$, $(k_1,l_2) \in \lambda$ for all $m_1 \leq l_1 \leq k_1$, $m_2 \leq l_2 \leq k_2$. We call $\lambda$ \emph{finite} if $|\lambda| < \infty$.  
\end{definition}

\begin{definition} \label{doublesquare}
Let $(i;i',i'') = (1;2,3)$, $(2;1,3)$, or $(3;1,2)$. Consider a triple of \emph{finite} 2D partitions ${\bslambda} = (\lambda_1, \lambda_2,\lambda_3)$ based respectively at $(u_{i'},u_{i''}+v_{i''})$, $(u_{i'}+v_{i'},u_{i''})$, $(u_{i'}+v_{i'},u_{i''}+v_{i''})$ and subject to the following condition. Let $C_1 := C(u_{i'},u_{i''}+v_{i''})$, $C_2 := C(u_{i'}+v_{i'},u_{i''})$, and $D:=C(u_{i'}+v_{i'},u_{i''}+v_{i''})$. Define 
\begin{align*}
\bslambda_\mathrm{in} &:= \lambda_1 \cap \lambda_2 \cap \lambda_3, \\
\bslambda_\mathrm{out} &:= (\lambda_1 \cup \lambda_2 \cup \lambda_3) \cap D.
\end{align*} 
\emph{Condition}: each element of $\bslambda_\mathrm{out} \setminus \bslambda_\mathrm{in}$ belongs to \emph{exactly two} of the partitions $\lambda_1, \lambda_2, \lambda_3$. We define an equivalence relation $\sim$ on the collection of such triples ${\bslambda} = (\lambda_1, \lambda_2,\lambda_3)$ as follows. Write ${\bslambda} \sim {\bsmu}$ whenever:
\begin{enumerate}
\item $\lambda_j \cap (C_j \setminus D) = \mu_j \cap (C_j \setminus D)$ for all $j=1,2$,
\item $\bslambda_\mathrm{in} = \bsmu_\mathrm{in}$,
\item $\bslambda_\mathrm{out} = \bsmu_\mathrm{out}$.
\end{enumerate}
We refer to such equivalence classes as \emph{double square configurations} in $\ccR_{\alpha\beta_i}$ and denote the collection of such equivalence classes by $\Lambda(\ccR_{\alpha\beta_i})$. For any $\bslambda \in \Lambda(\ccR_{\alpha\beta_i})$, we define the \emph{size} $|\bslambda|$ as
\[
|\bslambda| := \Big( \sum_{i=1}^{3} |\lambda_i| \Big) - |\bslambda_{\mathrm{out}}|.
\]
Note that $|\bslambda|$ is independent of the choice of representative of the equivalence class.
\end{definition}

\begin{example} 
Let $(u_1,u_2)=(0,0)$ and $(v_1,v_2)=(3,1)$. Consider the following triples of 2D partitions $$(\includegraphics[width=.4in]{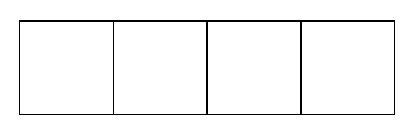}, \includegraphics[width=.21in]{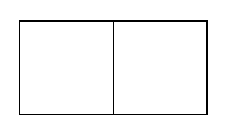}, \includegraphics[width=.12in]{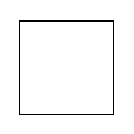}), (\includegraphics[width=.4in]{1+1+1+1}, \includegraphics[width=.21in]{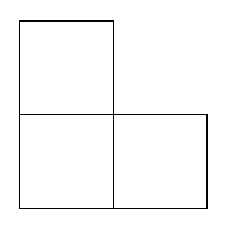}, \varnothing),   (\includegraphics[width=.3in]{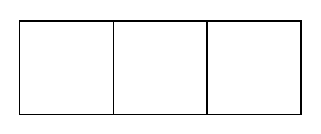}, \includegraphics[width=.21in]{2+1}, \includegraphics[width=.12in]{s}). $$ 
For each of these triples $\bslambda$, we have that $\lambda_1 \cap (C_1 \setminus D) = \includegraphics[width=.3in]{1+1+1}$, $\lambda_2 \cap (C_2 \setminus D) = \includegraphics[width=.21in]{1+1}$, $\bslambda_\mathrm{in} = \varnothing$, and $\bslambda_\mathrm{out} = \includegraphics[width=.12in]{s}$. Therefore all three triples are equivalent and determine the same double square configuration.
\end{example}

Given a point ${\bf{k}} = (k_1,k_2,k_3)$ in $\R^3$, we write 
\[
C({\bf{k}}) := \{(k_1+l_1,k_2+l_2,k_3+l_3) \ | \ l_i \in \R_{\geq 0} \}.
\]
Referring to Figure 1, we are interested in the following cones
\begin{align}  
\begin{split} \label{CCCD}
C_1 &:= C(u_1,u_2+v_2,u_3+v_3), \\
C_2 &:= C(u_1+v_1,u_2,u_3+v_3), \\
C_3 &:= C(u_1+v_1,u_2+v_2,u_3), \\
D &:= C(u_1+v_1,u_2+v_2,u_3+v_3).
\end{split}
\end{align}
For all $(i;i',i'') = (1;2,3), (2;1,3), (3;1,2)$, let $P_i : \R^3 \rightarrow \R^2$, $P_i(k_1,k_2,k_3) = (k_{i'},k_{i''})$ denote projection. 

\begin{definition}
A \emph{3D partition} based at $(m_1,m_2,m_3) \in \Z^3$ is a subset $\pi \subset \Z^3$ satisfying the following condition. If $(k_1,k_2,k_3) \in \pi$, then $k_1 \geq m_1$, $k_2 \geq m_2$, $k_3 \geq m_3$, and $(l_1,k_2,k_3) \in \pi$, $(k_1,l_2,k_3) \in \pi$, $(k_1,k_2,l_3) \in \pi$ for all $m_1 \leq l_1 \leq k_1$, $m_2 \leq l_2 \leq k_2$, $m_3 \leq l_3 \leq k_3$. We call $\pi$ \emph{finite} if $|\pi| < \infty$. 
\end{definition}

\begin{definition} \label{doubleboxconf}
Fix outgoing double square configurations ${\bslambda}_i \in \Lambda(\ccR_{\alpha\beta_i})$ for all $i=1,2,3$. Consider a triple of (not necessarily finite) 3D partitions ${\bspi} = (\pi_1, \pi_2, \pi_3)$ based at $(u_1,u_2+v_2,u_3+v_3)$, $(u_1+v_1,u_2,u_3+v_3)$, $(u_1+v_1,u_2+v_2,u_3)$ subject to the two conditions described below. Define 
\begin{align*}
\bspi_\mathrm{in} &:= \pi_1 \cap \pi_2 \cap \pi_3, \\
\bspi_\mathrm{out} &:= (\pi_1 \cup \pi_2 \cup \pi_3) \cap D.
\end{align*}
Take $N \gg 0$ sufficiently large such that the following projections do not depend on $N$
\begin{align*}
\lambda_{1}(\pi_1) &:= P_1(\pi_1 \setminus [u_1,u_1+N] \times \Z^2 ), \\
\lambda_{1}(\pi_2) &:= P_1(\pi_2 \setminus [u_1,u_1+N] \times \Z^2), \\
\lambda_{1}(\pi_3) &:= P_1(\pi_3 \setminus [u_1,u_1+N] \times \Z^2), \\
\lambda_{2}(\pi_1) &:= P_2(\pi_1 \setminus \Z \times [u_2, u_2+N] \times \Z), \\
\lambda_{2}(\pi_2) &:= P_2(\pi_2 \setminus \Z \times [u_2, u_2+N] \times \Z), \\
\lambda_{2}(\pi_3) &:= P_2(\pi_3 \setminus \Z \times [u_2, u_2+N] \times \Z), \\
\lambda_{3}(\pi_1) &:= P_3(\pi_1 \setminus \Z^2 \times [u_3,u_3+N]), \\
\lambda_{3}(\pi_2) &:= P_3(\pi_2 \setminus \Z^2 \times [u_3,u_3+N]), \\
\lambda_{3}(\pi_3) &:= P_3(\pi_3 \setminus \Z^2 \times [u_3,u_3+N]). 
\end{align*}


\noindent \emph{Condition 1}: Each element of $\bspi_\mathrm{out} \setminus \bspi_\mathrm{in}$ belongs to \emph{exactly two} partitions $\pi_i$. \\
\noindent  \emph{Condition 2}:
\begin{align*}
(\lambda_1(\pi_2),\lambda_1(\pi_3),\lambda_1(\pi_1)) &\sim {\bslambda}_1, \\
(\lambda_2(\pi_1),\lambda_2(\pi_3),\lambda_2(\pi_2)) &\sim {\bslambda}_2, \\
(\lambda_3(\pi_1),\lambda_3(\pi_2),\lambda_3(\pi_3)) &\sim {\bslambda}_3.
\end{align*}
We define an equivalence relation $\sim$ on the collection of such triples as follows. Write ${\bspi} \sim {\bsvpi}$ whenever:
\begin{enumerate}
\item $\pi_i \cap (C_i \setminus D) = \varpi_i \cap (C_i \setminus D)$ for all $i=1,2,3$,
\item $\bspi_\mathrm{in} = \bsvpi_\mathrm{in}$,
\item $\bspi_\mathrm{out} = \bsvpi_\mathrm{out}$.
\end{enumerate}
We refer to such equivalence classes as \emph{double box configurations} in $\ccR_\alpha$ and denote the collection of such equivalence classes by $\Pi(\ccR_\alpha, {\bslambda}_1,  {\bslambda}_2,  {\bslambda}_3)$. For any $\bspi \in \Pi(\ccR_\alpha,\bslambda_1, \bslambda_2, \bslambda_3)$, we define the \emph{size} of $\bspi$ as
\[
|\bspi| := \Big( \sum_{i=1}^{3} |\pi_i| \Big) - |\bspi_{\mathrm{out}}|,
\]
where $|\cdot|$ on the RHS denotes the renormalized volume.\footnote{The renormalized volume of a 3D partition $\pi$ is defined as $|\pi| = \sum_{\square \in \pi} \left( 1 - \# \mathrm{legs \ containing \ } \square \right)$. See \cite{MNOP1} for details.} Note that $|\bspi|$ is independent of the choice of representative of the equivalence class.
\end{definition}

The purpose of this definition is to give a new description of the characteristic function $\chi_\alpha$ of a rank 2 $T$-equivariant torsion free sheaf $\F_\alpha \subset \ccR_\alpha$ on $U_\alpha$. Since 
$$
\chi_{\F_\alpha} \leq \chi_{\F_{\alpha}^{**}} = \chi_{\ccR_\alpha},
$$
the difference $\chi_{\ccR_\alpha} - \chi_{\F_\alpha}$ indicates regions of dimension 0,1,2. Double box configurations are a combinatorial way of indexing all possible differences that can occur in this way. The precise statement is:
\begin{proposition} \label{lem1}
Let $\ccR_\alpha$ be a $T$-equivariant rank 2 reflexive sheaf on toric chart $U_\alpha \cong \C^3$. Let $\ccR_{\alpha\beta_1} := \ccR_\alpha|_{\C^* \times \C \times \C}$, $\ccR_{\alpha\beta_2} := \ccR_\alpha|_{\C \times \C^* \times \C}$, and $\ccR_{\alpha\beta_3} := \ccR_\alpha|_{\C \times \C \times \C^*}$. Denote by $\X(\ccR_\alpha)$ the collection of characteristic functions of $T$-equivariant rank 2 torsion free sheaves on $U_\alpha$ with $T$-equivariant reflexive hull $\ccR_\alpha$. Then there is a natural bijection
\[
\X(\ccR_\alpha) \cong \bigsqcup_{{\scriptsize{\begin{array}{c} {\bslambda}_1 \in \Lambda(\ccR_{\alpha\beta_1}) \\ {\bslambda}_2 \in \Lambda(\ccR_{\alpha\beta_2}) \\ {\bslambda}_3 \in \Lambda(\ccR_{\alpha\beta_3}) \end{array}}}} \Pi(\ccR_\alpha,  {\bslambda}_1,  {\bslambda}_2,  {\bslambda}_3).
\] 
\end{proposition}
\begin{proof}
Let $\chi \in \X(\ccR_\alpha)$ be the characteristic function of a $T$-equivariant rank 2 torsion free sheaf on $U_\alpha$ with reflexive hull $\ccR_\alpha$ described by toric data $({\bf{u}},{\bf{v}},{\bf{p}})$. Consider the regions $C_1,C_2,C_3,D$ of \eqref{CCCD}. The region $\chi|_{C_i \setminus D}$ has zeros and ones. These zeros form a 3D partition $\pi_{i}^{(0)}$. The region $\chi|_{D}$ has zeros, ones, and twos. The zeros give rise to a 3D partition $\bspi_{\mathrm{in}}$. The partition $\bspi_{\mathrm{in}}$ has the property that
$$
\pi_{1}^{(1)} := \pi_{1}^{(0)} \cup \bspi_{\mathrm{in}}, \ \pi_{2}^{(1)} := \pi_{2}^{(0)} \cup \bspi_{\mathrm{in}}, \ \pi_{3}^{(1)} := \pi_{3}^{(0)} \cup \bspi_{\mathrm{in}}
$$
are 3D partitions.
Finally, consider the connected components of 1's of $\chi|_{D}$. Such a connected component $\kappa$ has the property that $\bspi_{\mathrm{in}} \cup \kappa$ is a 3D partition. Moreover, at least two of 
$$
\pi_{1}^{(1)} \cup \kappa, \ \pi_{2}^{(1)} \cup \kappa, \ \pi_{3}^{(1)} \cup \kappa
$$
are 3D partitions. If only two of these form a 3D partition, say (without loss of generality) $\pi_{1}^{(1)} \cup \kappa, \pi_{2}^{(1)} \cup \kappa$, then define
$$
\pi_{1}^{(2)} := \pi_{1}^{(1)} \cup \kappa, \ \pi_{2}^{(2)} := \pi_{2}^{(1)} \cup \kappa, \ \pi_{3}^{(2)} := \pi_{3}^{(1)}.
$$
Otherwise choose any two, say $\pi_{1}^{(1)}, \ \pi_{2}^{(1)}$, and define 
$$
\pi_{1}^{(2)} := \pi_{1}^{(1)} \cup \kappa, \ \pi_{2}^{(2)} := \pi_{2}^{(1)} \cup \kappa, \ \pi_{3}^{(2)} := \pi_{3}^{(1)}.
$$
Proceeding in this way for all connected components $\kappa$, we end up with 3D partitions
$$
\pi_1 := \pi_{1}^{(0)} \cup \bspi_{\mathrm{in}} \cup \kappa_{i_1} \cup \cdots , \pi_2 := \pi_{2}^{(0)} \cup \bspi_{\mathrm{in}} \cup \kappa_{j_1} \cup \cdots, \pi_3 := \pi_{3}^{(0)} \cup \bspi_{\mathrm{in}} \cup \kappa_{k_1} \cup \cdots.
$$
Then $\bspi = (\pi_1, \pi_2, \pi_3)$ forms a double box configuration and the different choices made in the construction lead to equivalent double box configurations. We described a map 
\[
\X(\ccR_\alpha) \longrightarrow \bigsqcup_{{\scriptsize{\begin{array}{c} {\bslambda}_1 \in \Lambda(\ccR_{\alpha\beta_1}) \\ {\bslambda}_2 \in \Lambda(\ccR_{\alpha\beta_2}) \\ {\bslambda}_3 \in \Lambda(\ccR_{\alpha\beta_3}) \end{array}}}} \Pi(\ccR_\alpha,  {\bslambda}_1,  {\bslambda}_2,  {\bslambda}_3).
\] 
It is easy to check this is a bijection. 
\end{proof}

Next we assign moduli to each double box configuration. In the language of characteristic functions: we want to assign a moduli factor $\PP^1$ to each connected component of 1's which is \emph{not} a local face component. The reason we do not assign moduli to local face components is because the reflexive hull $\ccR_\alpha$ is fixed (Remark \ref{nofacemod}). A precise way of phrasing this in terms of double box configurations is as follows:
\begin{definition} \label{modfactor}
To any $\bspi \in \Pi(\ccR_\alpha,\bslambda_1, \bslambda_2, \bslambda_3)$ we associate a space $\cC_{\bspi}$ as follows. Analogous to Definition \ref{comp1s}, the partition $\bspi_\mathrm{out}$ can be written as a union of connected components. Define the regions
\begin{align*}
S_1 &:= \{u_1+v_1-1\} \times [u_2+v_2,\infty] \times [u_3+v_3,\infty], \\
S_2 &:= [u_1+v_1,\infty] \times \{u_2+v_2-1\} \times [u_3+v_3,\infty], \\
S_3 &:= [u_1+v_1,\infty] \times [u_2+v_2,\infty] \times \{u_3+v_3-1\}. 
\end{align*}
Let $P_i : \R^3 \rightarrow \R^2$, $i=1,2,3$ denote projections as before. Let $\kappa$ be a connected component of boxes contained in precisely two 3D partitions among $\pi_1,\pi_2,\pi_3$. If
\begin{align} \label{proj}
P_i(\kappa) \subset P_i ( \pi_i \cap S_i ),
\end{align}
for all $i=1,2,3$, then we associate a copy of $\PP^1$ to it. (Note that this is independent of choice of representative of the equivalence class.) Otherwise the connected component of boxes has no moduli assigned to it. We define the \emph{moduli factor} associated to $\bspi$ as the $n$-fold product of $\PP^1$, where $n$ is the number of connected components of boxes $\kappa$ satisfying (\ref{proj})
\[
\cC_{\bspi} := \underbrace{\PP^1 \times \cdots \times \PP^1}_{n}.
\]
\end{definition}

We illustrate these definitions in the following examples.
\begin{example}
Let ${\bspi} \in \Pi(\ccR_\alpha,(\varnothing,\varnothing,\varnothing),(\varnothing,\varnothing,\varnothing),(\varnothing,\varnothing,\varnothing))$ be given by Figure 2. The green, blue, yellow 3D partitions $\pi_1,\pi_2,\pi_3$ are representatives of $\bspi = (\pi_1,\pi_2,\pi_3)$. The red boxes correspond to $\bspi_{\mathrm{in}}$ and are contained in all three 3D partitions $\pi_1,\pi_2,\pi_3$. The white boxes correspond to $\bspi_{\mathrm{out}}$ and are contained in precisely two of $\pi_1, \pi_2, \pi_3$. Those with moduli have a label $m$. Note that only one of the three connected components of white boxes satisfies (\ref{proj}). Therefore $\cC_{\bspi} \cong \PP^1$. 
\end{example}

\begin{figure} 
\includegraphics[width=3.5in]{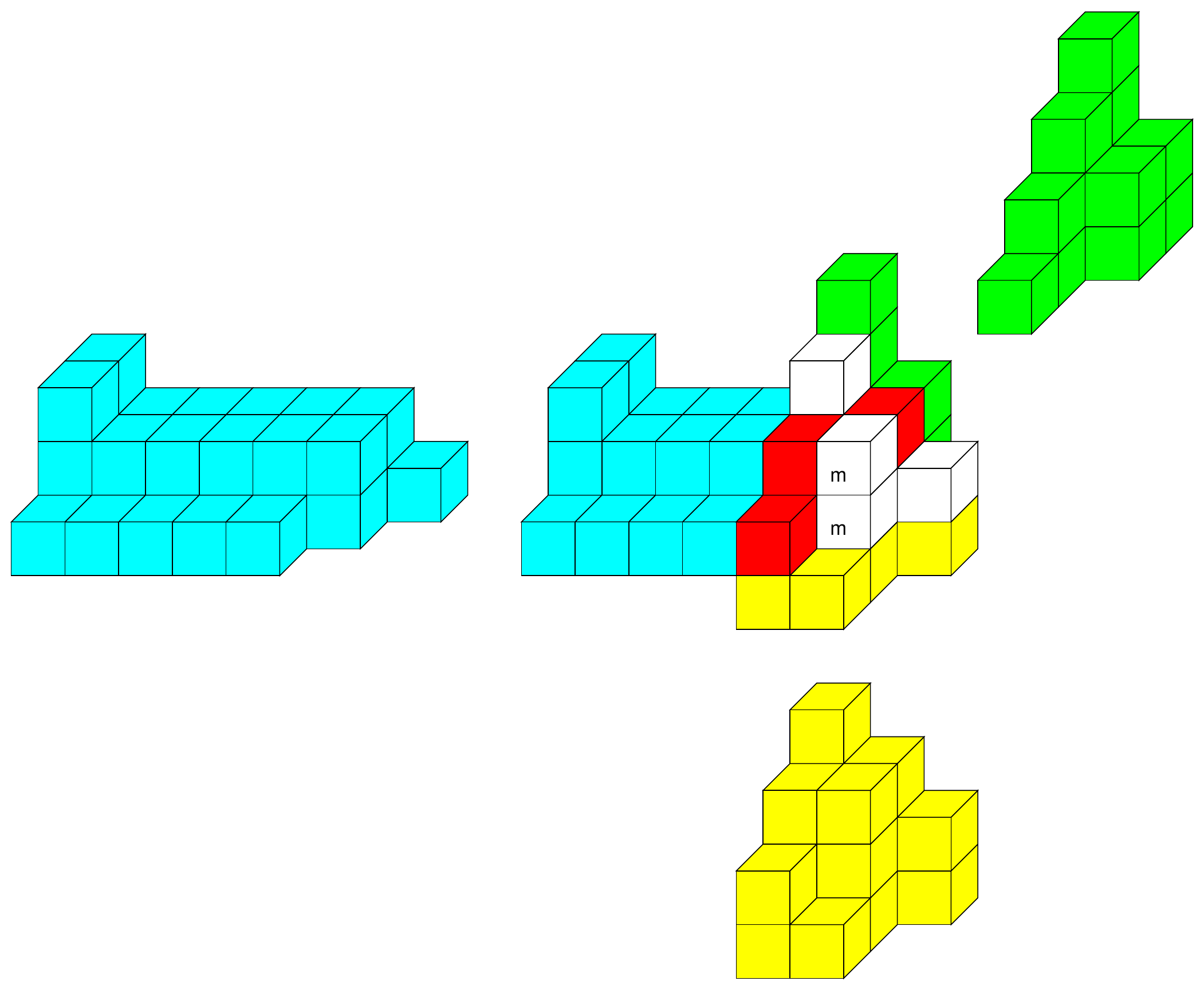}
\caption{Example with $\cC_{{\bspi}} \cong \PP^1$.}
\end{figure}

\begin{example} \label{moduliinlegs}
Let $\bspi \in \Pi(\ccR_\alpha, (\includegraphics[width=.3in]{1+1+1},\includegraphics[width=.2in]{2+1}, \varnothing),  (\varnothing, \includegraphics[width=.12in]{s}, \varnothing),  (\includegraphics[width=.3in]{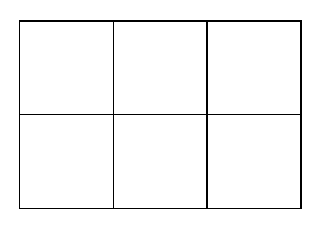},\includegraphics[width=.2in]{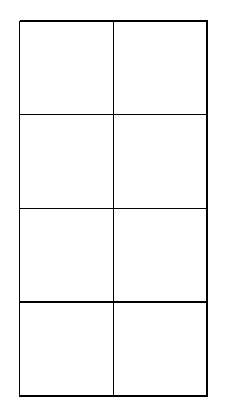},\includegraphics[width=.2in]{2+1}))$ be given by Figure 3. The dots indicate infinite legs. One connected component (of vertex type) fails (\ref{proj}) so has no moduli. The other two connected components (of edge type) satisfy (\ref{proj}) so $\cC_{\bspi} \cong \PP^1 \times \PP^1$.
\end{example}

\begin{figure} 
\includegraphics[width=4in]{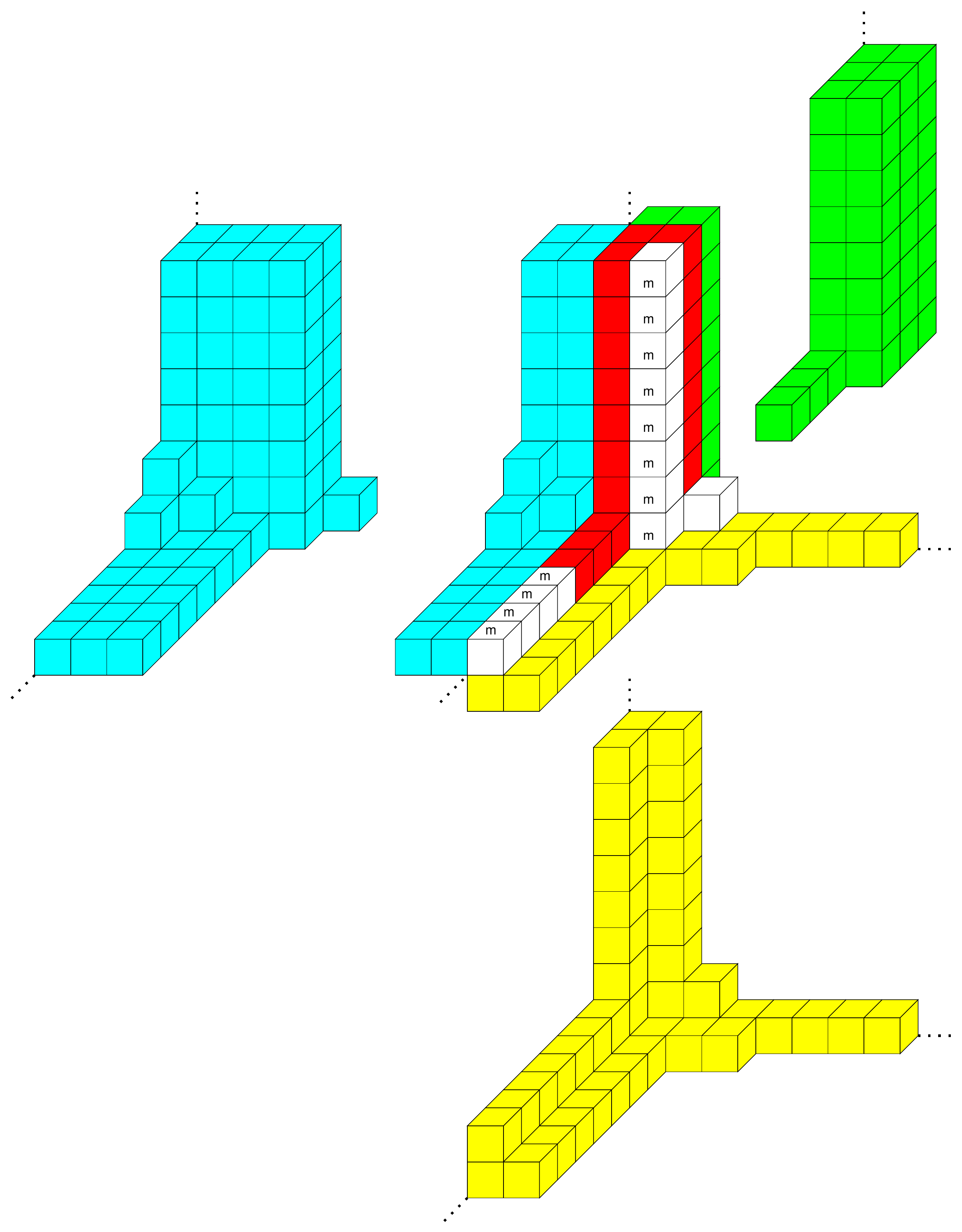}
\caption{Example with $\cC_{\bspi} \cong \PP^1 \times \PP^1$.}
\end{figure}


\begin{example} \label{littlebox}
Let $\bspi \in \Pi(\ccR_\alpha, (\includegraphics[width=.4in]{1+1+1+1}, \includegraphics[width=.21in]{1+1}, \varnothing),  (\varnothing, \includegraphics[width=.12in]{s}, \varnothing),  (\includegraphics[width=.3in]{2+2+2},\includegraphics[width=.2in]{4+4},\includegraphics[width=.2in]{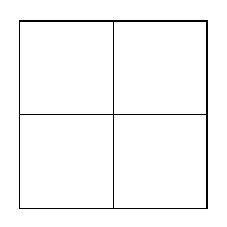}))$ be given by Figure 4 (left). Then $\cC_{\bspi} \cong \PP^1$. However, adding one box in the corner Figure 4 (right) kills the moduli and results in $\cC_{\bspi} \cong pt$. Note that in this example $\pi_{\mathrm{in}} = \varnothing$.
\end{example} 

\begin{figure} 
\includegraphics[width=6in]{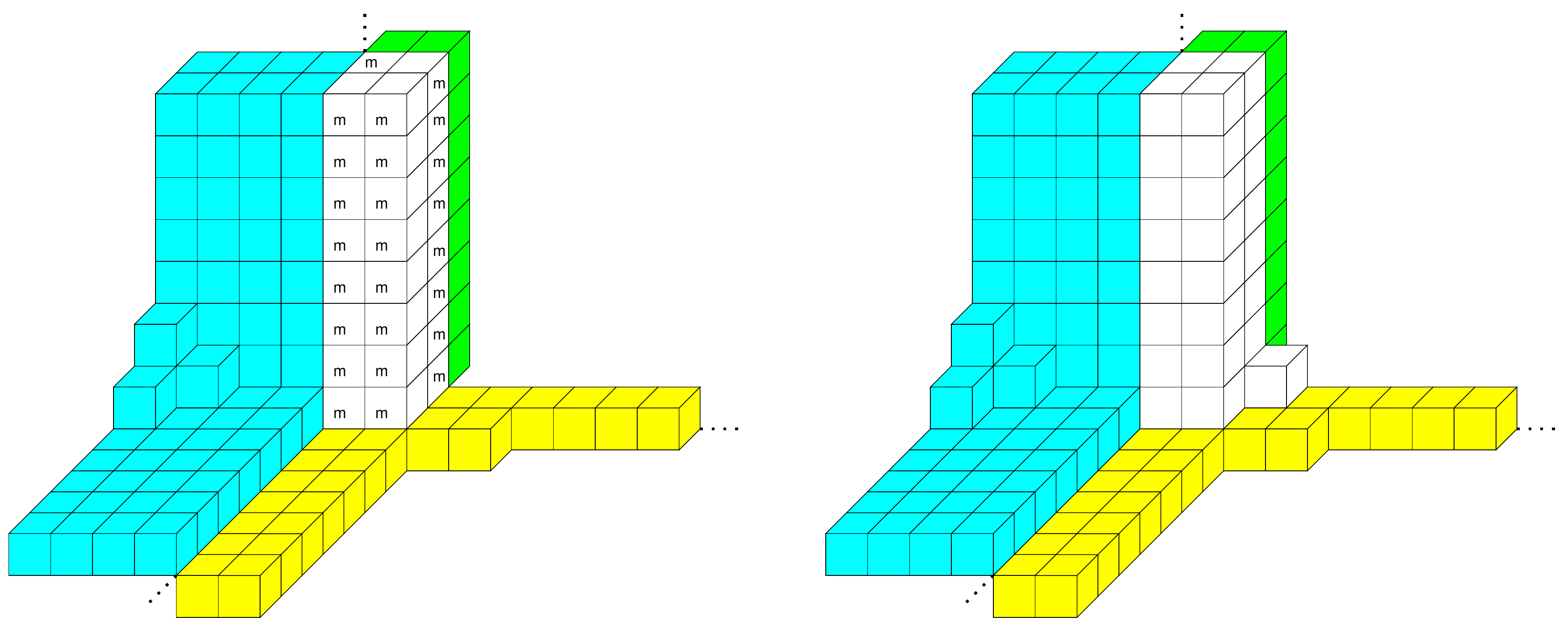}
\caption{Example with $\cC_{\bspi} \cong \PP^1$ (left) and $\cC_{\bspi} \cong pt$ (right).}
\end{figure}

We learn an important lesson from Example \ref{littlebox}: \emph{if we are given a double box configuration $\bspi \in \Pi(\ccR_\alpha, \bslambda_1,  \bslambda_2,  \bslambda_3)$, then we cannot decide from the $\bslambda_i$ alone whether there are edge moduli}. Informally: the ``corners'' influence the ``legs''. This means we ``cannot split'' vertex and edge contributions. A similar phenomenon occurs when discussing the vertex/edge formalism of rank 2 DT theory in Section \ref{vertex/edge} (see Example \ref{withlegs}). We are led to consider ``global'' double box configurations in the following sense.
\begin{definition} \label{Z}
Consider a smooth projective toric 3-fold $X$ and a $T$-equivariant rank 2 reflexive sheaf $\ccR$ on $X$ defined by toric data $({{\bf{u}}, {\bf{v}},\bf{p}})$ such that in each chart $\ccR_\alpha := \ccR|_{U_\alpha}$ is singular. (The general case, where some $\ccR_\alpha$ can be locally free, is discussed in Remark \ref{deco}.) For each edge $\alpha\beta \in E(X)$, fix a double square configuration 
$$
\hat{\bslambda} := \{\bslambda_{\alpha\beta}\}_{\alpha\beta \in E(X)}.
$$ 
For each $\alpha \in V(X)$ we denote its neighbouring vertices by $\beta_1, \beta_2, \beta_3$. The labeling is chosen such that $C_{\alpha \beta_i}$ is defined by $\{x_{i'} = x_{i''} = 0\}$ for all $i = 1,2,3$ and $i',i'' \in \{1,2,3\} \setminus \{i\}$ distinct. Fix a double box configuration 
\[
\bspi_\alpha \in \Pi(\ccR_\alpha, \bslambda_{\alpha\beta_1},  \bslambda_{\alpha\beta_2},  \bslambda_{\alpha\beta_3}).
\]
Then the collection $\hat{\bspi} := \{\bspi_\alpha\}_{\alpha \in V(X)}$ describes the characteristic function $\bschi$ of a $T$-equivariant rank 2 torsion free sheaf on $X$ with reflexive hull $\ccR$ by Proposition \ref{lem1}. We denote the collection of all such double box configurations by $\Pi(\ccR,\hat{\bslambda})$. To each $\bspi_\alpha$ we associated a moduli factor $\cC_{\bspi_\alpha}$ (Definition \ref{modfactor}). Then
\[
\cC_{\bschi} \cong \cC_{\hat{\bspi}} := \prod_{\alpha \in V(X)} \cC_{\bspi_\alpha} / \sim.
\]
Here $x \sim y$ when for some $\alpha\beta \in E(X)$, there is a local edge component containing $x$ in chart $U_\alpha$ and a local edge component containing $y$ in chart $U_\beta$ and both components match under gluing.\footnote{Note that it might happen that an edge component with moduli gets matched with an edge component without moduli, in which case the relation $\sim$ kills the moduli.} We define the \emph{weight} of $\hat{\bspi}$ as
\[
\omega(\hat{\bspi}) := e(\cC_{\hat{\bspi}}),
\]
where $e(\cdot)$ denotes topological Euler characteristic. Recall that for each $\alpha\beta \in E(X)$, we have (Section \ref{toric3folds})
\[
N_{C_{\alpha\beta}/X} \cong \O_{\PP^1}(m_{\alpha\beta}) \oplus \O_{\PP^1}(m_{\alpha \beta}^{\prime}).
\]
For each $\alpha\beta \in E(X)$, there are two faces $\rho_{1,\alpha\beta}$, $\rho_{2,\alpha\beta}$ which share $\alpha\beta$ as an edge and two disjoint faces $\rho_{3,\alpha\beta}$, $\rho_{4,\alpha\beta}$ connected by the edge $\alpha\beta$. Suppose face $\rho_{3,\alpha\beta}$ contains vertex $\alpha$ and $\rho_{4,\alpha\beta}$ contains $\beta$. We define the \emph{generating function} associated to $X$, $\ccR$, and $\bslambda_{\alpha\beta}$
\begin{align*}
\sfZ_{X, \ccR, \hat{\bslambda}}(q) &:= \sum_{\hat{\bspi} \in \Pi(\ccR, \hat{\bslambda})} \omega(\hat{\bspi}) q^{\chi(\hat{\bspi})}, \\
\chi(\hat{\bspi}) &:=  \sum_{\alpha \in V(X)} |\bspi_\alpha| + \sum_{\alpha\beta \in E(X)} \Big( f_{m_{\alpha\beta},m_{\alpha\beta}'}(\bslambda_{\alpha\beta}) + g_{u_{\rho_{3,\alpha\beta}},u_{\rho_{4,\alpha\beta}},v_{\rho_{3,\alpha\beta}},v_{\rho_{4,\alpha\beta}}}(\bslambda_{\alpha\beta}) \Big), \\
f_{m,m'}(\bslambda) &:= \sum_{i=1}^{3} \sum_{(k_1,k_2) \in \lambda_i} (-m k_1 - m' k_2 +1) - \sum_{(k_1,k_2) \in \bslambda_{\mathrm{out}}} (-m k_1 - m' k_2 +1), \\
g_{u,u',v,v'}(\bslambda) &:= - |\bslambda| (u+u'+v+v') + |\lambda_3|(v+v'),
\end{align*}
for any $\bslambda = (\lambda_1,\lambda_2,\lambda_3)$. There is one issue with this definition: $f_{m,m'}(\bslambda)$, $g_{u,u',v,v'}(\bslambda)$ depend on choice of representative of $\bslambda$. In the formula for $\chi(\hat{\bspi})$, it is assumed that for all $\alpha \in V(X)$, we choose representatives $(\pi_{\alpha,1},\pi_{\alpha,2},\pi_{\alpha,3})$ of $\bspi_\alpha$ with the following property: \\

\noindent For any $\alpha\beta \in E(X)$, suppose without loss of generality that $C_{\alpha\beta}$ is given by $\{x_2=x_3=0\}$ in chart $U_\alpha$ and $\{x_{2}'=x_{3}'=0\}$ in chart $U_\beta$. Suppose the toric data of $\ccR_\alpha$ in chart $U_\alpha$ is given by $(u_1,u_2,u_3)$, $(v_1,v_2,v_3)$ and in chart $U_\beta$ by $(u_1',u_2',u_3')$, $(v_1',v_2',v_3')$. Then for all $i=1,2,3$
\begin{align*}
&P_1(\pi_{\alpha,1} \setminus [u_1,u_1+N] \times \Z^2) = P_1(\pi_{\beta,1} \setminus [u_1',u_1'+N] \times \Z^2) \\
&P_1(\pi_{\alpha,2} \setminus [u_1,u_1+N] \times \Z^2) = P_1(\pi_{\beta,2} \setminus [u_1',u_1'+N] \times \Z^2) \\
&P_1(\pi_{\alpha,3} \setminus [u_1,u_1+N] \times \Z^2) = P_1(\pi_{\beta,3} \setminus [u_1',u_1'+N] \times \Z^2),
\end{align*}


where $P_1 : \R^3 \rightarrow \R^2$ denotes projection as before and $N \gg 0$. \\

\noindent This says that not only do $\bspi_\alpha$ and $\bspi_\beta$ have matching asymptotics (namely $\bslambda_{\alpha\beta}$), but the legs of their constituent partitions $\pi_{\alpha,i}$, $\pi_{\beta,i}$ have matching asymptotics too. This choice makes it more convenient to do the \v{C}ech calculation below. Even though $f_{m,m'}(\bslambda)$, $g_{u,u',v,v'}(\bslambda)$ depend on choice of representatives, the formula for $\chi(\hat{\bspi})$ does not by the following proposition:
\end{definition}

\begin{proposition} \label{formulachi}
Let $X$ be a smooth projective toric 3-fold. Let $\ccR$ be a $T$-equivariant rank 2 reflexive sheaf on $X$ described by toric data $({\bf{u}},{\bf{v}},{\bf{p}})$ and let $\ccR \twoheadrightarrow \cQ$ be a $T$-equivariant quotient of dimension $\leq 1$. Let $\hat{\bspi} = \{ \bspi_\alpha \}_{\alpha \in V(X)}$ correspond to the characteristic function of the kernel of $\ccR \twoheadrightarrow \cQ$, where 
\[
\bspi_\alpha \in \Pi(\ccR_\alpha, \bslambda_{\alpha\beta_1},  \bslambda_{\alpha\beta_2},  \bslambda_{\alpha\beta_3}).
\]
Suppose for all $\alpha \in V(X)$ we choose representatives of $\bspi_\alpha$ as described in Definition \ref{Z}. Then
\[
\chi(\cQ) = \chi(\hat{\bspi}).
\]
Moreover
\[
c_{3}(\cQ) = 2 \chi(\cQ) - \sum_{\alpha\beta \in E(X)} |\bslambda_{\alpha\beta}| (c_1(X) C_{\alpha\beta}).
\]
\end{proposition}
\begin{proof}
This follows from a \v{C}ech calculation using the $T$-invariant open affine cover $U_\alpha$, $\alpha \in V(X)$ like in \cite[Lem.~5]{MNOP1}. The upshot is to choose representatives of each $\bspi_\alpha$ as in the previous definition and then calculate the $T$-representation
$$
\bigoplus_{\alpha \in V(X)} H^0(U_\alpha,\cQ) - \bigoplus_{\alpha\beta \in E(X)} H^0(U_{\alpha\beta}, \cQ).
$$
Cancelling infinite terms and finally setting $t_1=t_2=t_3=1$ gives the formula for $\chi(\cQ)$. The formula for $c_3(\cQ)$ follows from Hirzebruch-Riemann-Roch.
\end{proof}

\begin{remark} \label{openZ}
In Definition \ref{Z}, we fixed a smooth \emph{projective} toric 3-fold $X$. Suppose $Y$ is any smooth toric 3-fold with vertices $V(Y)$, edges $E(Y)$, and $T$-equivariant rank 2 reflexive sheaf $\ccR$ on $Y$.\footnote{When $Y$ is non-compact, one can still associate an open polyhedron to it, which is a subset of the polyhedron of any toric compactification of $X$ (Definition \ref{toriccompact}).} Denote by $E_{c}(Y)$ the edges which correspond to \emph{compact} lines $\PP^1 \cong C_{\alpha\beta} \subset Y$. Suppose for each $\alpha\beta \in E_c(Y)$, we fix a double square configuration $\bslambda_{\alpha\beta}$, then one can still consider
\begin{align*}
\sfZ_{Y, \ccR, \hat{\bslambda}}(q) &:= \sum_{\hat{\bspi} \in \Pi(\ccR, \hat{\bslambda})} \omega(\hat{\bspi}) q^{\chi(\hat{\bspi})}, \\
\chi(\hat{\bspi}) &:= \sum_{\alpha \in V(Y)} |\bspi_\alpha| + \sum_{\alpha\beta \in E_c(Y)} \Big( f_{m_{\alpha\beta},m_{\alpha\beta}'}(\bslambda_{\alpha\beta}) + g_{u_{\rho_{3,\alpha\beta}},u_{\rho_{4,\alpha\beta}},v_{\rho_{3,\alpha\beta}},v_{\rho_{4,\alpha\beta}}}(\bslambda_{\alpha\beta}) \Big),
\end{align*}
where $\rho_{3,\alpha\beta}$ and $\rho_{4,\alpha\beta}$ are defined as in Definition \ref{Z}. In applications, $Y$ has toric compactification $X$ and $\ccR$ is the restriction of a $T$-equivariant rank 2 reflexive sheaf on $X$. Then $\sfZ_{Y, \ccR, \hat{\bslambda}}(q)$ is part of a generating function on $X$.

Sometimes, for computational reasons, it is useful to consider an analog of $\sfZ_{Y, \ccR, \hat{\bslambda}}(q)$, where one fixes a double square configuration $\bslambda_{\alpha\beta}$ for all (not necessarily compact!) edges $\alpha\beta \in E(Y)$
\begin{equation}
\sfZ_{Y, \ccR, \hat{\bslambda}}^{\comb}(q) := \sum_{\hat{\bspi} \in \Pi(\ccR, \hat{\bslambda})} \omega(\hat{\bspi}) q^{ \sum_{\alpha \in V(Y)} |\bspi_\alpha|}.
\end{equation}
When both defined, $\sfZ_{Y, \ccR, \hat{\bslambda}}^{\comb}(q)$ and $\sfZ_{Y, \ccR, \hat{\bslambda}}(q)$ coincide up to an overall multiplicative factor of $q$ to some power. This power involves information of the normal bundle to the lines $C_{\alpha\beta}$, whereas $\sfZ_{Y, \ccR, \hat{\bslambda}}^{\comb}(q)$ does not depend on the geometry but only on $\ccR$ and $\hat{\bslambda}$.
\end{remark}

\begin{remark} \label{deco}
It is not hard to redo this section in the case $\ccR_\alpha$ is locally free, i.e.~$v_i=0$ for some $i$ or $p_i = p_j$ for some $i,j$ (Proposition \ref{rank2singular}). We think of these as degenerate cases. Then $\ccR_\alpha$ is $T$-equivariantly decomposable. In this setting, a double box configuration is defined as a pair of partitions $\bspi = (\pi_1, \pi_2)$ without any conditions, where we impose an equivalence relation on such pairs as before (and similar for double square configurations). After adapting the formulae for $|\bspi|$, $|\bslambda|$ accordingly, the formula of Proposition \ref{formulachi} holds for any $T$-equivariant rank 2 reflexive sheaf on $X$. The details are straight-forward and left to the reader.
\end{remark}

\subsection{Combinatorial formulae} \label{combsection}

Let $Y = \C^3$ and $\ccR$ a $T$-equivariant rank 2 reflexive sheaf on $Y$ with toric data $({\bf{u}},{\bf{v}},{\bf{p}})$. The following generating function was introduced in Remark \ref{openZ}
\begin{equation*}
\mathsf{Z}_{\C^3, \ccR, \hat{\bslambda}}^{\comb}(q) := \sum_{\bspi \in \Pi(\ccR, \hat{\bslambda})} \omega(\bspi) q^{ |\bspi|}.
\end{equation*}
Here $\hat{\bslambda} = (\bslambda_1,\bslambda_2,\bslambda_3)$ are three double square configurations describing fixed asymptotics along the $x_1$-, $x_2$-, and $x_3$-axes. In Section \ref{Euler} we will see that these generating functions form the building blocks for calculating \eqref{tocalc} for rank $r=2$ and toric 3-folds $X$ in examples. 

Recall that $\ccR$ is singular if and only if all $v_i>0$ and all $p_i$ are mutually distinct (Proposition \ref{rank2singular}). In the case $\ccR$ is locally free, $\sfZ_{\C^3, \ccR, \hat{\bslambda}}^{\comb}(q)$ decomposes. Each double square configurations $\bslambda_i$ is given by a pair of 2D partitions $(\lambda_i,\mu_i)$ and
\begin{equation} \label{splitZ}
\mathsf{Z}_{\C^3, \ccR, \hat{\bslambda}}^{\comb}(q) = \mathsf{W}_{\lambda_1,\lambda_2,\lambda_3}(q) \mathsf{W}_{\mu_1,\mu_2,\mu_3}(q), 
\end{equation}
where $\sfW_{\lambda_1,\lambda_2,\lambda_3}(q)$ denotes the usual topological vertex first discovered in \cite{AKMV}.\footnote{By definition $\sfW_{\lambda_1,\lambda_2,\lambda_3}(q) = \sum_{\pi} q^{|\pi|}$, where the sum is over all 3D partitions $\pi$ with asymptotics $\lambda_1,\lambda_2,\lambda_3$ and $|\pi|$ denotes the renormalized volume mentioned in Definition \ref{doubleboxconf}.} This product structure follows from the fact that $\ccR$ splits $T$-equivariantly as a sum of two $T$-equivariant line bundles (Proposition \ref{rank2singular}), which gives an extra $\C^*$-scaling action on the factors. Using this extra $\C^*$, \eqref{splitZ} follows.

The interesting case is when $\ccR$ is singular. The generating functions $\sfZ_{\C^3, \ccR, \hat{\bslambda}}^{\comb}(q)$ can then be computed by relating double box configurations to double dimers. In the case of no legs (or small legs) this leads to closed formulae. The following two theorems are proved in a forthcoming combinatorial paper \cite{GKY}. The first two authors recently found a geometric proof motivated by wall-crossing and Hall algebra methods \cite{GK2, GK3}.
\begin{theorem} \cite{GKY, GK2} \label{Ben}
Let $\ccR$ be a singular $T$-equivariant rank 2 reflexive sheaf on $\C^3$ described by toric data $({\bf{u}}, {\bf{v}}, {\bf{p}})$. Then\footnote{Here we take $\hat{\bslambda} = (\bslambda_1,\bslambda_2,\bslambda_3)$ with each $\bslambda_i = (\varnothing,\varnothing,\varnothing)$.}
\[
\mathsf{Z}_{\C^3, \ccR, \varnothing, \varnothing, \varnothing}^{\comb}(q)  = M(q)^2 \prod_{i=1}^{v_1} \prod_{j=1}^{v_2} \prod_{k=1}^{v_3} \frac{1-q^{i+j+k-1}}{1-q^{i+j+k-2}},
\]
where $M(q) = \prod_{k>0} 1/(1-q^k)^k$ denotes the MacMahon function.
\end{theorem}

\begin{theorem} \cite{GKY, GK3} \label{Benwithlegs}
Let $\ccR$ be a singular $T$-equivariant rank 2 reflexive sheaf on $\C^3$ described by toric data $({\bf{u}}, {\bf{v}}, {\bf{p}})$. Suppose $(v_1,v_2,v_3) = (1,1,1)$ and $$(\bslambda_1,\bslambda_2,\bslambda_3) = ((\varnothing,\varnothing,\varnothing),(\includegraphics[width=.15in]{s},\varnothing,\varnothing),(\varnothing,\varnothing,\varnothing)).$$ Then
\[
\sfZ_{\C^3, \ccR, \hat{\bslambda}}^{\comb}(q) =M(q)^2 \frac{1+q^2}{1-q}.
\]
\end{theorem}


Denote by $\Quot_{\C^3}(\ccR,n)$ the Quot scheme of quotients
$$
\ccR \twoheadrightarrow \cQ
$$
where $\cQ$ is $0$-dimensional and has length $n$. The following is a corollary of Theorem \ref{Ben}:
\begin{corollary} \label{QuotC3}
Let $\ccR$ be a $T$-equivariant rank 2 reflexive sheaf on $\C^3$ described by toric data $({\bf{u}}, {\bf{v}}, {\bf{p}})$. Then 
$$
\sum_{n = 0}^{\infty} e(\Quot_{\C^3}(\ccR,n)) q^n = \left\{ \begin{array}{cc} M(q)^2 & \mathrm{if \ } \ccR \mathrm{ \ is \ locally \ free } \\ M(q)^2 \prod_{i=1}^{v_1} \prod_{j=1}^{v_2} \prod_{k=1}^{v_3} \frac{1-q^{i+j+k-1}}{1-q^{i+j+k-2}} & \mathrm{if \ } \ccR \mathrm{ \ is \ singular}. \end{array} \right.
$$
\end{corollary}

Since Theorems \ref{Ben} and \ref{Benwithlegs} play such a central role in our calculations, we give a brief outline of the proofs to give a flavour of the ideas involved. This sketch is by no means meant to be self-contained. \\


\noindent \textbf{Combinatorial proof.} We first replace our double box configurations with different combinatorial objects, namely, configurations of the hexagonal-lattice double dimer model with certain tripartite boundary conditions, as defined in~\cite{KW2}.  This is essentially done by separating the double-dimer model into two single-dimer models, and then using the standard correspondence between the single dimer model on the honeycomb lattice and plane partitions, though there is some work to do in order to see that the ``tripartite'' boundary conditions of~\cite{KW2} are the correct ones.  Figure 5 is the double dimer model we associate to the example of the double box configuration in Figure 2.
\begin{figure}
\includegraphics[width=2in]{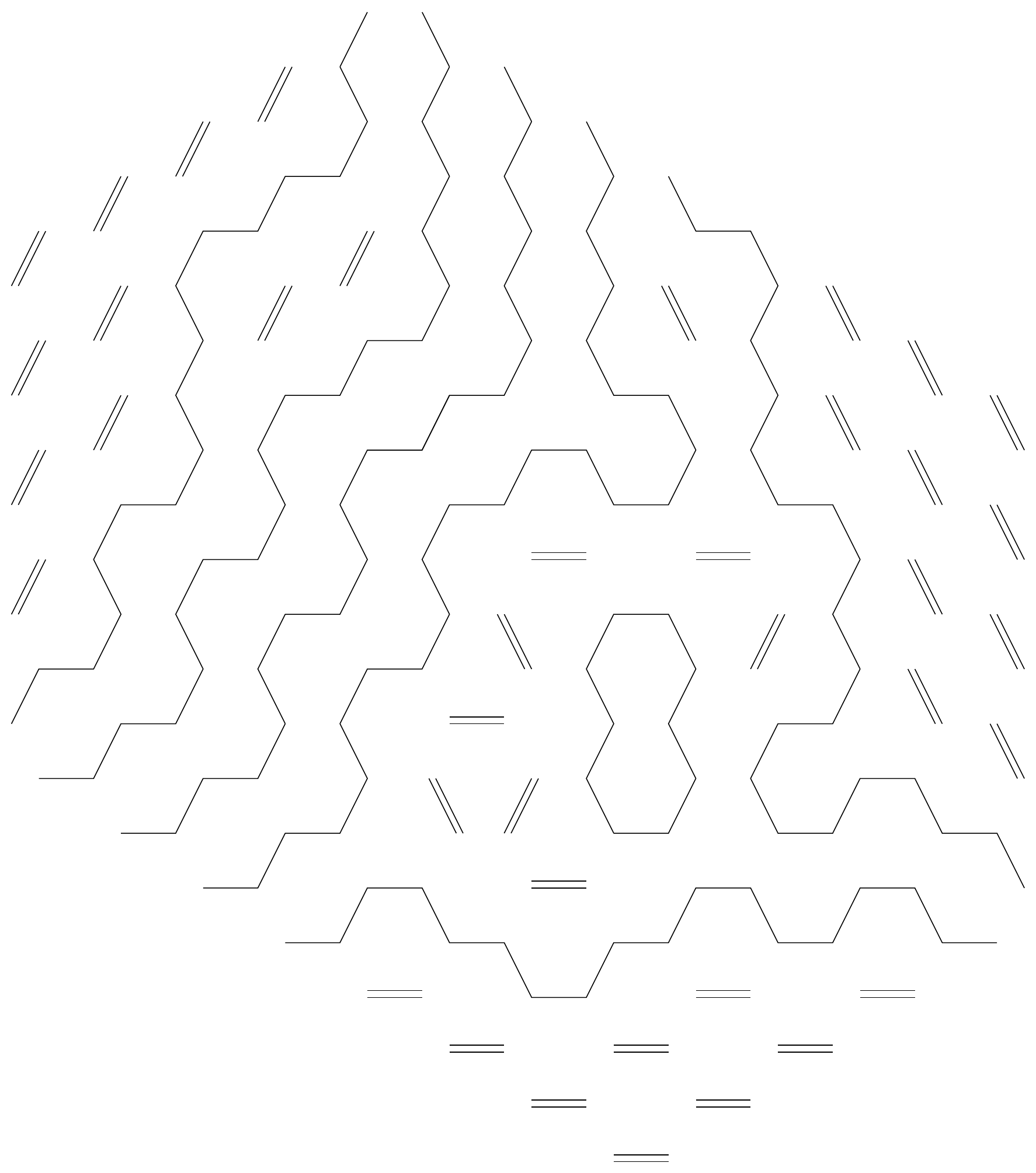}
\caption{Double dimer model associated to Fig.~2}
\end{figure}

We put these dimer models into a generating function by weighting each with $cq^m$ for some $c, m\in \mathbb{Z}_{\ge 0}$ as we did in the case of double box configurations. The correspondence between double box configurations and double dimer models is many to one, but it is weight-preserving, and it does preserve the generating functions. Weight preserving here means each single dimer configuration contributes the same power of $q$ as any of the double box configurations in its preimage, and that the coefficient of $q$ is the sum of the coefficients of $q$ of all double
box configurations in its preimage.

We then compute the generating function of this double dimer model using the technology of Kenyon-Wilson \cite{KW1, KW2}. It is a certain rational function in partition functions of single dimer models on the hexagonal lattice. We realize each of these single dimer model partition functions as a product of vertex operators, and use Fock space computational techniques due to Okounkov \cite{Oko, OR1, OR2, ORV} to simplify the result and identify the factor $M(q)^2$ in the generating function. We then reinterpret the resulting quantity as the generating function of a new single-dimer problem - namely the classical $q$-count of boxed plane partitions~\cite{Mac}. This last step accounts for the final product in Theorem \ref{Ben}. \\

\noindent \textbf{Geometric proof.} The geometric proof proceeds as follows. We first realize the left hand sides of Theorems \ref{Ben}, \ref{Benwithlegs} as generating functions of the Euler characteristics of Quot schemes of reflexive sheaves $\mathcal{R}$ as in Corollary \ref{QuotC3}. Subsequently we apply Hall algebra/wall-crossing arguments of  \cite{Bri, Joy, KS, ST} in order to to rewrite this as $M(q)^2$ times the generating function $$\sum_{n = 0}^{\infty} e(\operatorname{Hilb}^n(\operatorname{Sing}(\mathcal{R}),n)) q^n,$$ where $\operatorname{Sing}(\mathcal{R})$ is a certain 0-dimensional subscheme of $\mathbb{C}^3$ supported at the point where $\mathcal{R}$ is singular.  This generating function accounts for the final product in Theorem \ref{Ben}.

\section{Applications to Euler characteristics} \label{Euler}

Let $X$ be a smooth projective toric 3-fold with polarization $H$. Let $\M_{X}^{H}(2,c_1,c_2,c_3)$ be the moduli space of rank 2 $\mu$-stable torsion free sheaves on $X$ with indicated Chern classes. Recall that we do \emph{not} consider strictly $\mu$-semistable sheaves. Consider the generating function
\[
\sfZ_{X,H,c_1,c_2}(q) := \sum_{c_3 \in \Z}  e(\M_{X}^{H}(2,c_1,c_2,c_3)) q^{c_3}.
\]
In the introduction, we noted that the double dual map is a constructible morphism
\[
(\cdot)^{**} : \M_{X}^{H}(2,c_1,c_2,c_3) \longrightarrow \bigsqcup_{c_{2}', c_{3}'} \N_{X}^{H}(2,c_1,c_{2}',c_{3}').
\] 
Here $\N_{X}^{H}(2,c_1,c_{2}',c_{3}')$ is empty unless 
\[
\frac{c_{1}^{2} H}{4} \leq c_{2}' H \leq c_2 H,
\]
where the first is the Bogomolov inequality \cite[Thm.~7.3.1]{HL} and the second follows from ampleness of $H$. Moreover, for a given $c_{2}'$ there exists a constant $K(X,H,c_1,c_{2}')$ such that $\N_{X}^{H}(2,c_1,c_{2}',c_{3}')$ is empty unless  \cite[Prop.~3.6]{GK1}
\begin{equation} \label{bound}
0 \leq c_{3}' \leq K(X,H,c_1,c_2').
\end{equation}

The generating function $\sfZ_{X,H,c_1,c_2}(q)$ is a Laurent series in $q^{-1}$. Using the methods of the previous section, it can be expressed in terms of Euler characteristics of Quot schemes as follows:
\begin{itemize}
\item Find a GIT representation of all connected components $\cC$ of all reflexive fixed loci $\N_{X}^{H}(2,c_1,c_{2}',c_{3}')^T$. These are configuration spaces of points on $\PP^1$ or open subsets thereof. This was shown in \cite{GK1}. E.g.~for $X=\PP^3$, each connected component $\cC$ is isomorphic to either a reduced isolated point $pt$ or $\C^{*} \setminus \{pt\}$ \cite{GK1}. 
\item Compute the Euler characteristics $e(\cC)$. Euler characteristics of configuration spaces of points on $\PP^1$ have been computed in general by Klyachko \cite{Kly3}.
\item Let $\cC$ be a connected component of $\N_{X}^{H}(2,c_1,c_{2}',c_{3}')^T$. The toric data $({\bf{u}},{\bf{v}},{\bf{p}})$ of any $T$-equivariant reflexive sheaf $[\ccR] \in \cC$ has fixed ${\bf{u}}, {\bf{v}}$ and (in general) varying ${\bf{p}}$. We stress the fact that ${\bf{u}}$, ${\bf{v}}$ are constant over the connected component $\cC$ by writing ${\bf{u}}(\cC)$, ${\bf{v}}(\cC)$. The components of the vectors ${\bf{u}}(\cC)$, ${\bf{v}}(\cC)$ are indexed by rays $\rho \in F(X)$ (Definition \ref{toricdata})
\begin{align*}
{\bf{u}}(\cC) = \{ u_{\rho}(\cC) \}_{\rho \in F(X)}, \ {\bf{v}}(\cC) = \{ v_{\rho}(\cC) \}_{\rho \in F(X)}.
\end{align*}
The Chern classes $c_2', c_3'$ can be expressed in terms of 
$({\bf{u}}(\cC), {\bf{v}}(\cC), {\bf{p}}(\cC))$ as explained in Section \ref{equivsh}. In particular for any $\alpha \in V(X)$ either all restrictions $\ccR|_{U_\alpha}$ of elements of $[\ccR] \in \cC$ have a singularity at the origin or not. In the former case the length of the singularity is the same for all elements $[\ccR] \in \cC$ and is given by the formula of Proposition \ref{rank2singular}. Since $c_3'$ is the sum of the lengths of the singularities over all charts (Proposition \ref{rank2singular}), all elements $[\ccR] \in \cC$ have the same third Chern class.


At the level of closed points, the fibre of $(\cdot)^{**}$ over $[\ccR] \in \cC$ is $\Quot_X(\ccR,c_{2}'',c_{3}'')^T$
\begin{align*}
c_{2}' &= c_2 + c_{2}'', \\ 
c_{3}' &= c_3 + c_{3}'' + c_1 c_{2}''.
\end{align*}
Proposition \ref{prodP1s} implies that all fibres $\Quot_X(\ccR,c_{2}'',c_{3}'')^T$ over $[\ccR] \in \cC$ are isomorphic and we may write
\[
\Quot_X(\cC,c_{2}'',c_{3}'')^T := \Quot_X(\ccR,c_{2}'',c_{3}'')^T
\]
for the isomorphism class. The connected components of the fibres $\Quot_X(\cC,c_{2}'',c_{3}'')^T$ are all products of $\PP^1$'s and were described in Section \ref{3Dpartitions}. 
\item Putting everything together gives the following structure formula
\begin{equation} \label{general}
\sfZ_{X,H,c_1,c_2}(q) = \sum_{c_3} \sum_{c_{2}', c_{3}'} \sum_{\cC \subset \N_{X}^{H}(2,c_1,c_{2}',c_{3}')^T}  \sum_{\cC' \subset \Quot_X(\cC,c_{2}'-c_2,c_{3}'-c_3 - c_1 (c_{2}' - c_2))^T} e(\cC) e(\cC') q^{c_3}.
\end{equation}
\end{itemize}

In the rest of this section we work out this generating function in the following cases: 
\begin{itemize}
\item $c_2$ is minimal in the sense of Proposition \ref{minc2a}. 
\item $X = \PP^3$ and $c_2=1,2$.
\end{itemize}

\subsection{Minimal second Chern class} \label{sectionmin}

We start with two propositions.
\begin{proposition} \label{minc2a}
Let $X$ be a smooth projective toric 3-fold with polarization $H$ and fix $r$ and $c_1$. Let $c_2$ be such that $c_2 H$ is minimal with the property that there exists a rank $r$ $\mu$-stable torsion free sheaf on $X$ with Chern classes $c_1, c_2$. Then for any such sheaf the cokernel $\cQ = \F^{**} / \F$ is 0-dimensional.
\end{proposition}
\begin{proof}
Let $\F$ be as stated. Since $\F$ is $\mu$-stable, its Chern classes satisfy the Bogomolov inequality \cite[Thm.~7.3.1]{HL}:
$$
c_2 H \geq \frac{(r-1)c_{1}^{2} H}{2r}.
$$
Since $c_2 H$ is bounded below, it makes sense to speak of $c_2$ satisfying the property that $c_2 H$ is the smallest value for which there exist rank $r$ $\mu$-stable torsion free sheaves on $X$ with Chern classes $c_1, c_2$. Suppose $\cQ$ is 1-dimensional. Use the same notation for Chern classes $c_{i}', c_{i}''$ as in the beginning of this section. Then $(-c_{2}'') H > 0$ by ampleness, so
$$
c_2' H = c_2 H + c_2'' H < c_2 H.
$$
Therefore $\F^{**}$ is a rank $r$ $\mu$-stable torsion free (reflexive) sheaf on $X$ with Chern classes $c_1, c_2'$ contradicting minimality.
\end{proof} 

\begin{proposition} \label{minc2b}
Let $X$ be a smooth projective toric 3-fold with polarization $H$ and fix $c_1$. Let $c_2$ be such that $c_2 H$ is minimal with the property that there exists a rank 2 $\mu$-stable torsion free sheaf on $X$ with Chern classes $c_1, c_2$. For each $\alpha \in V(X)$, denote by $\rho_{1,\alpha}, \rho_{2,\alpha}, \rho_{3,\alpha} \in F(X)$ the faces sharing vertex $\alpha$. Then
\begin{align*}
\sfZ_{X,H,c_1,c_2}(q) = &M(q^{-2})^{2e(X)} \sum_{c_{3}'} \sum_{\cC \subset \N_{X}^{H}(2,c_1,c_{2},c_{3}')^T} e(\cC) q^{c_{3}'} \\
&\times \prod_{{\scriptsize{\begin{array}{c} \alpha \in V(X) \\ \forall [\ccR] \in \cC : \ccR|_{U_\alpha} \ \mathrm{is \ singular} \end{array}}}} \prod_{i=1}^{v_{\rho_{1,\alpha}}(\cC)} \prod_{j=1}^{v_{\rho_{2,\alpha}}(\cC)} \prod_{k=1}^{v_{\rho_{3,\alpha}}(\cC)} \frac{1-(q^{-2})^{i+j+k-1}}{1-(q^{-2})^{i+j+k-2}}.
\end{align*}
\end{proposition}
\begin{proof}
Combining \eqref{general}, Prop.~\ref{minc2a}, Prop.~\ref{formulachi} gives
\[
\sfZ_{X,H,c_1,c_2}(q) = \sum_{c_{3}'} \sum_{\cC \subset \N_{X}^{H}(2,c_1,c_{2},c_{3}')^T} e(\cC) q^{c_{3}'} \prod_{\alpha \in V(X)}  \sfZ_{U_\alpha, \ccR|_{U_\alpha}, \varnothing,\varnothing,\varnothing}(q^{-2}),
\]
where $\sfZ_{U_\alpha, \ccR|_{U_\alpha}, \varnothing,\varnothing,\varnothing}(q^{-2})$ was introduced in Remark \ref{openZ}. The latter generating function is given by Theorem \ref{Ben}. 
\end{proof}

If the polyhedron $\Delta(X)$ is ``small'' (defined below), the connected components $\cC$ in the previous proposition turn out to be isolated $T$-fixed points. 
\begin{definition} \label{small}
Let $\Delta$ be a convex polyhedron such that each vertex has exactly 3 edges. We view $\Delta$ up to arbitrary stretching of edges and angles between edges (but not contracting them to a point). We say that $\Delta$ is \emph{small} if the following property holds. \\

\noindent \emph{Condition}: For any way of colouring one face red and three faces blue (and not colouring any of the other faces), the red face must share an edge with a blue face and at least two blue faces must share an edge.
\end{definition} 

\begin{example}
Any convex polyhedron such that each vertex has exactly 3 edges and $\leq 6$ faces is small. There are 4 of these: tetrahedron, cube, prism over a triangle, prism over a triangle with one corner cut off. Hence $\rk \Pic(X) \leq 3$ ensures $\Delta(X)$ is small. Polyhedra with $\geq 7$ faces are generally not small. However, a prism over a 5-gon has 7 faces but is small. Therefore 16 out of 18 smooth toric Fano 3-folds have small $\Delta(X)$ (the two with $\rk \Pic(X) = 5$ do not have small $\Delta(X)$ \cite{Bat, WW}).
\end{example}

\begin{theorem} \label{Eulermin}
Let $X$ be a smooth projective toric 3-fold such that $\Delta(X)$ is small. Let $H$ be a polarization and $c_1$ a first Chern class such that $\gcd(2,c_1H^2)=1$. Let $c_2$ be such that $c_2 H$ is minimal with the property that there exists a rank 2 $\mu$-stable torsion free sheaf on $X$ with Chern classes $c_1, c_2$. For each $\alpha \in V(X)$, denote by $\rho_{1,\alpha}, \rho_{2,\alpha}, \rho_{3,\alpha} \in F(X)$ the faces sharing vertex $\alpha$. Then the connected components of $\N_{X}^{H}(2,c_1,c_{2},c_{3}')^T$ are isolated reduced points for all $c_3'$. Moreover
\begin{align*}
\sfZ_{X,H,c_1,c_2}(q) = &M(q^{-2})^{2e(X)} \sum_{c_{3}'} \sum_{[\ccR] \in \N_{X}^{H}(2,c_1,c_{2},c_{3}')^T} q^{c_{3}'} \\
&\times \prod_{{\scriptsize{\begin{array}{c} \alpha \in V(X) \\ \ccR|_{U_\alpha} \ \mathrm{is \ singular} \end{array}}}} \prod_{i=1}^{v_{\rho_{1,\alpha}}(\ccR)} \prod_{j=1}^{v_{\rho_{2,\alpha}}(\ccR)} \prod_{k=1}^{v_{\rho_{3,\alpha}}(\ccR)} \frac{1-(q^{-2})^{i+j+k-1}}{1-(q^{-2})^{i+j+k-2}}.
\end{align*}
In particular, $\sfZ_{X,H,c_1,c_2}(q)$ equals $M(q^{-2})^{2e(X)}$ times a Laurent polynomial in $q$.
\end{theorem}
\begin{proof}
Let $\ccR$ be a $T$-equivariant rank 2 $\mu$-stable reflexive sheaf on $X$ with Chern classes $c_1$, $c_2$, toric data $({\bf{u}},{\bf{v}},{\bf{p}})$, and characteristic function $\bschi$. Suppose ${\bschi}$ has $a$ face components (Definitions \ref{vefcomps}, \ref{vefcomps2}). We claim $a=3$. The case $a<3$ cannot happen, because then $\ccR$ is decomposable hence $\mu$-unstable. If $a>3$, then the connected component $\cC$ of the $T$-fixed locus containing $\ccR$ has positive dimension. We show that this leads to a contradiction with $\Delta(X)$ is small and $c_2$ minimal in the remainder of the proof. Once this is established, the formula of the theorem follows from Proposition \ref{minc2b}. Note that for fixed $c_1,c_2$, the third Chern class $c_3'$ is bounded by \eqref{bound}. Since the triple product of Theorem \ref{Ben} is a polynomial (counting boxed 3D partitions) polynomiality follows.

Suppose at $a > 3$, then the moduli space $\mathcal{N}_{X}^{H}(\bschi)$ of $T$-equivariant rank 2 $\mu$-stable reflexive sheaves on $X$ with characteristic function $\bschi$ has dimension $>0$. The space of equivariant reflexive sheaves $\mathcal{N}_{X}^{H}(\bschi)$ of Remark \ref{a=3} is realized as the GIT quotient of a non-empty Zariski open subset $U_{X}^{H} \subset (\PP^1)^a$ by $\SL(2,\C)$. We denote its elements by $(q_1, \ldots, q_a)$. Let $V \subset (\PP^1)^a$ denote the Zariski open subset of elements $(q_1, \ldots, q_a)$ with all $q_i$ mutually distinct. 
Therefore, there is a $c_3'$ such that $V / \SL(2,\C)$ is a positive dimensional Zariski open subset 
$$
V / \SL(2,\C) \subset \N_{X}^{H}(2,c_1,c_{2},c_{3}')^T \subset \M_{X}^{H}(2,c_1,c_{2},c_{3}')^T.
$$
Consider the classical GIT problem of $a$ (ordered) points $(q_1, \ldots, q_a)$ on $\PP^1$.\footnote{This problem is discussed in \cite[Ch.~11]{Dol}.} We assign weights $w_1, \ldots, w_a$ to these points as follows. The characteristic function ${\bschi}$ has $a$ face components. For each face component, there is a maximal number of faces $\rho_{i_1}, \ldots, \rho_{i_k}$ such that the corresponding local face components are glued together in ${\bschi}$. In particular, $p_{\rho_{i_1}} = \cdots = p_{\rho_{i_k}}$. 
To this face component we assign weight
$$
w_i := (D_{\rho_{i_1}} H^2) v_{\rho_{i_1}} + \cdots + (D_{\rho_{i_k}} H^2) v_{\rho_{i_k}} > 0.
$$
Set $w := \sum_{i=1}^a w_i$, then each $w_i$ is smaller than $\frac{w}{2}$. 

Next we use the assumption $\Delta(X)$ is small. By Lemma \ref{lemmasmall} below, there exists a GIT semistable point configuration $(q_1, \ldots, q_a)$ such that there are $i,j$ with the following properties: 
\begin{itemize}
\item $q_i = q_j$ is distinct from all other $q_k$ and all other $q_k$ are mutually distinct,
\item there exists a chart $U_\alpha$ containing the local face components corresponding to both $i$ and $j$.
\end{itemize}

\noindent Informally: on the polyhedron $\Delta(X)$, there are two adjacent faces whose corresponding points on $\PP^1$ can become equal. The unions of faces $\Theta_i$ of the lemma correspond to the unions of local face components glued together by $\bschi$. Note that $w_i+w_j \leq \frac{w}{2}$ is a strict inequality, because $\gcd(2,c_1 H^2)=1$. Therefore, the above point configuration $(q_1, \ldots, q_a)$ corresponds to a $\mu$-stable torsion free sheaf $\F$ with characteristic function $\bschi$ (and hence Chern classes $c_1, c_2$). 
The reflexive hull $\F^{**}$ is $\mu$-stable with Chern classes  $c_1, c_2'$ for some $c_2'$. Since $q_i, q_j$ are equal, the characteristic function of $\F^{**}$ is obtained from the characteristic function ${\bschi}$ of $\F$ by adding a (non-empty) infinite leg of 1's. Therefore $\mathrm{ch}_2'H > \mathrm{ch}_2H$, or equivalently $c_2'H < c_2H$, contradicting minimality of $c_2$.
\end{proof}

\begin{lemma} \label{lemmasmall}
Let $a \geq 4$, $w_1, \ldots, w_a \in \Z_{>0}$, $w := \sum_{i=1}^{a} w_i$, and assume
$$
w_i \leq \frac{w}{2},
$$
for all $i=1, \ldots, a$. Let $\Delta$ be a small polyhedron. Let $\{\Theta_i \ : \ i =1, \ldots, a \}$ be a collection, such that 
\begin{itemize}
\item $\Theta_i$ is a union of faces of $\Delta$, for all $i=1, \ldots, a$.
\item $\Theta_i$ minus the vertices is connected, for all $i=1, \ldots, a$. 
\item $\Theta_i$ and $\Theta_j$ have no faces in common, for each $i, j =1, \ldots, a$, $i \neq j$,
\item $\bigcup_{i=1}^{a} \Theta_i = \Delta$.
\end{itemize}
Then there exist $\Theta_i, \Theta_j$, which share an edge and
$$
w_i+w_j \leq \frac{w}{2}.
$$
\end{lemma}
\begin{proof}
It suffices to prove the case where each $\Theta_i$ is equal to a single face. Assume without loss of generality $w_1 \leq w_2 \leq \cdots \leq w_a$. Then 
$$
w_1 + w_i \leq \frac{w}{2},
$$
for all $i=2, \ldots, a-1$. We first prove the lemma for $a \geq 5$. Colour the face with label $w_1$ red and those with label $w_2, \ldots, w_{a-1}$ blue. The result follows from Definition \ref{small}. For $a=4$ we need a slightly more refined analysis. \\

\noindent Case 1: $w_1+w_4 \leq \frac{w}{2}$. Then we colour the face with label $w_1$ red and those with label $w_2, w_3, w_4$ blue. The result again follows from Definition \ref{small}. \\

\noindent Case 2: $w_1+w_4 \geq \frac{w}{2}$. Then $w_2+w_3 \leq \frac{w}{2}$. In this case, we colour the faces with label $w_1, w_2, w_3$ blue. The result follows from Definition \ref{small}. 
\end{proof}

\subsection{Applications to projective 3-space} \label{P3}

In this section, we calculate $\sfZ_{X,H,c_1,c_2}(q)$ for $X = \PP^3$ and $c_2=1,2$. We also classify the reflexive hulls for $c_2=3$. The generating function does not depend on $H$ and we drop $X$ and $H$ from the notation. We take $c_1=-1$ and also suppress it from the notation. Note that rank and degree are coprime and there are no strictly $\mu$-semistable sheaves. We use the notation of Example \ref{P3reflch}.

\begin{example} \label{c2=1}
Let $c_2=1$. Using the methods of Section \ref{fixedloci}, one can see that $c_2=1$ is the minimal value for which there exist rank 2 $\mu$-stable torsion sheaves on $\PP^3$ with Chern classes $c_1$, $c_2$ (see also \cite{GK1}). There are exactly four such sheaves which are $T$-equivariant and reflexive. Each of them has a length 1 singularity at one of the torus fixed points and no further singularities. The toric data $({\bf{u}},{\bf{v}},{\bf{p}})$ of these reflexive hulls is described as follows:
\begin{enumerate}
\item[(i)] For any choice $i,j,k,l \in \{1,2,3,4\}$ with $i,j,k,l$ mutually distinct
\begin{align*}
&u_1=u_2=u_3=0, \ u_4 = -1, \\
&v_i = 0, \ v_j = v_k =  v_l = 1, \\
&p_{j}, \ p_{k}, \ p_{l} \ \mathrm{mutually \ distinct}.
\end{align*}
Recall that $u_2=u_3=u_4=0$ is forced by equation (\ref{slice}). Note that for fixed $\bf{u}$ and $\bf{v}$, different choices of $p_j, p_k, p_l$ give $T$-equivariantly isomorphic reflexive hulls. 
\end{enumerate}
Theorem \ref{Eulermin} gives
\[
\sfZ_{c_2=1}(q) = 4(q^{-1}+q)M(q^{-2})^8. 
\]
We also note that any connected component $\cC$ of any $\M_{\PP^3}(2,-1,1,c_3)^T$ has constant reflexive hulls. This is useful for applications to DT theory in Part II.
\end{example}

\begin{example} \label{c2=2}
Let $c_2=2$. We describe all possible reflexive hulls of elements of
$$
\bigsqcup_{c_3} \M_{\PP^3}(2,-1,2,c_3)^T.
$$ 
Like the previous example, these reflexive hulls are isolated reduced points. However this time legs can appear, i.e.~$c_2$ and $c_{2}'$ may differ. There are three types of reflexive hulls. Their toric data $({\bf{u}},{\bf{v}},{\bf{p}})$ is described as follows:
\begin{enumerate}
\item[(i)] For any choice $i,j,k,l \in \{1,2,3,4\}$ with $i,j,k,l$ mutually distinct
\begin{align*}
&u_1=u_2=u_3=0, \ u_4 = -2 \\
&v_{i} = v_{j} = v_{k} = 1, \ v_{l} = 2, \\ 
&p_{i} = p_{j} \ \mathrm{and} \ p_{i}, \ p_{k}, \ p_{l} \ \mathrm{mutually \ distinct}.
\end{align*}
These reflexive hulls have $c_{2}'=2$.
\item[(ii)] For any choice $i,j,k,l \in \{1,2,3,4\}$ with $i,j,k,l$ mutually distinct
\begin{align*}
&u_1=u_2=u_3=0, \ u_4 = -2, \\
&v_{i} = 0, \ v_{j} = 1, \ v_{k} = v_{l} = 2, \\
&p_{j}, \ p_{k}, \ p_{l} \ \mathrm{mutually \ distinct}.
\end{align*}
These reflexive hulls have $c_{2}'=2$.
\item[(iii)] The reflexive hulls of Example \ref{c2=1}. Recall that these reflexive hulls have $c_{2}'=1$.
\end{enumerate}
Any $T$-equivariant rank 2 $\mu$-stable torsion free sheaves $\F$ on $\PP^3$ with Chern classes $c_1=-1$, $c_2=2$ has one of the above as its reflexive hull. In cases 1 and 2 the cokernel $\F^{**} / \F$ is 0-dimensional, whereas in case 3 the cokernel is 1-dimensional. Using formula \eqref{general}, Proposition \ref{formulachi}, and Theorems \ref{Ben}, \ref{Benwithlegs} one can compute the following formula
\begin{align*}
\sfZ_{c_2=2}(q) &= 12 M(q^{-2})^8 \Bigg (\underbrace{q^4(1+q^{-2}+(q^{-2})^2)^2}_{\mathrm{type \ (i) \ contribution}} +\underbrace{q^4(1+q^{-2}+2(q^{-2})^2+(q^{-2})^3+(q^{-2})^4)}_{\mathrm{type \ (ii) \ contribution}} \\
&\qquad\qquad\qquad\quad+\underbrace{2q^{4} \frac{1+(q^{-2})^2}{(1 -  q^{-2})^2} + q^{4} \frac{1+q^{-2}}{(1-q^{-2})^2}}_{\mathrm{type \ (iii) \ contribution}} \Bigg) \\
&=12 \Bigg( \frac{2q^{-4}-q^{-2}+1-4q^2+3q^4+5q^8}{(1-q^2)^2} \Bigg) M(q^{-2})^{8}. 
\end{align*}
We also note that any connected component $\cC$ of any $\M_{\PP^3}(2,-1,2,c_3)^T$ has constant reflexive hulls.
\end{example}

\begin{example} \label{c2=3}
Let $c_2=3$. We describe the connected components $\cC$ of 
$$
\bigsqcup_{c_{2}',c_{3}'} \N_{\PP^3}(2,-1,c_{2}',c_{3}')^T
$$ 
containing reflexive hulls of elements of 
\begin{equation} \label{eqc2is3}
\bigsqcup_{c_3} \M_{\PP^3}(2,-1,3,c_3)^T.
\end{equation}
This time, $\cC$ no longer needs to be isolated. The toric data $({\bf{u}},{\bf{v}},{\bf{p}})$ of these connected components is classified as follows:
\begin{enumerate}
\item[(i)] For any choice $i,j,k,l \in \{1,2,3,4\}$ with $i,j,k,l$ mutually distinct
\begin{align*}
&u_1=u_2=u_3=0, \ u_4 = -2,  \\
&v_{i} = v_{j} = v_{k} = 1, \ v_{l} = 2, \\ 
&p_{i}, \ p_{j}, \ p_{k}, \ p_{l} \ \mathrm{mutually \ distinct}.
\end{align*}
These reflexive hulls have $c_{2}'=3$ and form non-isolated connected components isomorphic to $\C^* \setminus pt$. See also \cite{GK1}.
\item[(ii)] For any choice $i,j,k,l \in \{1,2,3,4\}$ with $i,j,k,l$ mutually distinct
\begin{align*}
&u_1=u_2=u_3=0, \ u_4 = -3, \\
&v_{i} = 1, \ v_{j} = 2, \ v_{k} = 1, \ v_{l} = 3, \\
&p_{i} = p_{j} \ \mathrm{and} \ p_{i}, \ p_{k}, \ p_{l} \ \mathrm{mutually \ distinct}.
\end{align*}
These reflexive hulls have $c_{2}'=3$ and correspond to isolated reduced points.
\item[(iii)] For any choice $i,j,k,l \in \{1,2,3,4\}$ with $i,j,k,l$ mutually distinct
\begin{align*}
&u_1=u_2=u_3=0, \ u_4 = -3,  \\
&v_{i} = 0, \ v_{j} = 1, \ v_{k} = v_{l} = 3, \\
&p_{j}, \ p_{k}, \ p_{l} \ \mathrm{mutually \ distinct}.
\end{align*}
These reflexive hulls have $c_{2}'=3$ and correspond to isolated reduced points.
\item[(iv)] The reflexive hulls of Example \ref{c2=2}. Recall that these reflexive hulls have $c_{2}'=1$ or $c_{2}'=2$ and correspond to isolated reduced points.
\end{enumerate}
Any $T$-equivariant rank 2 $\mu$-stable torsion free sheaf $\F$ on $\PP^3$ with Chern classes $c_1=-1$, $c_2=3$ has one of the above as its reflexive hull. This time not all connected components of \eqref{eqc2is3} have constant reflexive hulls. Connected components with an element having a reflexive hull appearing in (i) do not, whereas those having a reflexive hull as in (ii) and (iii) do. For (iv) some do whereas others do not.
\end{example}

\hfill
\newline

\noindent {\Large{\textit{{Part II: Virtual}}}} \\
\addcontentsline{toc}{section}{\textit{Part II: Virtual}}

\section{Rank 2 Donaldson-Thomas type invariants}

In this section we define rank 2 DT type invariants on a toric Calabi-Yau 3-fold by $T$- and $T_0$-localization, where $T_0 \subset T$ is the torus preserving the Calabi-Yau volume form. We start with a discussion on the moduli space and $T$-equivariant Serre duality.

\subsection{Moduli space} \label{moduli}

Let $Y$ be a smooth toric Calabi-Yau 3-fold. Recall the notation used for toric varieties from Section \ref{toric3folds}. Then $Y$ is non-compact so has no good notion of moduli space of stable sheaves on it. We therefore want to compactify $Y$.
\begin{definition} \label{toriccompact}
Let $Y$ be a smooth toric Calabi-Yau 3-fold. Let $X$ be a smooth projective toric 3-fold. We say $X$ is a \emph{toric compactification} of $Y$ if $Y \subset X$ is a union of some of the invariant affine open subsets $\{U_\alpha\}_{\alpha \in V(X)}$ of $X$. In this context, we denote by $V(Y)$ the collection of vertices $\alpha$ for which $U_\alpha \subset Y$. Moreover, we write $E_c(Y)$ for the edges 
spanned by vertices of $V(Y)$. The subscript $c$ indicates that we consider edges 
corresponding to \emph{compact} lines in $Y$ only.
\end{definition}
Let $\alpha \in V(Y)$ and use coordinates $(x_1,x_2,x_3)$ on $U_\alpha$ as in Section \ref{toric3folds}. Define the subtorus $T_0 \subset T$ by $t_1 t_2 t_3 = 1$.\footnote{The torus $T_0$ preserves the Calabi-Yau volume form on $Y$.} Equation (\ref{Klocal}) implies that there exists a $T_0$-equivariant isomorphism
\[
K_{X}|_{Y} \cong \O_Y.
\]
Let $\alpha\beta \in E_c(Y)$. Then the degree of $N_{C_{\alpha \beta}/Y}$ is $-2$ by the Calabi-Yau property, i.e.~
\[
m_{\alpha \beta} + m_{\alpha \beta}^{\prime} = -2, 
\]
where $m_{\alpha\beta}$, $m_{\alpha\beta}'$ were defined in Section \ref{toric3folds}.
\begin{example} \label{extoricCY}
Important examples of toric compactifications are 
\begin{align*}
\C^3 &\subset \PP^3, \\
\mathrm{Tot}(\O_{\PP^1}(k) \oplus \O_{\PP^1}(-2-k)) &\subset \PP(\O_{\PP^1}(k) \oplus \O_{\PP^1}(-2-k) \oplus \O_{\PP^1}), \\
\mathrm{Tot}(K_S) &\subset \PP(K_S \oplus \O_S),
\end{align*}
where $\mathrm{Tot}$ denotes the total space of the indicated vector bundle, $k \in \Z$, and $S$ is any smooth projective toric surface. These compactifications satisfy $H^0(K_{X}^{-1}) \neq 0$.
\end{example}

Let $X$ be a smooth projective $3$-fold over $\C$ with polarization $H$. We denote by $\M_X := \M_{X}^{H}(r,c_1,c_2,c_3)$ the moduli space of $\mu$-stable torsion free sheaves on $X$ of rank $r$ and Chern classes $c_1, c_2, c_3$. As mentioned in the introduction, we do not consider strictly $\mu$-semistable sheaves but work directly on the open subset of $\mu$-stable sheaves.
Suppose $X$ is a toric compactification of a toric Calabi-Yau 3-fold $Y$ (Definition \ref{toriccompact}). We define a moduli space of $\mu$-stable sheaves on $Y$ using the compactification $X$ as follows 
\[
\M_{Y \subset X}:= \M_{Y \subset X}^{H}(r,c_1, c_2, c_3):=\{ [\F] \in \M_X \ | \  \F^{**} / \F \ \mathrm{has \ support \ in \ } Y \}.
\]
\begin{proposition} \label{openness}
$\M_{Y \subset X}^{H}(r,c_1, c_2, c_3) \subset \M_{X}^{H}(r,c_1, c_2, c_3)$ is open and $T$-invariant.
\end{proposition}
\begin{proof}
We first prove $\M_{Y\subset X}$ is $T$-invariant. Let $[\F] \in \M_{Y\subset X}$ and $t\in T$ be closed points. Then $t$ induces an automorphism $t : X \rightarrow X$. By definition, $\F$ fits into the short exact sequence $$0\longrightarrow \F \longrightarrow \F^{**}\longrightarrow \cQ \longrightarrow 0,$$ where $\text{Supp}(\cQ) \subset Y$. This sequence remains exact after applying $t^*$. Since $t^*$ commutes with taking (double) dual, we have 
$$
(t^*\F)^{**} / t^* \F \cong t^* \cQ \ \mathrm{and} \ \text{Supp}(t^*\cQ) = t(\text{Supp}(\cQ)).
$$
Since $\mathrm{Supp}(\cQ) \subset Y$ and $Y$ is $T$-invariant, the result follows.

Next, we prove $\M_{Y\subset X} $ is open in $\M_X$ by proving that its complement is closed. Suppose $C$ is a smooth quasi-projective curve and let $0$ be a point on $C$. Suppose we have a morphism $C \rightarrow \M_X$ such $C \setminus 0$ maps to $\M_X \setminus \M_{Y \subset X}$. It suffices to show $0$ also maps to $\M_X \setminus \M_{Y \subset X}$. Let $\F$ be the family induced by this map. We know $\text{Supp}((\F_t)^{**}/\F_t)$ intersects $X \setminus Y$ for all closed points $t \in C \setminus 0$ and we want to show $\text{Supp}((\F_0)^{**}/\F_0)$  intersects $X \setminus Y$. 

After possibly removing a finite number of points from $C \setminus 0$, the reflexive hulls $(\F_t)^{**}$, $t \in C \setminus 0$ form a flat family \cite{Kol}. We can take the flat limit of this family inside the (proper) moduli space of Gieseker semistable torsion free sheaves containing $(\F_t)^{**}$ with $t \in C \setminus 0$ to fill in the missing fiber over $0$. This gives a flat family $\ccR$ over $C$ such that $\ccR_t=(\F_t)^{**}$ for all closed points $t \in C \setminus 0$, but $\ccR_0$ may not be equal to $(\F_0)^{**}$ (because reflexive hulls of members of a flat family may jump). In particular, $\ccR_0$ may not be reflexive. A priori, $\ccR_0$ is only Gieseker semistable and hence $\mu$-semistable.   
 
Over $C \setminus 0$ we also have the flat family of quotients given by the short exact sequence 
\begin{equation} \label{Rg}
0\longrightarrow \F_t \longrightarrow \ccR_t \longrightarrow \cQ_t \to 0.
\end{equation}
We take the flat limit of these quotients inside the corresponding relative Quot scheme 
\begin{equation} \label{proper}
\mathrm{Quot}_{X \times C/C}(\ccR,P) \longrightarrow C,
\end{equation}
where $P$ is the Hilbert polynomial of $\ccR_0$ with respect to a polarization on $X$. The flat limit exists, because the map \eqref{proper} is proper. We obtain a short exact sequence $$0\longrightarrow \F' \longrightarrow \ccR \longrightarrow \cQ \to 0$$ extending \eqref{Rg}. We claim that $\F'_0 \cong \F_0$. Since $\F'_0$ and $\ccR_0$ only differ in codimension $\geq 2$, $\F'_0$ is $\mu$-semistable. There exists a projective moduli scheme of $\mu$-semistable sheaves $\M^{\mu ss}_{X}$ constructed by D.~Greb and M.~Toma \cite{GT}. Seperatedness of $\M^{\mu ss}_{X}$, implies equality of the equivalence classes of $\F_0$, $\F'_0$ in $\M^{\mu ss}_{X}$. This translates into $$[gr^{\mu}(\F_0)]^{**} \cong [gr^{\mu}(\F'_0)]^{**},$$ where $gr^{\mu}(\cdot)$ denotes the graded object associated to a choice of $\mu$-Jordan-H\"older filtration and $[\cdot]^{**}$ denotes double dual \cite[Thm.~5.5]{GT}.\footnote{For any $\mu$-semistable sheaf $\F$ and choice of $\mu$-Jordan-H\"older filtration of $\F$, the graded object $gr^{\mu}(\F)$ depends on the choice of filtration. However its double dual $[gr^{\mu}(\F)]^{**}$ is independent of this choice (up to isomorphism) \cite[Cor.~1.6.10]{HL}.} Since $\F_0$ is $\mu$-stable this equality implies $\F'_0$ cannot be strictly $\mu$-semistable. Since $\F'_0$ is $\mu$-stable, separateness of $\M_{X}$ implies $\F_0 \cong \F'_0$.

Although $\ccR_0$ and $(\F_0)^{**}$ may not be equal, we claim that $(\ccR_0)^{**} \cong (\F_0)^{**}$. We proved $\F_0$ injects into $\ccR_0$ with cokernel $\cQ_0$, so they are isomorphic on the complement of $\text{Supp}(\cQ_0)$. Since $\text{Supp}(\cQ_0)$ has codimension $\geq 2$, their reflexive hulls are completely determined on $X\setminus \text{Supp}(\cQ_0)$ and $(\ccR_0)^{**} \cong (\F_0)^{**}$ \cite[Prop.~1.6]{Har2}. From the chain of the inclusions $\F_0\to \ccR_0\to (\F_0)^{**}$, we deduce
$$
\text{Supp}(\ccR_0/\F_0) \subset \text{Supp}((\F_0)^{**}/\F_0).
$$
So if $\text{Supp}(\ccR_0/\F_0)$ intersects $X \setminus Y$, we are done. For this, it suffices to show that the locus of quotients in $\text{Quot}_{X \times C /C}(\ccR,P)$ whose support intersect $X\setminus Y$ forms a closed subset. 

Let $\G$ be any $S$-flat family of torsion free sheaves on $X\times S$ for arbitrary base $S$. Suppose the Hilbert polynomial of the members of the family is $P$ with respect to some polarization on $X$. Consider the corresponding relative Quot scheme $$Q:=\text{Quot}_{X \times S/S}(\G,P).$$ Let $\cQ$ be the universal quotient over $X\times Q$. Then both $\text{Supp}(\cQ)$ and $(X\setminus Y) \times Q$ are closed subsets of $X \times Q$ and so is their intersection $Z$. Since $X$ is proper, the image of $Z$ under projection to $Q$ gives a closed subset.
\end{proof}

In the case $r=1$, the moduli space $\M_{X}^{H}(1,c_1,c_2,c_3)$ is independent of the polarization $H$. Moreover, it is isomorphic to a Hilbert scheme $I_\chi (X,\beta)$ of ideal sheaves $I_C \subset \O_X$ of closed subschemes $C$ of dimension $\leq 1$ with $[C]=\beta$ and $\chi(\O_C) = \chi$. Here $\beta$ is Poincar\'e dual to $c_2$ and $\chi$ is determined by $c_1$, $c_2$, $c_3$. In this case, $\M_{Y \subset X}$ is the open subset of closed subschemes $C$ with support in $Y$. From this description, it is clear that $\M_{Y \subset X}$ is \emph{independent} of the choice of toric compactification $X$. There is no such independence in the case $r>1$. In general, not only does $\M_{Y \subset X}$ depend on the choice of toric compactification $X$, it also depends on the choice of polarization $H$ on $X$.\footnote{In fact, for certain choices of $H$ the moduli space $\M_X$ can be empty. For examples, see \cite{GK1}.} We think of $\M_{Y \subset X}$ as a moduli space of rank $r$ $\mu$-stable torsion free sheaves on $Y$.

\subsection{Serre duality}

On $\M_{Y \subset X}$, we have a type of ``Serre duality in the $K$-group'' as we will see in Proposition \ref{SD} below. Later we define DT type invariants of $Y$ by virtual localization on $\M_{Y \subset X}$. Roughly speaking, Proposition \ref{SD} states that at the $K$-group level, DT theory on $\M_{Y \subset X}$ is symmetric. This will allow us to specialize the equivariant parameters $s_1+s_2+s_3 = 0$ in Proposition-Definition \ref{propdef}. We start with some notation and a technical lemma.
\begin{definition} \label{baroper}
Let $X$ be a smooth projective variety, $\FF$ a $B$-flat family of coherent sheaves on $X$ over a base scheme $B$, and $\L$ a line bundle on $X \times B$. We use the common notation
\begin{align*}
R\hom_{p_{B}}(\FF,\FF \otimes \L) := Rp_{B*} R\hom(\FF,\FF \otimes \L).
\end{align*}
Moreover, we denote the corresponding $K$-group class by
\begin{align*}
\langle \FF, \FF \otimes \L \rangle &:= [R\hom_{p_{B}}(\FF,\FF \otimes \L)] \in K_0(B).
\end{align*}
The corresponding trace free parts are denoted by $R\hom_{p_{B}}(\FF,\FF \otimes \L)_0$, $\langle \FF, \FF \otimes \L \rangle_0$. In the $K$-group, we often omit square brackets $[\cdot]$ 
for brevity. Note that in the case $\FF = p_{X}^{*} \F$ and $\L = p_{X}^{*} L$, we have 
$$
\langle p_{X}^{*} \F, p_{X}^{*} \F \otimes p_{X}^{*} L \rangle_0 = \langle \F ,\F \otimes L \rangle_0 \otimes \O_B.
$$ 
On $K_0(B)$, we have an involution $\overline{(\cdot)}$, induced by taking the \emph{derived dual} $(\cdot)^\vee$. If $X$ is toric and $\FF$ is $T$-equivariant, then $\langle \FF, \FF \otimes \L \rangle$ is an element of the $T$-equivariant $K$-group $K_{0}^{T}(B)$ where we endow $B$ with the trivial $T$-action. Since $(\cdot)^{\vee\vee} \cong (\cdot)$, the operation $\overline{(\cdot)}$ is an involution on $K_{0}^{T}(B)$. For any coherent sheaf $\F$ on $B$, we have
\[
\overline{\F \otimes t_{1}^{w_1} t_{2}^{w_2} t_{3}^{w_3}} = \F^\vee \otimes t_{1}^{-w_1} t_{2}^{-w_2} t_{3}^{-w_3} \in K_{0}^{T}(B). 
\]
\end{definition}

\begin{lemma} \label{hulls}
Let $X$ be a smooth projective variety with polarization $H$ and let $\M_X = \M_{X}^{H}(r,c_1, \ldots, c_n)$ be the moduli space of $\mu$-stable torsion free sheaves on $X$ with indicated Chern classes. Let $\FF$ be a $B$-flat family of torsion free sheaves in $\M_{X}$ over a variety $B$, such that all members $\FF_b$, $b \in B_{\mathrm{cl}}$ have the same reflexive hull $\ccR$. Then there exists a line bundle $L$ on $B$ and a short exact sequence
$$
0 \longrightarrow \FF \longrightarrow \ccR \boxtimes L \longrightarrow \Q \longrightarrow 0,
$$
where $\Q$ is $B$-flat. For each $b \in B_{\mathrm{cl}}$, this induces a short exact sequence
$$
0 \longrightarrow \FF_b \longrightarrow \ccR \longrightarrow \Q_b \longrightarrow 0,
$$
where the first map is the natural inclusion $\FF_b \hookrightarrow (\FF_b)^{**} \cong \ccR$.
\end{lemma}
\begin{proof}
For $\FF$ \emph{any} $B$-flat family, Koll\'ar \cite{Kol} constructs a stratification $S_i$ of $B$ such that over each stratum $B_i$, $\FF|_{X \times B_i}$ has a $B_i$-flat hull
$$
\FF|_{X \times B_i} \rightarrow \GG_i.
$$ 
Here $\GG_i$ is a $B_i$-flat quasi-coherent sheaf and the morphism restricts to the natural inclusion $\FF_b \hookrightarrow (\FF_{b})^{**}$ for all $b \in B_{\mathrm{cl}}$. By the proof of [Kol, Thm.~21], the stratification only depends on the Hilbert polynomials of $(\mathbb{F}_b)^{**}$ for all $b\in B_{\mathrm{cl}}$. In our case, all members of $\FF$ have isomorphic reflexive hulls, so Koll\'ar's stratification is trivial and we denote the hull by $\phi: \FF \to \GG$. Then $\GG$ is flat over $B$ by the triviality of the stratification and $\phi_b:\FF_b\to \GG_b \cong (\FF_b)^{**}$ is the natural inclusion into the reflexive hull for all $b\in B_{\mathrm{cl}}$. We claim $\phi : \FF \rightarrow \GG$ is injective with $B$-flat cokernel $\Q$. Consider the exact sequence defined by taking kernel and cokernel
$$
0 \longrightarrow \mathbb{K} \longrightarrow \FF \stackrel{\phi}{\longrightarrow} \GG \longrightarrow \Q \longrightarrow 0.
$$ 
Fibres over $b \in B_{\mathrm{cl}}$ give short exact sequences
$$
0 \longrightarrow \FF_b \longrightarrow (\FF_{b})^{**} \longrightarrow \Q_b \longrightarrow 0,
$$
where the first map is the natural inclusion. Since $\FF$ is $B$-flat and $(\FF_{b})^{**} \cong \ccR$ by assumption, we conclude that the Hilbert polynomials of $\Q_b$ are constant. Hence $\Q$ is $B$-flat. This implies $\mathbb{K}$ is $B$-flat. Therefore $\mathbb{K}_b \cong 0$ for all closed points $b \in B$, so $p_{B*} \mathbb{K} \cong 0$ and $\mathbb{K} \cong 0$.

Finally, we show $\GG \cong \ccR \boxtimes L$. Since $\FF_b$ is $\mu$-stable, its reflexive hull $\ccR$ is $\mu$-stable.
Therefore, $\ccR$ is an element of a moduli space of $\mu$-stable reflexive sheaves $\N_X$ on $X$. The $B$-flat family $\GG$ gives a morphism
$$
f : B \longrightarrow \N_X.
$$
Suppose $\N_X$ has a universal family $\R$. Then 
$$\GG \cong (f \times 1_X)^* \R \otimes p_{B}^{*}L,$$ for some line bundle $L$ on $B$. Since $f$ factors through the closed point $b=\{[\ccR]\}$ the result follows from $\R|_b \cong \ccR$. In general, $\N_X$ may not have a universal family. In this case, an \'etale cover of $\N_X$ does have a universal family \cite[4.D.VI]{HL}. Working on the \'etale cover, the same result can be obtained. 
\end{proof}

\begin{proposition} \label{SD}
Let $Y$ be a smooth toric Calabi-Yau 3-fold and $X$ be a toric compactification of $Y$. Let $H$ be a polarization on $X$ and $\M_{Y \subset X} := \M_{Y \subset X}^{H}(r,c_1, c_2, c_3)$, $\M_X := \M_{X}^{H}(r,c_1, c_2, c_3)$. Let $\FF$ be a $B$-flat family of torsion free sheaves in $\M_X$ over a variety $B$. Assume all members $\FF_b$, $b \in B_{\mathrm{cl}}$ lie in $\M_{Y \subset X}$ and have the same reflexive hull $\ccR$. Then
$$
\langle \ccR, \ccR \rangle_0 \otimes \O_B-\langle\FF,\FF \rangle_0 = - \overline{\big( \langle \ccR, \ccR \rangle_0 \otimes \O_B-\langle \FF,\FF \rangle_0 \big)} \in K_0(B). 
$$
If $\FF$ is a $T_0$-equivariant family, then the same equality holds in $K_{0}^{T_0}(B)$. If $\FF$ is a $T$-equivariant family, then
$$
\langle \ccR, \ccR \rangle_0 \otimes \O_B -\langle \FF,\FF \rangle_0 = - \overline{\big( \langle \ccR, \ccR \rangle_0 \otimes \O_B - \langle \FF,\FF \rangle_0 \big)} \otimes (t_{1} t_{2} t_{3})^{-1} \in K_{0}^{T}(B). 
$$
\end{proposition}
\begin{proof}
Consider the short exact sequence of Lemma \ref{hulls}
\begin{equation} \label{FRQfam}
0 \longrightarrow \FF \longrightarrow \ccR \boxtimes L \longrightarrow \Q \longrightarrow 0.
\end{equation}
From this short exact sequence, we get exact triangles
\begin{align}
R\hom_{p_{B}}(\FF,\FF) &\longrightarrow R\hom_{p_{B}}(\FF,\ccR \boxtimes L) \longrightarrow R\hom_{p_{B}}(\FF,\Q), \label{sesA} \\
R\hom_{p_B}(\Q,\ccR \boxtimes L ) &\longrightarrow R\Hom(\ccR,\ccR) \otimes \O_B \longrightarrow R\hom_{p_B}(\FF, \ccR \boxtimes L). \label{sesB}
\end{align}
 The natural map $$\O_X \rightarrow \hom(\ccR,\ccR)$$ induces 
$$
R\Gamma(\O_X) \longrightarrow R \Hom(\ccR,\ccR) \longrightarrow R \Hom(\ccR,\ccR)_0.
$$
We define
$$
A^{\mdot} := \mathrm{Cone}\{R\Gamma(\O_X) \otimes \O_B \rightarrow R\hom_{p_B}(\FF, \ccR \boxtimes L)\},
$$
where the map is the natural composition 
$$
R\Gamma(\O_X) \otimes \O_B \longrightarrow R\Hom(\ccR,\ccR) \otimes \O_B 
\longrightarrow R\hom_{p_B}(\FF,\ccR \boxtimes L).
$$
This leads to the following equalities in the $K$-group
\begin{align*}
\langle \FF,\FF \rangle_0 &= [A^{\mdot}] - \langle \FF,\Q \rangle \\
&= \langle \ccR, \ccR \rangle_0 \otimes \O_B - \langle \Q, \ccR \boxtimes L \rangle - \langle \FF,\Q \rangle.
\end{align*}
The first equality follows from \eqref{sesA} and the definition of $A^\mdot$. The second equality follows from \eqref{sesB} and the definition of $A^\mdot$. Repeating the same reasoning on \eqref{FRQfam} tensored by $p_{X}^{*} K_X$ gives
\begin{align*}
\langle \FF,\FF \otimes p_{X}^{*} K_X \rangle_0  =& \langle \ccR,\ccR \otimes K_X \rangle_0 \otimes \O_B - \langle \Q, \ccR \boxtimes L \boxtimes K_X \rangle - \langle \FF,\Q \otimes p_{X}^{*} K_X \rangle \\
=& \langle \ccR,\ccR \otimes K_X \rangle_0 \otimes \O_B - \langle \Q,\ccR \boxtimes L \rangle \otimes t_1t_2t_3 - \langle \FF,\Q \rangle \otimes t_1t_2t_3,
\end{align*}
where we used that $\Q$ is supported in $Y \times B \subset X \times B$. We deduce
\begin{align*}
&\langle \FF,\FF \rangle_0 - \langle \FF,\FF \otimes p_{X}^{*} K_X \rangle_0 \otimes (t_1t_2t_3)^{-1} \\
&= \langle \ccR,\ccR \rangle_0 \otimes \O_B - \langle \ccR,\ccR \otimes K_X \rangle_0 \otimes \O_B \otimes (t_1t_2t_3)^{-1}.
\end{align*}
(Equivariant) Serre duality gives the results.
\end{proof}

\subsection{Donaldson-Thomas invariants} \label{DTinv}

When $X$ is a smooth projective 3-fold with $H^0(K_{X}^{-1}) \neq 0$ (e.g.~Calabi-Yau or Fano), the moduli space $\M_X$ carries a natural perfect obstruction theory constructed in \cite{Tho}
\begin{equation} \label{DTcx}
E^{\mdot} \rightarrow \LL_{\M_X},
\end{equation}
where $E^{\mdot}$ is a 2-term complex of locally free sheaves on $\M_X$ and $\LL_{\M_X}$ is the truncated cotangent complex of $\M_X$. Over a closed point $[\F] \in \M_X$, we have
\[
E^{\mdot \vee}|_{[\F]} \cong R\Hom(\F,\F)_0[1].
\]
If $X$ is Calabi-Yau, then the virtual dimension $\vd$ of this perfect obstruction theory is 0. If $X$ is toric, the above perfect obstruction theory is $T$-equivariant. A smooth projective 3-fold $X$ cannot be Calabi-Yau and toric at the same time. 

If $\M_X$ is compact, then the perfect obstruction theory $E^{\mdot}$ gives rise to a virtual cycle $[\M_X]^{\vir}$. In the case $X$ is Calabi-Yau, the Donaldson-Thomas invariants are defined as 
\[
\deg([\M_{X}]^{\vir}) \in \Z.
\]
For other 3-folds, we can take any $\alpha \in H^{*}(\M_X,\Q)$ of degree $\vd$ and consider
\[
\int_{[\M_{X}]^{\vir}} \alpha  \in \Z.
\]
When $X$ is toric, $\alpha$ can be chosen in $H^{*}_{T}(\M_X,\Q)$ and the invariant lies in 
\[
H_{T}^{*}(pt,\Q) \cong \Q[s_1,s_2,s_3].
\]
In this case, one can use the virtual torus localization formula of T.~Graber and R.~Pandharipande to compute the invariant as explained in \cite{GP}. This works as follows. The $T$-fixed part of $E^\mdot$ induces a perfect obstruction theory on $\M_{X}^{T}$. Moreover, the virtual normal bundle is defined as
\[
N^\vir := E^{\mdot \vee}|_{\M_{X}^{T}}^{m},
\]
where $(\cdot)^m$ means taking the moving part \cite{GP}. The complex $N^\vir \cong \{W_0 \rightarrow W_1\}$ is a $T$-equivariant 2-term perfect complex. Consider the insertion 
\[
e(N^\vir) := \frac{e(W_0)}{e(W_1)} \in H^{*}(\M_{X}^{T},\Q) \otimes_\Q \Q(s_{1},s_{2},s_{3}),
\]
where $e(\cdot) = c_{\ttop}^{T}(\cdot)$ denotes top $T$-equivariant Chern class. The virtual localization formula states
\[
\int_{[\M_{X}]^{\vir}} \alpha = \int_{[\M_{X}^{T}]^{\vir}} \frac{1}{e(N^\vir)} \ \alpha|_{\M_{X}^{T}}. 
\]
The rank 1 DT theory and its connections to Gromov-Witten theory have been studied in the seminal papers \cite{MNOP1, MNOP2}.

\subsection{Rank 2 DT type invariants by $T$-localization} \label{TequivDT}

Let $X$ be a smooth projective toric 3-fold with $H^0(K_{X}^{-1}) \neq 0$. Consider the moduli space $\M_X$. So far in this section, the rank $r$ was arbitrary. From now on, $r=2$. In Section \ref{fixedloci}, we studied $\M_{X}^{T}$ for $r=2$ and showed it is smooth (Theorem \ref{fixedlocithm}). If $\M_X$ is projective, then  
\begin{align}  \label{E}
\begin{split}
[\M_X]^{\vir} &= \sum_{\cC \subset \M_{X}^{T}} \iota_* \Big( \frac{1}{e(N^\vir)} \cap [\cC]^{\vir} \Big) \\
&= \sum_{\cC \subset \M_{X}^{T}} \iota_* \Big( e(T_{\cC} - E^{\mdot \vee})  \cap [\cC] \Big),
\end{split}
\end{align}
where the sum is over all connected components $\cC \subset \M_{X}^{T}$, $\iota : \cC \subset \M_X$ denotes inclusion, and $T_{\cC}$ denotes the tangent bundle of $\cC$. The first equality is the virtual localization formula and the second equality follows from smoothness of $\cC$ as in \cite[Sect.~4.2]{PT2}. In Section \ref{vertex/edge}, we develop a vertex/edge formalism for this class. In what follows, we do not assume $\M_X$ is projective.

Suppose $Y$ is a smooth toric Calabi-Yau 3-fold with toric compactification $X$ with $H^0(K_{X}^{-1}) \neq 0$. Openness of $\M_{Y\subset X}$ (Proposition \ref{openness}) implies that it inherits the DT perfect obstruction theory from $\M_X$. Unlike the rank 1 case, for each closed point $[\F] \in \M_{Y \subset X}$, the class 
$$
E^{\mdot \vee}|_{[\F]} = - \langle \F,\F \rangle_0
$$
does \emph{not} have rank zero even though $\mathrm{Supp}(\F^{**} / \F) \subset Y$. However, the class
$$
E^{\mdot \vee}|_{[\F]} + \langle \ccR,\ccR \rangle_0,
$$
has rank zero and a convenient symmetry by Proposition \ref{SD}. We therefore want to ``split off'' the part $- \langle \ccR,\ccR \rangle_0$. In order to be able to do this in families, we make the following assumption:

\begin{assumption} \label{as1}
Let $\cC \subset \M_{X}^{T}$ be a connected component such that $\cC \subset \M_{Y \subset X}$. Assume all closed points of $\cC$ have the same (i.e.~isomorphic) reflexive hull $\ccR$.  
\end{assumption}


\begin{remark} \label{unioncomp}
By Remark \ref{a=3}, Assumption \ref{as1} implies $\ccR$ is an isolated reduced point in the $T$-fixed locus (i.e.~$\Ext^1(\ccR,\ccR)^T=0$) and $\cC$ is a product of $\PP^1$'s. 
Conversely, suppose $\ccR$ is a $T$-equivariant rank 2 $\mu$-stable reflexive sheaf on $X$ with $\Ext^1(\ccR,\ccR)^T=0$. Any connected component $\cC \subset \M_{Y \subset X}^{T}$ with all closed points of $\cC$ having reflexive hull $\ccR$ is isomorphic to a product of $\PP^1$'s, hence compact. This fact follows from the toric description of Section \ref{fixedlocigeneral}. 
\end{remark}

\begin{example} \label{ex1as1}
Let $X$ a smooth projective toric 3-fold with $\rk \Pic(X) \leq 3$. Let $H$ be a polarization on $X$ and $c_1, c_2$ Chern classes such that $\gcd(2, c_1 H^2) = 1$ and $c_2 H$ is minimal such that there exist rank 2 $\mu$-stable torsion free sheaves on $X$ with Chern classes $c_1, c_2$. Then for any connected component $\cC$ of $\M_{X}^{T}$, all closed points of $\cC$ have the same reflexive hull (Theorem \ref{Eulermin}). In this case the cokernels $\F^{**} / \F$ of closed points $[\F] \in \cC$ are always 0-dimensional by Proposition \ref{minc2a}. E.g.~Example \ref{c2=1}. 
\end{example}

\begin{example} \label{ex2as1}
In Example \ref{c2=2}, the closed points of $\cC$ have the same reflexive hull for all connected components $\cC$ of $\M_{X}^{T}$. However in this case $\F^{**} / \F$ can be 1-dimensional. In Example \ref{c2=3} some components have constant reflexive hulls whereas others do not.
\end{example}

Let $\cC \subset \M_{X}^{T}$ be a connected component satisfying Assumption \ref{as1}. Recall that $\Hom(\ccR,\ccR)_0 = 0$ and $\Ext^3(\ccR,\ccR)=0$ since $H^0(K_{X}^{-1}) \neq 0$.
Since by assumption $\Ext^1(\ccR,\ccR)^T=0$, the following expression is well-defined 
$$
\frac{1}{e(- \langle \ccR,\ccR \rangle_0 \otimes \O_\cC)} = e(\langle \ccR,\ccR \rangle_0 \otimes \O_\cC).
$$
Therefore we can factor it out
\begin{align*}
\frac{1}{e(N^{\vir})} \cap [\cC]^{\vir} &= e(T_{\cC} - E^{\mdot \vee}) \cap [\cC] \\
&= e(\langle \ccR,\ccR \rangle_0 \otimes \O_\cC)  \cdot e(T_{\cC} - E^{\mdot \vee} - \langle \ccR,\ccR \rangle_0 \otimes \O_\cC) \cap [\cC]. 
\end{align*}
We therefore write somewhat sloppily
$$
\frac{1}{e(N^{\vir}) \ e(\langle \ccR,\ccR \rangle_0 \otimes \O_\cC)} \cap [\cC]^{\vir} = e(T_{\cC} - E^{\mdot \vee} - \langle \ccR,\ccR \rangle_0 \otimes \O_\cC) \cap [\cC].
$$
Here $E^{\mdot \vee} + \langle \ccR,\ccR \rangle_0 \otimes \O_\cC$ has rank zero by Proposition \ref{SD}. 
\begin{propositiondefinition} \label{propdef}
Let $Y$ be a smooth toric Calabi-Yau 3-fold with toric compactification $X$ with $H^0(K_{X}^{-1}) \neq 0$ and polarization $H$. Let $\M_X := \M_{X}^{H}(2,c_1,c_2,c_3)$, $\M_{Y \subset X} := \M_{Y\subset X}^{H}(2,c_1,c_2,c_3)$ and let $\cC \subset \M_{X}^{T}$ be a connected component satisfying Assumption \ref{as1} with reflexive hull $\ccR$. Define
\begin{align} \label{int}
\begin{split}
\DT(\cC) := \int_{[\cC]^{\vir}} \frac{1}{e(N^{\vir}) \ e(\langle \ccR,\ccR \rangle_0 \otimes \O_\cC)}.
\end{split}
\end{align}
The specialization
$$
\DT(\cC) \Big|_{s_1+s_2+s_3=0} \in \Q
$$
is well-defined.  
\end{propositiondefinition}
\begin{proof}
As mentioned above, Assumption \ref{as1} implies $\cC \cong (\PP^1)^N$ for some $N \geq 0$ (Remark \ref{a=3}). By Remark \ref{a=3} there exists a universal family $\FF$ over $\cC$ which fits in a short exact sequence
$$
0 \longrightarrow \FF \longrightarrow p_{X}^{*} \ccR \longrightarrow \Q \longrightarrow 0,
$$
where the cokernel $\Q$ is a $\cC$-flat coherent sheaf and its class in $K_{0}^{T}(\cC)$ can be expressed in terms of the classes $h_i := c_1(\O_{\PP^1}(1))$ of the factors of $\cC = (\PP^1)^N$ and classes of structure sheaves of toric lines and fixed points in $X$ tensored by equivariant line bundles. Using this, we show in Section \ref{vertex/edge} (Theorem \ref{vertexedgeformalism}) that 
$$
E^{\mdot \vee} + \langle \ccR,\ccR \rangle_0 \otimes \O_\cC \in K^{T}_{0}(\cC)
$$
is a finite sum of elements $c t_{1}^{w_1} t_{2}^{w_2} t_{3}^{w_3}$, where each coefficient $c$ is the class of a line bundle (possibly trivial) on a factor $\PP^1$ of $\cC = (\PP^1)^N$.

Consider $E^{\mdot \vee} + \langle \ccR,\ccR \rangle_0 \otimes \O_\cC$. We first consider terms in this expression which are \emph{not} $T_0$-fixed. By Proposition \ref{SD} such terms come in pairs
\[
c t_1^{w_1} t_2^{w_2} t_3^{w_3} - c^* t_1^{-w_1-1} t_2^{-w_2-1} t_3^{-w_3-1},
\]
where $c \in \Pic(\PP^1)$ for some factor $\PP^1$ of $\cC = (\PP^1)^N$. Each such pair contributes
\[
\frac{c_1(c^*) + (-w_1-1) s_1 + (-w_2-1) s_2 + (-w_3-1) s_3}{c_1(c) + w_1 s_1 + w_2 s_2 + w_3 s_3} \Big|_{s_1+s_2+s_3 = 0} = -1,
\]
where we use $c_1(c^*) = -c$.

Next we consider $T_0$-fixed terms of $E^{\mdot \vee} + \langle \ccR,\ccR \rangle_0 \otimes \O_\cC$, which (a priori) could lead to poles when specializing. Each $T_0$-fixed term has a weight of the form $(t_1t_2t_3)^w$. We first consider all terms with $w \in \{0,-1\}$. The $T$-fixed part of $E^{\mdot \vee}$ is the tangent bundle $T_{\cC}$ and its dual $$- T_{\cC}^{*} \otimes (t_1 t_2 t_3)^{-1}$$ has weight $w=-1$. 
Denote all $T$-fixed terms \emph{not} appearing in $T_{\cC}$ by $- a_{k}$, where $a_{k} \in \Pic(\PP^1)$ for some factor $\PP^1$ of $\cC = (\PP^1)^N$ and $k$ runs through some (possibly empty) index set. Note that these terms must come from obstructions and hence have a minus sign. Each such term comes with a dual term $a_{k}^{*} \otimes (t_1 t_2 t_3)^{-1}$, by Proposition \ref{SD}. In conclusion, the contribution of all terms with weight $w \in \{0,-1\}$ to the integral \eqref{int} is
\begin{align} \label{contrI}
\prod_{i=1}^{N}(-2h_i-s) \times \prod_{k} \frac{c_1(a_{k})}{-c_1(a_{k}) -s} = \prod_{i=1}^{N}(-2h_i-s) \times \prod_k \frac{-c_1(a_{k})}{s}, 
\end{align}
where $s:= s_1+s_2+s_3$.

Finally consider a $T_0$-fixed term of $E^{\mdot \vee} + \langle \ccR,\ccR \rangle_0 \otimes \O_\cC$ with $w \notin \{0, -1\}$. Denote the terms occurring with a plus sign by
$$
b_{l} (t_1 t_2 t_3)^{w_{l}},
$$
where $b_{l} \in \Pic(\PP^1)$ for some factor $\PP^1$ of $\cC = (\PP^1)^N$ and $l$ runs through some (possibly empty) index set. By Proposition \ref{SD}, such a term comes with a dual term 
$$
-b_{l}^{*} (t_1 t_2 t_3)^{-w_{l}-1}.
$$
Together their contribution to \eqref{int} is
\begin{align} \label{contrII}
\begin{split}
\prod_l \frac{c_1(b_{l}^{*}) -(w_{l}+1) s}{c_1(b_{l}) + w_{l} s} &= \prod_l \Big[ -\frac{(w_{l}+1) s}{w_{l} s} + \Big(-\frac{1}{w_{l} s} + \frac{(w_{l}+1)s}{(w_{l} s)^2} \Big) c_1(b_{l}) \Big] \\
&= \prod_l \Big[-\frac{(w_{l}+1)}{w_{l}} + \frac{1}{(w_{l})^2 s} c_1(b_{l}) \Big].
\end{split}
\end{align}
Multiplying out \eqref{contrI} and \eqref{contrII} gives the contribution of all $T_0$-fixed terms. We see that no powers of $1/s^{>N}$ can occur since $\cC = (\mathbb{P}^1)^N$ and each factor $1/s$ comes with a class of degree 1. Moreover after multiplying out  \eqref{contrI} and \eqref{contrII}, each non-zero term containing a factor $c_1(a_{k})/s$ with $a_k \in \Pic(\PP^1)$ and $\PP^1$ equal to the $i$th factor of $\cC$ also contains $-2h_i-s$, which cancels the pole. Similarly for $c_1(b_{l})/s$. Here we use that the factors of $\cC$ are 1-dimensional. We conclude that all poles cancel.
\end{proof}

\subsection{Rank 2 DT type invariants by $T_0$-localization} \label{T0equivDT}

As in the previous sections, we consider a smooth toric Calabi-Yau 3-fold $Y$ with toric compactification $X$ with $H^0(K_{X}^{-1}) \neq 0$ and polarization $H$. Let $\M_{Y \subset X} \subset \M_X$ be as before. Consider the subtorus $T_0 \subset T$ defined by $t_1t_2t_3=1$. In \cite[Conj.~2]{PT2}, Pandharipande-Thomas conjecture that the $T_0$-fixed locus of the moduli space of stable pairs is smooth. We conjecture the analog in our setting (see Remark \ref{evidence} for some discussion): 

\begin{conjecture} \label{T0conj1}
Let $X$ be a smooth projective toric 3-fold with polarization $H$ and let $\M_{X} := \M_{X}^{H}(2,c_1,c_2,c_3)$. Then $\M_{X}^{T_0}$ is smooth. 
\end{conjecture}

We consider connected components $\cC \subset \M_{X}^{T_0}$ such that $\cC \subset \M_{Y \subset X}$. We want to endow $\cC$ with a symmetric perfect obstruction theory coming from $T_0$-localization. In the stable pairs setting of \cite{PT2} this part is automatic. However, in our case one must factor out the part coming from reflexive hulls similar to the previous section. In order to achieve this, we make an assumption:
\begin{assumption} \label{as2}
Let $\cC \subset \M_{X}^{T_0}$ be a compact connected component such that $\cC \subset \M_{Y \subset X}$. Assume all closed points of $\cC$ have the same (i.e.~isomorphic) reflexive hull $\ccR$. Moreover, assume $\Ext^1(\ccR,\ccR)^{T_0} = \Ext^2(\ccR,\ccR)^{T_0} = 0$. 
\end{assumption}

\begin{remark} \label{remas2}
Compared to the $T$-equivariant case (Assumption \ref{as1}), there are two more assumptions: $\Ext^1(\ccR,\ccR)^{T_0} = 0$, and $\Ext^2(\ccR,\ccR)^{T_0} = 0$. In the $T$-equivariant case, the first vanishing was a \emph{consequence} and the second vanishing was not needed. We have evidence (from examples) that in this case both vanishings are consequences as well.
\end{remark}

Let $\cC \subset \M_{X}^{T_0}$ satisfy Assumption \ref{as2}. Then $\cC$ has an induced perfect obstruction theory from $T_0$-localization. Assuming Conjecture \ref{T0conj1}, we obtain  
\begin{align*}
 \frac{1}{e(N^{\vir,0})} \cap [\cC]^{\vir} &= e(T_{\cC} - E^{\mdot \vee})  \cap [\cC], \\
&= e(\langle \ccR, \ccR \rangle_0 \otimes \O_{\cC}) \cdot e(T_{\cC} - E^{\mdot \vee} - \langle \ccR,\ccR \rangle_0 \otimes \O_\cC), 
\end{align*}
where
$$
N^{\vir,0} := E^{\mdot \vee}|_{\cC}^{m_0},
$$
denotes the virtual normal bundle and $m_0$ denotes the $T_0$-moving part.

\begin{definition} \label{T0DT}
Let $Y$ be a smooth toric Calabi-Yau 3-fold with toric compactification $X$ with $H^0(K_{X}^{-1}) \neq 0$ and polarization $H$. Suppose Conjecture \ref{T0conj1} holds. Let $\M_X := \M_{X}^{H}(2,c_1,c_2,c_3)$, $\M_{Y \subset X} := \M_{Y\subset X}^{H}(2,c_1,c_2,c_3)$, and let $\cC \subset \M_{X}^{T_0}$ be a connected component satisfying Assumption \ref{as2} with reflexive hull $\ccR$. We define
\begin{align*}
\DT(\cC) := \int_{[\cC]^{\vir}} \frac{1}{e(N^{\vir,0}) \ e(\langle \ccR, \ccR \rangle_0 \otimes \O_{\cC})}, 
\end{align*}
where $e(\cdot)$ denotes the $T_0$-equivariant Euler class. Recall that $\cC$ is a connected component of $\M_{Y \subset X}^{T_0}$ and not $\M_{Y \subset X}^{T}$ as in Proposition-Definition \ref{propdef}. Therefore, we allow ourselves to use the same notation. 
A priori, these invariants have equivariant parameters, but Proposition \ref{T0DT=TDT} below implies that they are numbers.
\end{definition}

The $T_0$-equivariant DT invariants of Definition \ref{T0DT} and the $T$-equivariant DT invariants of Proposition-Definition \ref{propdef} are related in the obvious way.
\begin{proposition} \label{T0DT=TDT}
Let $Y$ be a smooth toric Calabi-Yau 3-fold with toric compactification $X$ satisfying $H^0(K_{X}^{-1}) \neq 0$ and polarization $H$. Suppose Conjecture \ref{T0conj1} holds. Let $\M_X := \M_{X}^{H}(2,c_1,c_2,c_3)$, $\M_{Y \subset X} := \M_{Y\subset X}^{H}(2,c_1,c_2,c_3)$, and let $\cC \subset \M_{X}^{T_0}$ be a connected component satisfying Assumption \ref{as2} with reflexive hull $\ccR$. Then $T/T_0 \cong \C^*$ acts on $\cC$ and the connected components $\cC_i$ of $\cC^{\C^*}$ are connected components of $\M_{Y\subset X}^{T}$. Moreover
$$
\DT(\cC) = \sum_{\cC_i \subset \cC^{\C^*}} \DT(\cC_i) \Big|_{s_1+s_2+s_3 = 0}.
$$
\end{proposition}
\begin{proof}
The formula follows by using the $T/T_0 \cong \C^*$-action on $\cC$ and applying virtual localization to $\DT(\cC)$.
\end{proof}

\begin{theorem} \label{sym}
Let $Y$ be a smooth toric Calabi-Yau 3-fold with toric compactification $X$ with $H^0(K_{X}^{-1}) \neq 0$ and polarization $H$. Suppose Conjecture \ref{T0conj1} holds. Let $\M_X := \M_{X}^{H}(2,c_1,c_2,c_3)$, let $\M_{Y \subset X} := \M_{Y\subset X}^{H}(2,c_1,c_2,c_3)$, and let $\cC \subset \M_{X}^{T_0}$ be a connected component satisfying Assumption \ref{as2} with reflexive hull $\ccR$. Then DT theory induces a symmetric perfect obstruction theory $E^{\mdot \vee}|_{\cC}^{T_0}$ on $\cC$. Moreover 
$$
E^{\mdot \vee} + \langle \ccR,\ccR \rangle_0 \otimes \O_\cC = - \overline{\big( E^{\mdot \vee} + \langle \ccR,\ccR \rangle_0 \otimes \O_\cC \big)} \in K_{0}^{T_0}(\cC).
$$
\end{theorem}
\begin{proof}
Let $\FF$ be the universal family on $X \times \cC$.\footnote{Using our assumptions on $\ccR$ and smoothness of $\cC$, one can show that $\cC$ is scheme-theoretically isomorphic to a connected component of the $T_0$-fixed locus of the Quot scheme $\Quot_X(\ccR)$ of quotients $\ccR \rightarrow \cQ$. The universal sheaf $\FF$ is the kernel of the universal quotient $p_{X}^{*} \ccR \rightarrow \Q$ restricted to $X \times \cC$.} Then $\FF$ has a $T_0$-equivariant structure. By Assumption \ref{as2} and Lemma \ref{hulls}, we have a short exact sequence
\begin{equation} \label{begin}
0 \longrightarrow \FF \longrightarrow \ccR \boxtimes L \longrightarrow \Q \longrightarrow 0.
\end{equation}
The argument starts as in Proposition \ref{SD}. From the above short exact sequence, we obtain exact triangles
\begin{align}
R\hom_{p_\cC}(\FF,\FF)^{T_0} &\longrightarrow R\hom_{p_\cC}(\FF, \ccR \boxtimes L)^{T_0} \longrightarrow R\hom_{p_\cC}(\FF,\Q)^{T_0}, \label{sesC} \\
R\hom_{p_\cC}(\Q,\ccR \boxtimes L)^{T_0} &\longrightarrow (R\Hom(\ccR,\ccR) \otimes \O_\cC)^{T_0} \longrightarrow R\hom_{p_\cC}(\FF, \ccR \boxtimes L)^{T_0}. \label{sesD}
\end{align}
The natural map $$\O_X \rightarrow \hom(\ccR,\ccR)$$ induces an exact triangle
\begin{equation*} 
R\Gamma(\O_X)^{T_0} \longrightarrow R \Hom(\ccR,\ccR)^{T_0} \longrightarrow R \Hom(\ccR,\ccR)_{0}^{T_0}.
\end{equation*}
We define
$$
A^{\mdot} := \mathrm{Cone}\{ (R\Gamma(\O_X) \otimes \O_\cC)^{T_0} \rightarrow R\hom_{p_\cC}(\FF, \ccR \boxtimes L)^{T_0} \},
$$
where the map is the composition 
$$
(R\Gamma(\O_X) \otimes \O_\cC)^{T_0} \longrightarrow (R \Hom(\ccR,\ccR) \otimes \O_\cC)^{T_0} \longrightarrow R\hom_{p_\cC}(\FF,\ccR \boxtimes L)^{T_0}.
$$
From \eqref{sesD} and the definition of $A^\mdot$, we obtain the exact triangle
$$
(R\Hom(\ccR,\ccR)_0 \otimes \O_\cC)^{T_0} \longrightarrow A^{\mdot} \longrightarrow R\hom_{p_\cC} ( \Q,\ccR \boxtimes L )^{T_0}[1].
$$
By Assumption \ref{as2}, $R\Hom(\ccR,\ccR)_{0}^{T_0} \cong 0$, hence
\begin{equation*} 
A^{\mdot} \cong R\hom_{p_\cC} ( \Q,\ccR \boxtimes L )^{T_0}[1].
\end{equation*}
Combining this isomorphism with \eqref{sesC}, gives an exact triangle 
\begin{equation} \label{FF}
R\hom_{p_\cC}(\FF,\FF)^{T_0}_{0} \longrightarrow R\hom_{p_\cC}(\Q,\ccR \boxtimes L)^{T_0}[1] \longrightarrow R\hom_{p_\cC}(\FF,\Q)^{T_0}.
\end{equation}
We can apply a similar reasoning to \eqref{begin} tensored by $p_{X}^{*} K_{X}$.
This gives an exact triangle
\begin{equation*} 
R\hom_{p_\cC}(\FF,\FF \otimes p_{X}^{*} K_X )^{T_0}_{0} \longrightarrow R\hom_{p_\cC}(\Q,\ccR \boxtimes L \boxtimes K_{X})^{T_0}[1] \longrightarrow R\hom_{p_\cC}(\FF,\Q \otimes p_{X}^{*}K_X )^{T_0},
\end{equation*}
where we used
$$
R\Hom( \ccR,\ccR \otimes K_X)^{T_0}_{0} \cong R\Hom( \ccR,\ccR )^{T_0 \vee}_{0} [-3] \cong 0. 
$$
Since $\Q$ is scheme-theoretically supported on $Y \times \cC \subset X \times \cC$ and $Y$ is Calabi-Yau, we have $$\Q \otimes p_{X}^{*} K_X \cong \Q.$$ We obtain the exact triangle
\begin{equation} \label{FFK}
R\hom_{p_\cC}(\FF,\FF \otimes p_{X}^{*} K_X)^{T_0}_{0} \longrightarrow R\hom_{p_\cC}(\Q,\ccR \boxtimes L )^{T_0}[1] \longrightarrow R\hom_{p_\cC}( \FF,\Q )^{T_0}.
\end{equation}
Combining exact triangles \eqref{FF}, \eqref{FFK}, and a non-zero section $\O_X \hookrightarrow K_{X}^{-1}$ gives an isomorphism
$$
R\hom_{p_\cC}(\FF,\FF )^{T_0}_{0} \cong R\hom_{p_\cC}(\FF,\FF \otimes p_{X}^{*} K_X )^{T_0}_{0} \cong R\hom_{p_\cC}(\FF,\FF )_{0}^{T_0 \vee}[-3].
$$
This isomorphism provides the non-degenerate symmetric bilinear form required for a symmetric perfect obstruction theory\footnote{For a symmetric perfect obstruction theory, it is also required that $\theta^{\vee} [1] \cong \theta$, where $\theta : E^{\mdot} \rightarrow L_{M}$ is the map to the cotangent complex \cite{Beh}. In our applications, we only use the isomorphism $E^{\mdot \vee} [1] \cong E^{\mdot}$ (namely for calculations in the $T_0$-equivariant $K$-group). Therefore, we do not prove the duality of $\theta$.}
$$
E^{\mdot}|_{\cC}^{T_0} \cong (E^{\mdot}|_{\cC}^{T_0})^{\vee} [1].
$$

Next, the second part of the theorem. Exactly as in the proof of Proposition \ref{SD}, we obtain the following equalities in $K_{0}^{T_0}(\cC)$
\begin{align*}
\langle \FF,\FF \rangle_0 =& \langle \ccR,\ccR \rangle_0 \otimes \O_{\cC} - \langle \Q,\ccR \boxtimes L \rangle  - \langle \FF,\Q \rangle,  \\
\langle \FF,\FF \otimes p_{X}^{*} K_X \rangle_0 =& \langle \ccR,\ccR \otimes K_X \rangle_0 \otimes \O_\cC - \langle \Q,\ccR \boxtimes L \boxtimes K_X \rangle - \langle \FF,\Q \otimes p_{X}^{*} K_X \rangle \\
=& \langle \ccR,\ccR \otimes K_X \rangle_0 \otimes \O_\cC - \langle \Q,\ccR \boxtimes L \rangle - \langle \FF,\Q \rangle.
\end{align*}
Subtracting these equations and using $T_0$-equivariant Serre duality gives the desired result.
\end{proof}

\section{Vertex/edge formalism} \label{vertex/edge}

Let $X$ be a smooth projective toric 3-fold with $H^0(K_{X}^{-1}) \neq 0$. Let $H$ be a polarization on $X$ and let $$\M_X := \M_{X}^{H}(2,c_1,c_2,c_3).$$ Let $\M_{X}^{H}(\bschi)$ be a connected component of the fixed locus $\M_{X}^{T}$, where $\bschi \in \X_{(2,c_1,c_2,c_3)}^{\slice}$ (Theorem \ref{fixedlocithm}). Recall that the space $\M_{X}^{H}(\bschi)$ is smooth and equal to the following GIT quotient (Theorem \ref{equivmod})
\[
\M_{X}^{H}(\bschi) = U_{X}^{H}(\bschi) \times (\PP^1)^b \times (\PP^1)^c \ \slash \ \SL(2,\C), \ U_{X}^{H}(\bschi) \subset (\PP^1)^a.
\]
We will write $\cC = \M_{X}^{H}(\bschi)$. In this section, we assume $\cC$ satisfies: \\

\noindent \emph{All closed points of $\cC$ have the same (i.e.~isomorphic) reflexive hull $\ccR$.} \\  

\noindent Recall that there are plenty of examples of such $\cC$ (Examples \ref{ex1as1} and \ref{ex2as1}). 
 In applications, $X$ is often a toric compactification of a toric Calabi-Yau 3-fold $Y$ and $\cC \subset \M_{Y \subset X}^{T}$ as in Section \ref{TequivDT}. Assuming $\cC$ has constant reflexive hulls has several nice consequences: 
\begin{itemize}
\item $\cC$ has a universal family $\FF$ (Remark \ref{a=3}). 
\item We have a short exact sequence
\begin{equation*} \label{ses}
0 \longrightarrow \FF \longrightarrow p_{X}^{*} \ccR \longrightarrow \Q \longrightarrow 0,
\end{equation*}
where the cokernel $\Q$ is $\cC$-flat and $\Ext^1(\ccR,\ccR)^{T} = 0$ (Remark \ref{a=3}). \\
\item $\cC$ is a connected component of the fixed locus of the Quot scheme $\Quot_X(\ccR,c_{2}'',c_{3}'')^T$ (Section \ref{3Dpartitions}). Therefore, $\cC$ is isomorphic to a product of $\PP^1$'s.
\end{itemize}

\begin{remark}
In this section we set up a vertex/edge formalism closely following \cite{MNOP1} in the rank 1 case and \cite{PT2} in the stable pair case. Since our fixed loci are non-isolated (products of $\PP^1$'s) our situation is most similar to \cite{PT2}. In order to follow their method, we need the existence of a universal family on $\cC$, which is ensured by assuming $\cC$ has constant reflexive hulls. Unlike the case of \cite{PT2}, we can have edge moduli, i.e.~moduli in the legs (Example \ref{moduliinlegs}).
\end{remark}

We want to compute the $K$-group class of the DT complex (\ref{DTcx}) restricted to $\cC$
\[
E^{\mdot \vee}|_{\cC} \in K_{0}^{T}(\cC).
\]

A torsion free sheaf $[\F] \in \cC_{\mathrm{cl}}$ has a $T$-equivariant structure as explained in Section \ref{fixedloci}. The complex $E^{\mdot \vee}$ restricted $[\F]$ is the virtual tangent space $\T_{[\F]}$ at $[\F]$
\begin{align*}
\T_{[\F]} &= \Ext^1(\F,\F) - \Ext^2(\F,\F) \\
&= - \langle \F,\F \rangle_0.
\end{align*}
Here we used $\Hom(\F,\F)_0 = 0$ and $\Ext^3(\F,\F) \cong \Hom(\F,\F \otimes K_{X}) = 0$, because $\F$ is stable and $H^0(K_{X}^{-1}) \neq 0$. Also note that for toric varieties $h^{>0}(\O_X) = 0$ so tracelessness is automatic. We write this as 
\[
\T_{[\F]} = \langle \ccR,\ccR \rangle - \langle \F,\F \rangle - \langle \ccR,\ccR \rangle_0,
\]
where we keep the trace-free part in the last term or else we would get a unwanted factor of $\Hom(\ccR,\ccR) \cong \C$. The term $\langle \ccR,\ccR \rangle - \langle \F,\F \rangle$ is well-behaved and can be described by a nice vertex/edge formalism. The term $- \langle \ccR,\ccR \rangle_0$ only contributes an overall constant in $\Z[t_{1}^{\pm},t_{2}^{\pm},t_{3}^{\pm}]$. 

Let $\{U_\alpha\}_{\alpha \in V(X)}$ be the $T$-invariant open affine cover of $X$ and denote the double intersections by $U_{\alpha\beta} = U_\alpha \cap U_\beta$. Denote the restriction of any sheaf $\F$ to $U_\alpha$, $U_{\alpha\beta}$ by $\F_\alpha$, $\F_{\alpha\beta}$ respectively. Denote the global section functor by $\Gamma(\cdot)$. Using the local-to-global spectral sequence for Ext groups and computing sheaf cohomology by a \v{C}ech calculation gives
\begin{align*}
\langle \ccR,\ccR \rangle - \langle \F,\F \rangle &=\sum_{\alpha,i} (-1)^i \Big( \Gamma(U_\alpha, \ext^i(\ccR_\alpha,\ccR_\alpha)) - \Gamma(U_\alpha, \ext^i(\F_\alpha,\F_\alpha)) \Big) \\
&-\sum_{\alpha\beta,i} (-1)^i \Big( \Gamma(U_{\alpha\beta}, \ext^i(\ccR_{\alpha\beta},\ccR_{\alpha\beta})) - \Gamma(U_{\alpha\beta}, \ext^i(\F_{\alpha\beta},\F_{\alpha\beta})) \Big),
\end{align*}
where contributions of triple and higher intersections vanish because $\F$ and $\ccR$ only differ on a $T$-invariant closed subscheme of dimension $\leq 1$. Our goal is to calculate
\begin{align*}
\langle \ccR_\alpha,\ccR_\alpha \rangle - \langle \F_\alpha,\F_\alpha \rangle &= \sum_i (-1)^i \Big(\Gamma(U_\alpha, \ext^i(\ccR_\alpha,\ccR_\alpha)) - \Gamma(U_\alpha, \ext^i(\F_\alpha,\F_\alpha)) \Big), \\
\langle \ccR_{\alpha\beta},\ccR_{\alpha\beta} \rangle - \langle \F_{\alpha\beta},\F_{\alpha\beta} \rangle &= \sum_i  (-1)^i \Big( \Gamma(U_{\alpha\beta}, \ext^i(\ccR_{\alpha\beta},\ccR_{\alpha\beta})) - \Gamma(U_{\alpha\beta}, \ext^i(\F_{\alpha\beta},\F_{\alpha\beta})) \Big).
\end{align*}

\subsection{Vertex calculation}

This section calculates $$\langle \ccR_\alpha,\ccR_\alpha \rangle - \langle \F_\alpha,\F_\alpha \rangle.$$ Suppose the action of $T$ on $U_\alpha \cong \C^3$ is in standard form $(t_1,t_2,t_3) \cdot x_i = t_i x_i$ for all $i=1,2,3$. Denote the Poincar\'e polynomial of $\F_\alpha$ by 
\[
P(\F_\alpha) \in K_{0}^{T}(U_\alpha) \cong \Z[t_{1}^{\pm}, t_{2}^{\pm}, t_{3}^{\pm}].
\]
Consider the trace map 
\[
\tr : K_{0}^{T}(U_\alpha) \longrightarrow \Z(\!(t_{1}, t_{2}, t_{3})\!),
\]
which associates to a $T$-equivariant vector bundle its character (i.e.~decomposition into weight spaces). Following the reasoning in \cite[Sect.~4.4]{PT2}, we obtain 
\[
\tr_{\langle \F_\alpha,\F_\alpha \rangle} = \frac{P(\F_\alpha) \overline{P}(\F_\alpha)}{(1-t_1)(1-t_2)(1-t_3)},
\]
where the operation $\overline{(\cdot)}$ was introduced in Definition \ref{baroper}. 

We need a family version of the above. This is easily done replacing $\F$ by $\FF$, $\ccR$ by $p^{*}_{X} \ccR$, $U_\alpha$ by $U_\alpha \times \cC$, and $U_{\alpha\beta}$ by $U_{\alpha\beta} \times \cC$. This gives
\[
\tr_{\langle \FF_\alpha,\FF_\alpha \rangle} = \frac{P(\FF_\alpha) \overline{P}(\FF_\alpha)}{(1-t_1)(1-t_2)(1-t_3)} \in K_0(\cC) \otimes_{\Z} \Z(\!(t_{1}, t_{2}, t_{3})\!).
\]
The Poincar\'e polynomial of $p_{U_\alpha}^{*} \ccR_\alpha$ can be expressed in terms of the equivariant parameters $t_i$ as follows. Let $({\bf{u}},{\bf{v}},{\bf{p}})$ be the toric data associated to $\ccR_\alpha$ (Definition \ref{toricdata}). \\

\noindent \textbf{Case 1}: $\ccR_\alpha$ is not locally free, or equivalently all $v_i>0$ and all $p_i$ are mutually distinct (Proposition \ref{rank2singular}). Then
\[
P(p_{U_\alpha}^{*} \ccR_\alpha) = t_{1}^{u_1} t_{2}^{u_2+v_2} t_{3}^{u_3+v_3} + t_{1}^{u_1+v_1} t_{2}^{u_2} t_{3}^{u_3+v_3} + t_{1}^{u_1+v_1} t_{2}^{u_2+v_2} t_{3}^{u_3} - t_{1}^{u_1+v_1} t_{2}^{u_2+v_2} t_{3}^{u_3+v_3}. 
\]

\noindent \textbf{Case 2}: $\ccR_\alpha$ is locally free. In this case $\ccR_\alpha$ has two homogeneous generators and $P(p_{U_\alpha}^{*} \ccR_\alpha)$ is the sum of their characters. \\

The character of $\Q_\alpha$ is more complicated and involves moduli. Recall that $\cC \cong (\PP^1)^N$ for some $N \geq 0$. Each factor $\PP^1$ of $\cC$ has a tautological subbundle $\O_{\PP^1}(-1) \subset \O_{\PP^1} \oplus \O_{\PP^1}$ attached to it. The character $\sfQ_\alpha$ of $\Q_\alpha$ can be expressed in terms of the equivariant parameters $t_i$ and the classes of these $\O_{\PP^1}(1)$'s. Let $\bspi$ correspond to $\chi_\alpha$ by Proposition \ref{lem1}. Then
\[
\sfQ_\alpha = \sum_{(k_1,k_2,k_3) \in \pi_1 \cup \pi_2 \cup \pi_3} \sfQ_{\alpha, (k_1,k_2,k_3)},
\]
where
\begin{itemize}
\item $\sfQ_{\alpha, (k_1,k_2,k_3)} = 2 t_{1}^{k_1} t_{2}^{k_2} t_{3}^{k_3}$ when $(k_1,k_2,k_3) \in \pi_1 \cap \pi_2 \cap \pi_3$, 
\item $\sfQ_{\alpha, (k_1,k_2,k_3)} = t_{1}^{k_1} t_{2}^{k_2} t_{3}^{k_3}$ when $(k_1,k_2,k_3)$ lies in exactly one $\pi_i$,
\item $\sfQ_{\alpha, (k_1,k_2,k_3)} = t_{1}^{k_1} t_{2}^{k_2} t_{3}^{k_3}$ when $(k_1,k_2,k_3)$ lies in exactly two $\pi_{i}$'s but \emph{not} in a region with moduli (see Section \ref{3Dpartitions}),
\item $\sfQ_{\alpha, (k_1,k_2,k_3)} = [\O_C(1)]t_{1}^{k_1} t_{2}^{k_2} t_{3}^{k_3}$ when $(k_1,k_2,k_3)$ lies in exactly two $\pi_{i}$'s and in a region $C \cong \mathbb{P}^1$ with moduli (see Section \ref{3Dpartitions}),
\item $\sfQ_{\alpha, (k_1,k_2,k_3)} = 0$ otherwise.
\end{itemize}
Reasoning analogous to \cite[Sect.~4.4]{PT2} we obtain
\begin{equation} \label{tr1}
\tr_{\langle p_{U_\alpha}^{*} \ccR_\alpha,p_{U_\alpha}^{*} \ccR_\alpha \rangle - \langle \FF_\alpha,\FF_\alpha \rangle} = \sfQ_\alpha \overline{P}(p_{U_\alpha}^{*} \ccR_\alpha) - \frac{\overline{\sfQ}_\alpha P(p_{U_\alpha}^{*}\ccR_\alpha)}{t_1 t_2 t_3} + \sfQ_\alpha \overline{\sfQ}_\alpha \frac{(1-t_1)(1-t_2)(1-t_3)}{t_1 t_2 t_3},
\end{equation}
which lies in 
\[
K_0(\cC) \otimes_{\Z} \Z(\!(t_{1}, t_{2}, t_{3})\!).
\]

\subsection{Edge calculation}

This section calculates $\langle \ccR_{\alpha\beta},\ccR_{\alpha\beta} \rangle - \langle \F_{\alpha\beta},\F_{\alpha\beta} \rangle$. Assume the coordinates are chosen such that the toric line $C_{\alpha \beta} \cap U_\alpha$ is given by $\{x_2 = x_3 = 0\}$. We work over the ring $\Gamma(U_{\alpha\beta}) \cong \C[x_{1}^{\pm 1},x_2,x_3]$ and define the formal $\delta$-function
\[
\delta(t) := \sum_{k \in \Z} t^k. 
\]
Let $\bspi$ correspond to $\chi_\alpha$ by Proposition \ref{lem1}. Denote the limiting double square configuration along the $x_1$-axis arising from $\bspi$ by $\bslambda = (\lambda_1,\lambda_2,\lambda_3)$. Denoting the character of $\Q_{\alpha\beta}$ by $\sfQ_{\alpha\beta}$, we obtain
\[
\sfQ_{\alpha\beta} = \sum_{(k_2,k_3) \in \lambda_1 \cup \lambda_2 \cup \lambda_3} \sfQ_{\alpha\beta,(k_2,k_3)},
\]
where 
\begin{itemize}
\item $\sfQ_{\alpha\beta,(k_2,k_3)} = 2 t_{2}^{k_2} t_{3}^{k_3}$ if $(k_2,k_3) \in \lambda_1 \cap \lambda_2 \cap \lambda_3$,
\item $\sfQ_{\alpha\beta,(k_2,k_3)} = t_{2}^{k_2} t_{3}^{k_3}$ if $(k_2,k_3)$ lies in exactly one $\lambda_i$,
\item $\sfQ_{\alpha\beta,(k_2,k_3)} = t_{2}^{k_2} t_{3}^{k_3}$ if $(k_2,k_3)$ lies in exactly two $\lambda_i$'s but not in a region with moduli,
\item $\sfQ_{\alpha\beta,(k_2,k_3)} = [\O_C(1)]t_{2}^{k_2} t_{3}^{k_3}$ if $(k_2,k_3)$ lies in exactly two $\lambda_i$'s and in a region $C \cong \PP^1$ with moduli,
\item $\sfQ_{\alpha\beta,(k_2,k_3)} = 0$ otherwise.
\end{itemize}
Recall from Example \ref{littlebox} that from the data of $\bslambda$ alone, we cannot determine whether a square lying in exactly two $\lambda_i$'s has moduli or not. This depends on what happens ``in the corners''. This already indicates we cannot split the vertex and edge contributions as easily as in the rank 1 case. We come back to this in the examples of the next section.

A reasoning similar to the previous section gives
\begin{align} \label{Galphabeta}
\begin{split}
&-\tr_{\langle p_{U_{\alpha\beta}}^{*}\ccR_{\alpha\beta}, p_{U_{\alpha\beta}}^{*} \ccR_{\alpha\beta} \rangle - \langle \FF_{\alpha\beta},\FF_{\alpha\beta}\rangle} = \delta(t_1) \sfG_{\alpha\beta}, \\
&\sfG_{\alpha\beta} := - \sfQ_{\alpha\beta} \overline{P}(\ccR_{\alpha\beta}) - \frac{\overline{\sfQ}_{\alpha\beta} P(\ccR_{\alpha\beta})}{t_2 t_3} + \sfQ_{\alpha\beta} \overline{\sfQ}_{\alpha\beta} \frac{(1-t_2)(1-t_3)}{t_2 t_3},
\end{split}
\end{align}
where $P(\ccR_{\alpha\beta})$ denotes the Poincar\'e polynomial of $\ccR_{\alpha\beta}$.

\subsection{Redistribution}

Exactly as in \cite{MNOP1} and \cite{PT2}, the terms of the previous two sections can be redistributed so they become Laurent \emph{polynomials} in $t_1, t_2, t_3$. Denote the vertices neighbouring $\alpha \in V(X)$ by $\beta_1, \beta_2, \beta_3$.  
Define
\begin{align} \label{defVE}
\begin{split}
\sfV_\alpha &:= \tr_{\langle p_{U_\alpha}^{*} \ccR_\alpha,p_{U_\alpha}^{*} \ccR_\alpha \rangle - \langle \FF_\alpha,\FF_\alpha \rangle} + \sum_{i=1}^{3} \frac{\sfG_{\alpha \beta_i}(t_{i'},t_{i''})}{1-t_i}, \\
\sfE_{\alpha\beta} &:= t_{1}^{-1} \frac{\sfG_{\alpha\beta}(t_2,t_3)}{1-t_{1}^{-1}} - \frac{\sfG_{\alpha\beta}(t_2 t_{1}^{-m_{\alpha\beta}},t_3 t_{1}^{-m_{\alpha\beta}'})}{1-t_{1}^{-1}},
\end{split}
\end{align}
where $\{t_1,t_2,t_3\} = \{t_{i},t_{i'},t_{i''}\}$. Here the labelling is such that $C_{\alpha \beta_i} \cap U_\alpha$ is given by $\{x_{i'} = x_{i''} = 0\}$ for all $i,i',i'' \in \{1,2,3\}$ mutually distinct. Recall the definition of $m_{\alpha\beta}$, $m'_{\alpha\beta}$ from Section \ref{toric3folds}. This redistribution does not lead to over-counting due to the choice of the second term in $\sfE_{\alpha\beta}$ (see \eqref{transf}). The reason these are Laurent polynomials is the same as \cite[Lem.~4]{PT2}. We summarize this section in one theorem.
\begin{theorem} \label{vertexedgeformalism}
Let $X$ be a smooth projective toric 3-fold with $H^0(K_{X}^{-1}) \neq 0$ and polarization $H$. Let $\M_X := \M_{X}^{H}(2,c_1,c_2,c_3)$. Suppose $\cC \subset \M_{X}^{T}$ is a connected component of the fixed locus such that all its closed points have reflexive hull $\ccR$. Then the $T$-character of DT theory $\T_{\FF} := E^{\mdot \vee}|_{\cC} \in K_{0}^{T}(\cC)$ is given by 
\begin{equation} \label{trT} 
\tr_{\T_{\FF}} = \tr_{- \langle \ccR,\ccR \rangle_0} + \sum_{\alpha \in V(X)} \sfV_\alpha + \sum_{\alpha\beta \in E(X)} \sfE_{\alpha\beta},
\end{equation}
where $\tr_{- \langle \ccR,\ccR \rangle_0} \in \Z[t_{1}^{\pm 1},t_{2}^{\pm 1},t_{3}^{\pm 1}]$ and for all $\alpha,\beta \in V(X)$ 
\begin{align*}
\sfV_\alpha &\in K_0(\cC) \otimes_{\Z} \Z[t_{1}^{\pm 1},t_{2}^{\pm 1},t_{3}^{\pm 1}], \\
\sfE_{\alpha\beta} &\in K_0(\cC) \otimes_{\Z} \Z[t_{1}^{\pm 1},t_{2}^{\pm 1},t_{3}^{\pm 1}].
\end{align*}
\end{theorem}

\section{Applications to toric Calabi-Yau 3-folds}

Let $Y$ be a smooth toric Calabi-Yau 3-fold and let $X$ be a toric compactification with $H^0(K_{X}^{-1}) \neq 0$ and polarization $H$. We fix $c_1, c_2$ and consider the moduli spaces
\begin{equation} \label{MX}
\M_X := \bigsqcup_{c_3} \M_{X}^{H}(2,c_1,c_2,c_3).
\end{equation}
In Section \ref{moduli}, we introduced the $T$-invariant open subset
$$
\M_{Y \subset X} := \bigsqcup_{c_3} \M_{Y \subset X}^{H}(2,c_1,c_2,c_3)
$$
of torsion free sheaves $\F$ with $\F^{**} / \F$ supported in $Y$. We fix a connected component $\cC \subset \M_{Y \subset X}^{T}$ satisfying Assumption \ref{as1}: \\

\noindent \emph{All closed points of $\cC$ have the same (i.e.~isomorphic) reflexive hull $\ccR$.} \\  


We introduced invariants $\DT(\cC)$ in Proposition-Definition \ref{propdef}. We are interested in the ``Calabi-Yau specialization'' $s_1+s_2+s_3=0$ of these invariants. In this section, we conjecture a formula for generating functions of these invariants in terms of $\sfZ_{Y,\ccR|_Y, \hat{\bslambda}}(q)$ of Remark \ref{openZ}. 

Suppose $[\F] \in \cC$ is a closed point with characteristic function corresponding to double box configurations $\hat{\bspi}$ (Proposition \ref{lem1}), where
$$
\bspi_\alpha \in \Pi(\ccR_{\alpha}, \bslambda_{\alpha\beta_1},  \bslambda_{\alpha\beta_2},  \bslambda_{\alpha\beta_3}),
$$
for each $\alpha \in V(X)$ with neighbouring vertices $\beta_1, \beta_2, \beta_3$. Since the cokernel $\cQ = \ccR / \F$ lies entirely in $Y$, the asymptotic double square configurations $\bslambda_{\alpha\beta}$ are zero unless $\alpha\beta \in E_c(Y)$ (i.e.~$C_{\alpha\beta} \subset Y$) and $\bspi_\alpha = 0$ unless $\alpha \in V(Y)$. Note that $\hat{\bslambda}$ determines the second Chern class and $\hat{\bspi}$ determines the third Chern class (Section \ref{equivsh}, Proposition \ref{formulachi}).

\begin{mainconjecture} \label{mainconj}
Let $Y$ be a smooth toric Calabi-Yau 3-fold with toric compactification $X$ satisfying $H^0(K_{X}^{-1}) \neq 0$. Fix a polarization $H$ on $X$ and Chern classes $c_1$, $c_2$. Let $\ccR$ be a $T$-equivariant rank 2 $\mu$-stable reflexive sheaf on $X$ with toric data $({\bf{u}},{\bf{v}},{\bf{p}})$. Suppose any connected component $\cC$ of $\bigsqcup_{c_3} \M_{Y \subset X}^{H}(2,c_1,c_2,c_3)^T$ has constant reflexive hulls $\ccR$.
For each $\alpha\beta \in E_c(Y)$, there are two faces $\rho_{1, \alpha\beta}$, $\rho_{2,\alpha\beta}$ which share $\alpha\beta$ as an edge and two disjoint faces $\rho_{3,\alpha\beta}$, $\rho_{4,\alpha\beta}$ connected by the edge $\alpha\beta$.  
Then
\begin{align*}
&\sum_{c_3} \sum_{\cC \subset \M_{Y \subset X}^{H}(2,c_1,c_2,c_3)^T} \DT(\cC)\Big|_{s_1+s_2+s_3=0} q^{c_3} = \\
&\sum_{\hat{\bslambda}} \sfZ_{Y,\ccR|_Y, \hat{\bslambda}}(q^{-2}) q^{c_{3}(\ccR)} \prod_{\alpha\beta \in E_c(Y)} (-1)^{|\bslambda_{\alpha\beta}| (m_{\alpha\beta} (v_{\rho_{1,\alpha\beta}} + v_{\rho_{2,\alpha\beta}} + 1) +v_{\rho_{3,\alpha\beta}} + v_{\rho_{4,\alpha\beta}})} q^{|\bslambda_{\alpha\beta}| C_{\alpha\beta} (c_1(X) + c_1)} ,
\end{align*}
where $\cC$ runs over all connected components and $\hat{\bslambda} = \{\bslambda_{\alpha\beta} \in \Lambda(\ccR_{\alpha\beta})\}_{\alpha\beta \in E_c(Y)}$ runs over all double square configurations giving rise to second Chern class $c_2$.
\end{mainconjecture}

The assumptions of this conjecture are satisfied in numerous interesting cases e.g.~Examples \ref{ex1as1} and \ref{ex2as1}.
We provide the following evidence for Main Conjecture \ref{mainconj}:
\begin{itemize}
\item The examples of Sections \ref{nolegs}, \ref{O(-1,-1)}.
\item For connected components $\cC$, ``with expected obstructions'' we compute the invariants $\DT(\cC)|_{s_1+s_2+s_3=0}$ using methods of \cite{MNOP1} (Section \ref{good}). The answers are compatible with Main Conjecture \ref{mainconj}. 
\item Assuming two analogs of conjectures in the stable pairs case \cite[Conj.~2, Conj.~3]{PT2}, we prove Main Conjecture \ref{mainconj} using $T_0$-localization (Section \ref{T0loc}).
\end{itemize}

\begin{remark}
Note that the formula of Conjecture \ref{mainconj} only depends on the compactification through whether $\ccR$ is $\mu$-stable or not and through an overall multiplicative factor of $q^{\cdots}$. In particular, the formula
\begin{align*}
\sum_{\hat{\bslambda}} \sfZ_{Y,\ccR|_Y, \hat{\bslambda}}(q^{-2}) \prod_{\alpha\beta \in E_c(Y)} (-1)^{|\bslambda_{\alpha\beta}| (m_{\alpha\beta} (v_{\rho_{1,\alpha\beta}} + v_{\rho_{2,\alpha\beta}} + 1) +v_{\rho_{3,\alpha\beta}} + v_{\rho_{4,\alpha\beta}})} 
\end{align*}
does not depend on the compactification. We expect this formula to be of the form $M(q)^{2e(Y)}$ times a rational function in $q$, which is invariant under $q \leftrightarrow q^{-1}$ (up to an overall multiplicative factor). We prove this expectation in the case all $\hat{\bslambda} = \varnothing$ (no legs), where the rational function is in fact a Laurent polynomial (Remark \ref{invariance}).
\end{remark}

\subsection{Rank 2 equivariant vertex without legs} \label{nolegs}

In this section, we consider $Y = \C^3$ and $\ccR$ an arbitrary $T$-equivariant rank 2 reflexive sheaf on $\C^3$ described by toric data $({\bf{u}},{\bf{v}},{\bf{p}})$. We will see below that $\ccR$ is always the restriction of a $T$-equivariant rank 2 $\mu$-stable reflexive sheaf on a (polarized) toric compactification $X$ of $\C^3$ with $H^0(K_{X}^{-1}) \neq 0$ (Lemma \ref{(P^1)^3}). Nevertheless, we want to think of the results of this section as being intrinsic to $\C^3$. Since $E_c(\C^3) = \varnothing$, we choose all legs empty. Let $\bspi \in \Pi(\ccR,\varnothing,\varnothing,\varnothing)$. Then $\cC_{\bspi}$ is the moduli space of $T$-equivariant rank 2 torsion free sheaves on $\C^3$ with characteristic function corresponding to $\bspi$ (Proposition \ref{lem1}). Even though we did not start with a compact 3-fold, the vertex expression \eqref{tr1} still makes sense for $\cC_{\bspi}$ and is denoted by $\sfV_{\bspi}$.
\begin{example} \label{normal} Suppose ${\bf{v}}=(1,1,1)$ and $p_1,p_2,p_3$ are mutually distinct. Consider $\bspi$ given by
$$\sfQ_{\bspi}=t_{{1}}t_{{3}}+t_{{2}}t_{{3}}+{t_{{1}}}^{2}t_{{3}}+{t_{{2}}}^{2}t_{{3}
}+t_{{1}}t_{{2}}t_{{3}}+{t_{{1}}}^{3}t_{{3}}+{t_{{1}}}^{2}t_{{2}}t_{{3
}}+{t_{{2}}}^{2}t_{{1}}t_{{3}},$$ 
so $|\bspi|=8$ (see Figure 6). 
Then we have 
\begin{align*}
\sfV_{\bspi}=&t_{{3}}+t_{{1}}t_{{3}}+t_{{2}}t_{{3}}+{\frac {{t_{{1}}}^{3}}{t_{{2}}t_
{{3}}}}-{\frac {1}{t_{{1}}t_{{2}}}}-3\,{\frac {1}{t_{{1}}t_{{3}}}}-3\,
{\frac {1}{t_{{2}}t_{{3}}}}-{\frac {1}{{t_{{1}}}^{2}t_{{3}}}}+{\frac {
{t_{{1}}}^{3}}{{t_{{2}}}^{3}}}+{\frac {{t_{{2}}}^{2}}{{t_{{1}}}^{4}}}+
{t_{{3}}}^{-1}\\&+3\,{t_{{1}}}^{-1}+3\,{t_{{2}}}^{-1}+{t_{{1}}}^{-2}-{t_{
{2}}}^{-3}-{\frac {t_{{2}}}{{t_{{1}}}^{3}t_{{3}}}}+{\frac {t_{{1}}t_{{
3}}}{t_{{2}}}}-{\frac {t_{{1}}}{t_{{2}}t_{{3}}}}+{\frac {t_{{2}}t_{{3}
}}{t_{{1}}}}-{\frac {{t_{{1}}}^{3}}{{t_{{2}}}^{3}t_{{3}}}}-{\frac {{t_
{{2}}}^{2}}{{t_{{1}}}^{4}t_{{3}}}}\\&+{\frac {{t_{{2}}}^{2}}{t_{{1}}t_{{3
}}}}-{\frac {{t_{{1}}}^{2}}{{t_{{2}}}^{2}t_{{3}}}}+{\frac {{t_{{1}}}^{
2}t_{{3}}}{t_{{2}}}}-{\frac {{t_{{1}}}^{2}}{{t_{{2}}}^{3}t_{{3}}}}-{
\frac {{t_{{2}}}^{2}}{{t_{{1}}}^{3}t_{{3}}}}+2\,{\frac {t_{{1}}t_{{2}}
}{t_{{3}}}}-{t_{{1}}}^{-4}-{\frac {1}{{t_{{2}}}^{2}t_{{1}}}}-{\frac {1
}{{t_{{1}}}^{2}t_{{2}}}}+2\,{\frac {{t_{{1}}}^{2}}{t_{{3}}}}\\&+{\frac {{
t_{{2}}}^{2}}{t_{{3}}}}-2\,{\frac {1}{{t_{{1}}}^{2}{t_{{2}}}^{2}}}+2\,
{\frac {t_{{2}}}{{t_{{1}}}^{2}}}+{\frac {t_{{1}}}{t_{{3}}}}+2\,{\frac 
{t_{{1}}}{{t_{{2}}}^{2}}}+{\frac {t_{{3}}}{t_{{1}}}}+{\frac {t_{{3}}}{
t_{{2}}}}-{\frac {1}{t_{{1}}t_{{2}}{t_{{3}}}^{2}}}-{\frac {1}{{t_{{1}}
}^{2}t_{{2}}{t_{{3}}}^{2}}}-{\frac {1}{t_{{1}}{t_{{2}}}^{2}{t_{{3}}}^{
2}}}\\&-{\frac {1}{t_{{2}}{t_{{3}}}^{2}}}-{\frac {1}{{t_{{2}}}^{2}{t_{{3}
}}^{2}}}-{\frac {1}{{t_{{1}}}^{2}{t_{{3}}}^{2}}}-{\frac {1}{t_{{1}}{t_
{{3}}}^{2}}}+{\frac {{t_{{1}}}^{2}}{{t_{{2}}}^{2}}}+{\frac {{t_{{1}}}^
{2}}{{t_{{2}}}^{3}}}+{\frac {{t_{{2}}}^{2}}{{t_{{1}}}^{3}}}+{\frac {t_
{{2}}}{{t_{{1}}}^{3}}}-2\,{\frac {1}{{t_{{1}}}^{3}t_{{2}}}}-{\frac {1}
{t_{{1}}{t_{{2}}}^{3}}}+{\frac {t_{{2}}}{t_{{3}}}}\\&+{\frac {t_{{1}}}{t_
{{2}}}}-2\,{\frac {t_{{1}}}{{t_{{2}}}^{2}t_{{3}}}}-2\,{\frac {t_{{2}}}
{{t_{{1}}}^{2}t_{{3}}}}-{\frac {1}{{t_{{1}}}^{3}{t_{{3}}}^{2}}}
\end{align*} and hence 
$$\frac{1}{e(\sfV_{\bspi})} \Big|_{s_1+s_2+s_3=0}=1=\omega(\bspi),$$
where $\omega(\bspi)$ is the weight defined in Definition \ref{Z}.
\end{example}

\begin{figure} 
\includegraphics[width=2in]{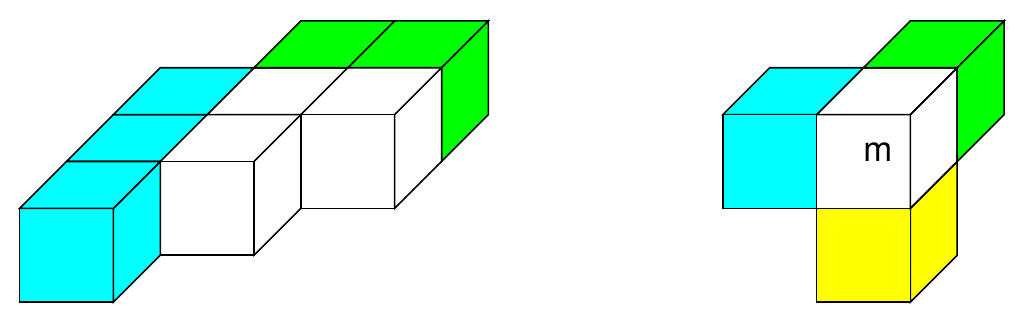}
\caption{Double box configuration of Examples \ref{normal} and \ref{normal2}.}
\end{figure}

Next we consider an example with moduli.
\begin{example} \label{normal2}
Suppose ${\bf{v}}=(1,1,1)$ and $p_1,p_2,p_3$ are mutually distinct. We consider an example with $|\bspi|=4$ and $\cC_{\bspi} \cong \PP^1$, namely 
$$\sfQ_{\bspi}=t_1t_2+t_1t_3+t_2t_3+[\O_{\PP^1}(1)] t_1t_2t_3.$$ 
See Figure 6.
Then we have 
$$
\sfV_{\bspi}=2[\O_{\PP^1}(1)]-1-2\,{\frac{[\O_{\PP^1}(-1)]}{t_{{1}}t_{{2}}t_{{3}}}}+{\frac {1}{t_{{1}}t_{{2}} t_3}}+\cdots,$$
where $\cdots$ indicates the $T_0$-moving terms. Since $T_{\cC_{\bspi}} = 2[\O_{\PP^1}(1)]-1$, we obtain
$$
T_{\cC_{\bspi}} - \sfV_{\bspi} = 2\,{\frac{[\O_{\PP^1}(-1)]}{t_{{1}}t_{{2}}t_{{3}}}}-{\frac {1}{t_{{1}}t_{{2}} t_3}}+\cdots. 
$$
A direct calculation shows that $\cdots$ produces a sign $-1$. We conclude 
$$\int_{\cC_{\bspi}} e(T_{\cC_{\bspi}} - \sfV_{\bspi}) \big|_{s_1+s_2+s_3=0}=2=\omega(\bspi).$$
\end{example}

These two examples suggest 
$$
\int_{\cC_{\bspi}} e(T_{\cC_{\bspi}} - \sfV_{\bspi}) \big|_{s_1+s_2+s_3=0}  \stackrel{?}{=} \omega(\bspi),
$$
for any $\bspi \in \Pi(\ccR,\varnothing,\varnothing,\varnothing)$. However this is \emph{false} as is shown by the following example.
\begin{example} \label{abnormal} 
Suppose ${\bf{v}}=(1,1,1)$ and $p_1,p_2,p_3$ are mutually distinct. We consider $\bspi$ and $\bspi'$ given by  
\begin{align*}
\sfQ_{\bspi}&=t_{{1}}t_{{3}}+t_{{2}}t_{{3}}+{t_{{1}}}^{2}t_{{3}}+{t_{{2}}}^{2}t_{{3}
}+t_{{1}}t_{{2}}t_{{3}}+t_{{1}}t_{{2}}+{t_{{1}}}^{2}t_{{2}}t_{{3}}+{t_
{{2}}}^{2}t_{{1}}t_{{3}}, \\
\sfQ_{\bspi'}&=t_{{1}}t_{{3}}+t_{{2}}t_{{3}}+{t_{{1}}}^{2}t_{{3}}+{t_{{2}}}^{2}t_{{3}
}+t_{{1}}t_{{2}}t_{{3}}+{t_{{1}}}^{2}{t_{{2}}}^{2}t_{{3}}+{t_{{1}}}^{2
}t_{{2}}t_{{3}}+{t_{{2}}}^{2}t_{{1}}t_{{3}},
\end{align*}
so $|\bspi|=|\bspi'|=8$ (see Figure 7). Note that neither $\cC_{\bspi}$ nor $\cC_{\bspi'}$ has moduli. We obtain 
\begin{align*}
\sfV_{\bspi}&=\frac{1}{t_1t_2t_3}-1+\cdots, \\
\sfV_{\bspi'}&=t_1t_2t_3 -\frac{1}{t_1^2t_2^2t_3^2}+\cdots,
\end{align*}
where $\cdots$ indicates the $T_0$-moving terms. Further calculation shows 
\begin{align*}
\frac{1}{e(\sfV_{\bspi})} \Big|_{s_1+s_2+s_3=0}=0 \neq \omega(\bspi) = 1, \\
\frac{1}{e(\sfV_{\bspi'})} \Big|_{s_1+s_2+s_3=0}=2 \neq \omega(\bspi') = 1.
\end{align*}
However, we do have
$$
\Big(\frac{1}{e(\sfV_{\bspi})} + \frac{1}{e(\sfV_{\bspi'})} \Big)\Big|_{s_1+s_2+s_3=0}= 2 =\omega(\bspi) +\omega(\bspi').
$$
In \cite{GKY}, we associate to any double box configuration $\bspi$ a double dimer model $\mathbb{D}(\bspi)$. It turns out that in this example
$$
\mathbb{D}(\bspi) = \mathbb{D}(\bspi').
$$
\end{example}

\begin{figure} 
\includegraphics[width=2in]{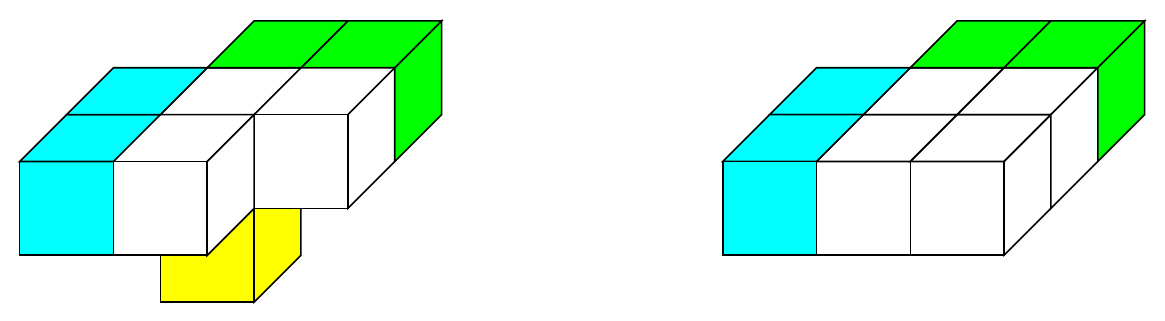}
\caption{Double box configurations of Example \ref{abnormal}.}
\end{figure}

In analogy to \cite{MNOP1, PT2}, we define the rank 2 equivariant vertex measure:
\begin{definition} \label{measurevertex}
Let $\ccR$ be a $T$-equivariant rank 2 reflexive sheaf on $\C^3$. For each $\bspi \in \Pi(\ccR, \varnothing, \varnothing, \varnothing)$ define the \emph{rank 2 equivariant vertex measure} associated to $\bspi$ by
\[
\mathsf{w}(\bspi) := \int_{\cC_{\bspi}} e(T_{\cC_{\bspi}} - \sfV_{\bspi}).
\]
Define the \emph{rank 2 equivariant vertex} as
\begin{equation} \label{W}
\sfW_{\ccR,  \varnothing, \varnothing, \varnothing}(q) := \sum_{\bspi \in \Pi(\ccR,  \varnothing, \varnothing, \varnothing)} \mathsf{w}(\bspi) q^{|\bspi|} \in \Z(s_1,s_2,s_3)[\![q]\!].
\end{equation}
\end{definition}

It is not a priori clear the ``Calabi-Yau specialization'' $s_1+s_2+s_3=0$ of the rank 2 equivariant vertex is well-defined.
\begin{proposition} \label{welldef1}
Let $\ccR$ be a $T$-equivariant rank 2 reflexive sheaf on $\C^3$. Then 
\begin{equation*} 
\sfW_{\ccR,  \varnothing, \varnothing, \varnothing}(q) \Big|_{s_1+s_2+s_3=0} \in \Z[\![q]\!]
\end{equation*}
is well-defined.
\end{proposition}
\begin{proof}
By Lemma \ref{(P^1)^3} below, this is a special case of Proposition-Definition \ref{propdef}. We only have to note that each $\cC_{\bspi}$ satisfies Assumption \ref{as1}.
\end{proof}

\begin{lemma} \label{(P^1)^3}
Let $\ccR$ be a $T$-equivariant rank 2 reflexive sheaf on $Y=\C^3$. We can view $Y$ as a $T$-invariant affine open subset of $X = (\PP^1)^3$. There exists a polarization $H$ on $X$ and a $T$-equivariant rank 2 $\mu$-stable reflexive sheaf $\tilde{\ccR}$ on $X$ such that $\tilde{\ccR}|_Y \cong \ccR$ and $\tilde{\ccR}$ restricted to any other $T$-invariant affine open subset is locally free.
\end{lemma}
\begin{proof}
Let $({\bf{u}},{\bf{v}},{\bf{p}})$ be the toric data of $\ccR$ and denote the pull-back of the class of the point via projection to each of the three factors of $X = (\PP^1)^3$ by $D_1,D_2,D_3$. We prove the case $\ccR$ is singular, i.e.~$v_1,v_2,v_3>0$ and $p_1,p_2,p_3$ are mutually distinct (Proposition \ref{rank2singular}). The degenerate cases can be treated in a similar fashion. Let 
$$
H = a_1D_1+a_2D_2+a_3D_3 
$$ 
be any polarization on $X$, i.e.~$a_1,a_2,a_3 \in \Z_{>0}$. Define $\tilde{\ccR}$ by the following toric data. For each of the faces $\rho \in \Delta(X)$ corresponding to the chart $Y = \C^3$, we take $u_\rho$, $v_\rho$, and $p_\rho$ equal to the corresponding value of $u_i$, $v_i$, and $p_i$ coming from $\ccR$. For each other $\rho \in \Delta(X)$, we take $u_\rho=v_\rho=0$. Then clearly $$\tilde{\ccR} |_Y \cong \ccR$$ and $\tilde{\ccR}$ restricted to the other standard charts is locally free (Proposition \ref{rank2singular}). It suffices to ensure $\tilde{\ccR}$ is $\mu$-stable. This is equivalent to (Section \ref{equivsh})
$$
(D_i H^2) v_i < (D_j H^2) v_j + (D_k H^2) v_k,
$$
for all $i,j,k \in \{1,2,3\}$ mutually distinct. These inequalities are equivalent to
$$
2a_j a_k v_i < 2a_i a_k v_j + 2a_i a_j v_k,
$$
for all $i,j,k \in \{1,2,3\}$ mutually distinct. Taking $a_1 = v_1$, $a_2 = v_2$, $a_3 = v_3$ proves the lemma.
\end{proof}

\begin{conjecture} \label{edgelessconj}
Let $\ccR$ be a $T$-equivariant rank 2 reflexive sheaf on $\C^3$. Then 
\begin{equation*}
\sfW_{\ccR,  \varnothing, \varnothing, \varnothing}(q) \Big|_{s_1+s_2+s_3=0}  = \sfZ_{\C^3,\ccR, \varnothing, \varnothing, \varnothing}(q).
\end{equation*}
\end{conjecture}

\begin{remark}
Recall from Section \ref{combsection} that $\sfZ_{\C^3,\ccR, \varnothing, \varnothing, \varnothing}(q)$ is just the generating function of Quot schemes
$$
\sum_{n = 0}^{\infty} e(\Quot_{\C^3}(\ccR,n)) q^n.
$$
There are two essentially different cases of Conjecture \ref{edgelessconj}. 
\begin{enumerate}
\item $\ccR$ is locally free, i.e.~some $v_i =0$ or some $p_i$ are equal (Proposition \ref{rank2singular}). 
Then Conjecture \ref{edgelessconj} and equation \eqref{splitZ} imply
\begin{equation*}
\sfW_{\ccR,  \varnothing, \varnothing, \varnothing}(q) \Big|_{s_1+s_2+s_3=0}  = M(q)^2.
\end{equation*}
\item $\ccR$ is singular, i.e.~all $v_i>0$ and $p_i$ are mutually distinct. In this case \eqref{W} only depends on ${\bf{v}}$ (by \eqref{tr1}) and we define 
\begin{equation*} 
\sfW_{v_1,v_2,v_3, \varnothing, \varnothing, \varnothing}(q) := \sfW_{\ccR,  \varnothing, \varnothing, \varnothing}(q). 
\end{equation*}
The right hand side of Conjecture \ref{edgelessconj} is computed in Theorem \ref{Ben}. Conjecture \ref{edgelessconj} is an interesting non-trivial \emph{combinatorial identity} stating
$$
\sfW_{v_1,v_2,v_3, \varnothing, \varnothing, \varnothing}(q) \Big|_{s_1+s_2+s_3=0}   = M(q)^2 \prod_{i=1}^{v_1} \prod_{j=1}^{v_2} \prod_{k=1}^{v_3} \frac{1-q^{i+j+k-1}}{1-q^{i+j+k-2}}.
$$
\end{enumerate}
\end{remark}

\begin{remark} \label{invariance}
Conjecture \ref{edgelessconj} implies Main Conjecture \ref{mainconj} in the case all $\hat{\bslambda} = \varnothing$ (no legs). In this case, the LHS of Main Conjecture \ref{mainconj} equals $\sfZ_{Y,\ccR|_Y}(q^{-2})$ up to an overall multiplicative factor of $q$ to some power. Since there are no legs, the formula of the previous remark shows the generating function is $M(q)^{2e(Y)}$ times a Laurent polynomial in $q^{-2}$. As can be checked by a small calculation, this Laurent polynomial is invariant under $q \leftrightarrow q^{-1}$ (up to an overall multiplicative factor).
\end{remark}

\begin{remark} \label{buddies1}
In Example \ref{abnormal}, we saw that in general $\mathsf{w}(\bspi) \neq \omega(\bspi)$. However, we do expect the following refinement of Conjecture \ref{edgelessconj} to be true. Let $\ccR$ be a $T$-equivariant rank 2 reflexive sheaf on $\C^3$. In \cite{GKY}, we associate a double dimer model $\mathbb{D}(\bspi)$ to any $\bspi \in \Pi(\ccR,\varnothing,\varnothing,\varnothing)$. For any double dimer model $\mathbb{D}$, we conjecture 
\begin{equation*}
 \sum_{{\footnotesize{\begin{array}{c} \bspi \in \Pi(\ccR,  \varnothing, \varnothing, \varnothing) \\ \mathbb{D}(\bspi) = \mathbb{D} \end{array}}}} \mathsf{w}(\bspi) \Big|_{s_1+s_2+s_3=0} q^{|\bspi|} =  \sum_{{\footnotesize{\begin{array}{c} \bspi \in \Pi(\ccR,  \varnothing, \varnothing, \varnothing) \\ \mathbb{D}(\bspi) = \mathbb{D} \end{array}}}} \omega(\bspi) q^{|\bspi|}. 
\end{equation*}
In \cite{GKY}, we refer to elements $\Pi(\ccR,  \varnothing, \varnothing, \varnothing)$ with the same double dimer model as \emph{buddies}. In Section \ref{T0loc}, we give a (conjectural) geometric characterization of buddies: $\bspi_1, \bspi_2$ are buddies if and only if $\cC_{\bspi_1}$, $\cC_{\bspi_2}$ are $T / T_0$-fixed connected components inside the same $T_0$-fixed connected component (Remark \ref{buddies2}). 
\end{remark}

\subsection{Rank 2 equivariant vertex/edge for one leg} \label{O(-1,-1)} 

For $Y = \C^* \times \C^2$, one can fix a double square configuration $\bslambda$ along $\C^*$ (Definition \ref{doublesquare}). Unlike the rank 1 case, the edge term \eqref{defVE} along this leg is not determined by $\bslambda$ alone. Therefore, $\sfE_{\bslambda}$ is not well-defined. The reason is that whether the leg has moduli or not depends on what happens in the ``corners attached to the leg'' (Example \ref{littlebox}). In this section, we consider 
$$
Y = \mathrm{Tot}(\O_{\PP^1}(-1) \oplus \O_{\PP^1}(-1)).
$$ 
Let $\ccR$ be a $T$-equivariant rank 2 reflexive sheaf on $Y$ described by toric data $(\mathbf{u},\mathbf{v},\mathbf{p})$. Let $0,1$ label the two vertices of $\Delta(Y)$. The polyhedron $\Delta(Y)$ has four non-compact faces. Two of them share the unique compact edge and we denote them by $\rho_1, \rho_2$. The other two we denote by $\rho_3, \rho_4$, where $\rho_3$ has vertex $0$ and $\rho_4$ has vertex $1$. We will see that $\ccR$ is always the restriction of a $T$-equivariant rank 2 $\mu$-stable reflexive sheaf on a (polarized) toric compactification $X$ of $Y$ (Lemma \ref{F1}). In order for this extension of $\ccR$ to have no $T$-fixed moduli, we assume $$v_{\rho_4} = 0.$$  

Fix a double square configuration $\bslambda$ along the compact leg and $\hat{\bspi} = \{\bspi_\alpha\}_{\alpha=0,1}$, where
$$
\bspi_\alpha \in \Pi(\ccR_\alpha,\bslambda,\varnothing,\varnothing).
$$
Then $\hat{\bspi}$ has at most 3 face components and can have several leg and vertex components of 1's (Definition \ref{vefcomps2}). Therefore, there are no moduli associated to the face components. In particular, the moduli space $\cC_{\bspi}$ (Definition \ref{Z}) decomposes according to vertex and edge components
$$
\cC_{\bspi} \cong \cC_{\mathrm{vert}} \times \cC_{\mathrm{edge}}.
$$
Note that the expression for the edge in \eqref{Galphabeta}, \eqref{defVE} makes sense without reference to a compactification of $X$. We write the corresponding expression as $\sfE_{\hat{\bspi}}$. It is tempting to define an equivariant edge measure 
\begin{equation} \label{edgemeasure}
\int_{\cC_{\mathrm{edge}}} e(T_{\cC_{\mathrm{edge}}} - \sfE_{\hat{\bspi}}).
\end{equation}
However, this may not be well-defined as illustrated by the following example.
\begin{example} \label{withlegs} 
Suppose $\ccR$ is given by toric data $(\mathbf{u},\mathbf{v},\mathbf{p})$ with $(v_{\rho_1}, v_{\rho_2}, v_{\rho_3}, v_{\rho_4}) = (1,1,1,0)$ and $p_1, p_2, p_3$ mutually distinct. We consider the following three double box configurations $\hat{\bspi}, \hat{\bspi}', \hat{\bspi}''$ with leg along the zero section of $Y$
\begin{align*}
(\sfQ_{\bspi_{0}}, \sfQ_{\bspi_{1}}) &= \left(t_1t_2+\frac{t_3(t_1+t_2+2t_1t_2)}{1-t_3}, \frac{t_1t_3^{-1}+t_2t_3^{-1}+2t_1t_2t_3^{-2}}{1-t_3^{-1}} \right), \\
(\sfQ_{\bspi_{0}'}, \sfQ_{\bspi_{1}'})&= \left(t_1t_2+\frac{t_3(t_1+t_2+[\O_{\PP^1}(1)]t_1t_2)}{1-t_3}, \frac{t_1t_3^{-1}+t_2t_3^{-1}+[\O_{\PP^1}(1)]t_1t_2t_3^{-2}}{1-t_3^{-1}} \right), \\
(\sfQ_{\bspi_{0}''}, \sfQ_{\bspi_{1}''})&= \left(\frac{t_3(t_1+t_2+t_1t_2)}{1-t_3}, \frac{t_1t_3^{-1}+t_2t_3^{-1}+t_1t_2t_3^{-1}}{1-t_3^{-1}} \right).
\end{align*}
The double box configurations in chart $U_0$ are depicted in Figure 8.
One can compute
\begin{align*}
(\sfV_{\bspi_0},\sfV_{\bspi_1}) =& \left(-\frac{1}{t_1^2}-\frac{1}{t_1}-\frac{1}{t_2^2}-\frac{1}{t_2}+\frac{1}{t_1 t_3}+\frac{1}{t_2 t_3}+\frac{t_1}{t_2 t_3}+\frac{t_2}{t_1 t_3},0 \right), \\
(\sfV_{\bspi_{0}'}, \sfV_{\bspi_{1}'}) =& \left(-\frac{1}{t_1^2}+\frac{1}{t_1}-\frac{[\O_{\PP^1}(1)]}{t_1}-\frac{1}{t_2^2}+\frac{1}{t_2}-\frac{[\O_{\PP^1}(1)]}{t_2}-\frac{2}{t_1 t_2}+\frac{[\O_{\PP^1}(1)]}{t_1 t_2}+\frac{2}{t_3} \right. \\
&\left. -\frac{[\O_{\PP^1}(-1)]}{t_3}-\frac{1}{t_1 t_3}+\frac{[\O_{\PP^1}(-1)]}{t_1 t_3}-\frac{1}{t_2 t_3}+\frac{[\O_{\PP^1}(-1)]}{t_2 t_3}+\frac{t_1}{t_2 t_3}+\frac{t_2}{t_1 t_3}, 0 \right), \\
(\sfV_{\bspi_{0}''}, \sfV_{\bspi_{1}''})=& \left(-1-\frac{1}{t_1}-\frac{1}{t_2}+\frac{1}{t_1 t_3}+\frac{1}{t_2 t_3}+\frac{1}{t_1 t_2 t_3},0 \right), \\
\sfE_{\hat{\bspi}}=&\frac{2 t_1}{t_2}+\frac{2 t_2}{t_1}-\frac{2}{t_1^2 t_3}-\frac{2}{t_2^2 t_3}, \\
\sfE_{\hat{\bspi}'}=&2 [\O_{\PP^1}(1)]-1+\frac{[\O_{\PP^1}(1)] t_1}{t_2}+\frac{[\O_{\PP^1}(1)] t_2}{t_1}-\frac{[\O_{\PP^1}(-1)]}{t_1^2 t_3}-\frac{[\O_{\PP^1}(-1)]}{t_2^2 t_3}\\
&-\frac{2 [\O_{\PP^1}(-1)]}{t_1 t_2 t_3}+\frac{1}{t_1 t_2 t_3}, \\
\sfE_{\hat{\bspi}''}=&1+\frac{t_1}{t_2}+\frac{t_2}{t_1}-\frac{1}{t_1^2 t_3}-\frac{1}{t_2^2 t_3}-\frac{1}{t_1 t_2 t_3}.
\end{align*}
Interestingly, $\sfE_{\hat{\bspi}''}$ has a positive $T$-fixed part, so $1/e(\sfE_{\hat{\bspi}''})$ is not well-defined. Therefore \eqref{edgemeasure} is in general not well-defined. As a consistency check: an explicit calculation using these expressions leads to the sign predicted by Main Conjecture \ref{mainconj}. 
\end{example}

\begin{figure} 
\includegraphics[width=2.5in]{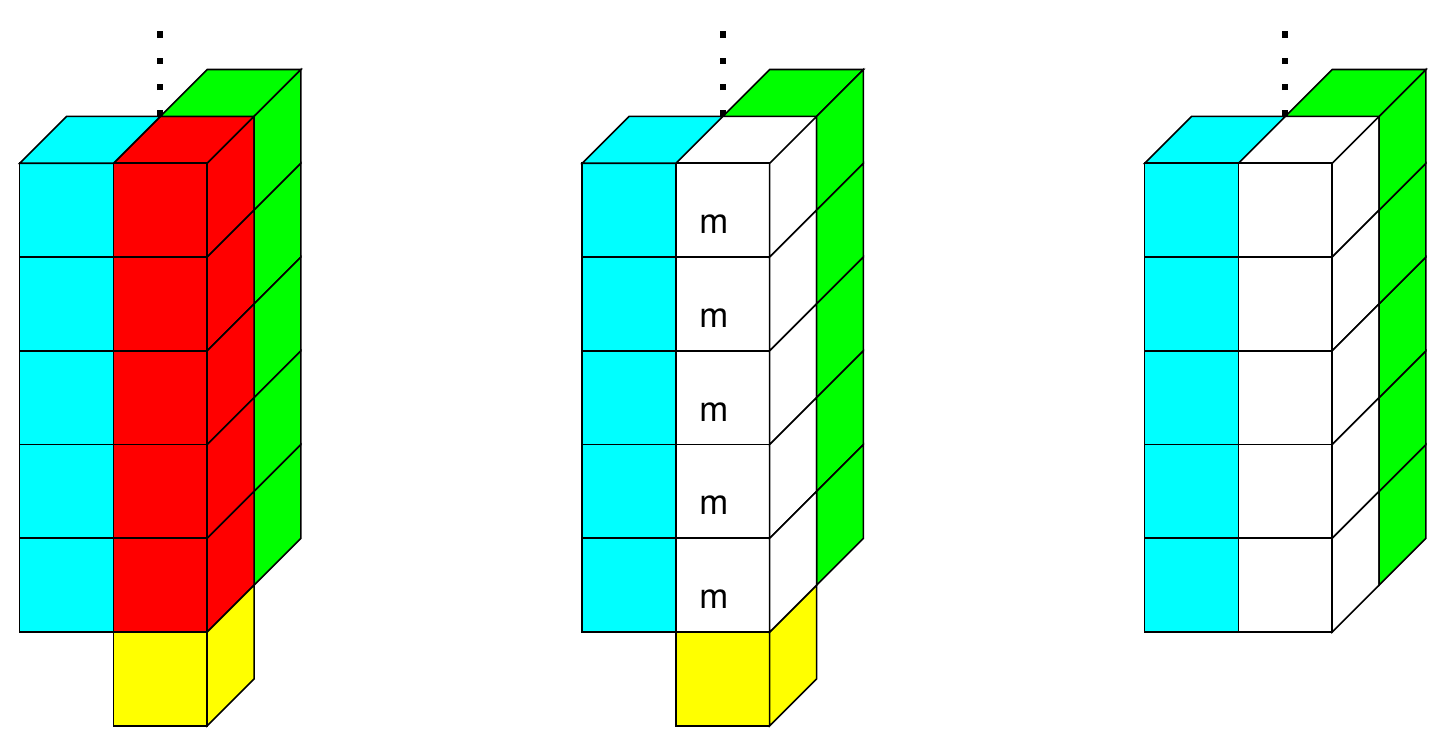}
\caption{Double box configurations $\bspi_{0}, \bspi_{0}', \bspi_{0}''$ of Example \ref{withlegs}.}
\end{figure}

In the previous example, the obstruction term in $ \sfE_{\hat{\bspi}''}$ cancels after adding $\sfV_{\bspi_{0}''}$, $\sfV_{\bspi_{1}''}$. We should therefore not separate these three terms. This motivates the following definition.
\begin{propositiondefinition}
Let $Y = \mathrm{Tot}(\O_{\PP^1}(-1) \oplus \O_{\PP^1}(-1))$ and label the vertices and faces of $\Delta(Y)$ as before. Let $\ccR$ be a $T$-equivariant rank 2 reflexive sheaf on $Y$ described by toric data $(\mathbf{u},\mathbf{v},\mathbf{p})$ with $v_{\rho_4} = 0$. Let $\bslambda$ be a double square configuration along the compact edge and let $\hat{\bspi} = \{\bspi_\alpha\}_{\alpha = 0,1}$, where
$$
\bspi_\alpha \in \Pi(\ccR_\alpha,\bslambda,\varnothing,\varnothing), \ \mathrm{for} \ \alpha=0,1.
$$
Then 
$$
\mathsf{w}(\hat{\bspi}) := \int_{\cC_{\hat{\bspi}}} e(T_{\cC_{\hat{\bspi}}} - (\sfV_{\bspi_0} + \sfV_{\bspi_1} + \sfE_{\hat{\bspi}})) \in \Q(s_1,s_2,s_3)
$$
is the rank 2 equivariant resolved conifold measure associated to $\hat{\bspi}$. The specialization 
$$
\mathsf{w}(\hat{\bspi}) \Big|_{s_1+s_2+s_3=0} \in \Q
$$
is well-defined. Define the generating function of these measures as
\begin{equation*} 
\sfW_{\ccR, \bslambda, \varnothing, \varnothing}(q) = \sum_{\hat{\bspi} \in \Pi(\ccR, \bslambda, \varnothing, \varnothing)} \mathsf{w}(\hat{\bspi}) q^{|\bspi| + f_{-1,-1}(\bslambda) + g_{u_{\rho_3}, u_{\rho_4}, v_{\rho_3}, 0}(\bslambda)} \in \Z(s_1,s_2,s_3)[\![q]\!],
\end{equation*}
where $f_{-1,-1}(\bslambda)$, $g_{u_{\rho_3}, u_{\rho_4}, v_{\rho_3}, 0}(\bslambda)$ were defined in Definition \ref{Z}.
\end{propositiondefinition}
\begin{proof}
We have to prove the Calabi-Yau specialization $s_1+s_2+s_3=0$ is well-defined. The proof is similar to that of Proposition \ref{welldef1} and follows from the following lemma and Proposition-Definition \ref{propdef}.
\end{proof}

\begin{lemma} \label{F1}
Let $\ccR$ be any $T$-equivariant rank 2 reflexive sheaf on $Y = \mathrm{Tot}(\O_{\PP^1}(-1) \oplus \O_{\PP^1}(-1))$ described by toric data $(\mathbf{u},\mathbf{v},\mathbf{p})$ (not necessarily with $v_{\rho_4} = 0$). Then there exists a polarized toric compactification $X$ of $Y$ and a rank 2 $\mu$-stable reflexive sheaf $\tilde{\ccR}$ on $X$, such that $\tilde{\ccR}|_{Y} \cong \ccR$ and the restriction of $\tilde{\ccR}$ to any $T$-invariant affine open subset not contained in $Y$ is locally free. 
\end{lemma} 
\begin{proof}
We treat the case all $v_i>0$ and all $p_i$ are mutually distinct. The degenerate cases can be done similarly.

Let $\FF_1 = \PP(\O_{\PP^1}\oplus\O_{\PP^1}(1))$ be the first Hirzebruch surface and let $C$ be the $T$-fixed line with $C^2=-1$. Let $F$ be a $T$-fixed fibre. Define 
$$
 \pi : X = \PP(\O_{\FF_1} \oplus \O_{\FF_1}(F)) \longrightarrow \FF_1.
$$
The following divisors form a $\mathbb{Z}$-basis for $\Pic(X)$ 
$$
D_1 = \pi^{-1}(C), \ D_2 = \pi^{-1}(F), \ D_3=\PP(\O_{\FF_1} \oplus 0)\subset X.
$$ 
Then $C$ is a $(-1,-1)$ curve in $X$ contained in two toric charts $U_0, U_1$. Their union is isomorphic to $Y$. We define $\tilde{\ccR}$ on $X$ by the following toric data $(\tilde{\mathbf{u}},\tilde{\mathbf{v}},\tilde{\mathbf{p}})$. The faces $\rho_1, \rho_2, \rho_3, \rho_4$ of $\Delta(Y)$ (introduced at the beginning of this section) correspond to faces of $\Delta(X)$ and we take $\tilde{u}_{\rho_i}, \tilde{v}_{\rho_i}, \tilde{p}_{\rho_i}$ from the given toric data $(\mathbf{u},\mathbf{v},\mathbf{p})$. For all other faces $\rho$ of $\Delta(X)$, we set $\tilde{u}_\rho = \tilde{v}_\rho = 0$. Then $\tilde{\ccR}|_Y \cong \ccR$ and $\tilde{\ccR}$ restricted to any other other toric chart is locally free (Proposition \ref{rank2singular}). We are reduced to finding a polarization $H$ on $X$ such that $\tilde{\ccR}$ is $\mu$-stable. 

The intersection numbers among the toric divisors are
\begin{align*}
&D_1D_2D_3 = pt, \ D_1D_3^2=D_1^2 D_3=-pt, \\
&D_2^2=D_1^2D_2=D_3^2D_2=D_1^3=D_3^3=0,
\end{align*}
where $pt$ denotes the class of a point. A divisor $$H = a_1 D_1 + (a_1 + a_2 + a_3) D_2 + a_3 D_3$$ is ample if and only if $a_1,a_2,a_3>0$. Take $\lambda \in \Q_{>0}$ arbitrary and $a_1 = \lambda v_1$, $a_2 = v_2+v_4$, $a_3 = v_3$. Using the above intersection numbers, it is not hard to see that there exists a $\lambda \in \Q_{>0}$ such that the following four stability inequalities hold 
$$
(H^2D_i) v_i < (H^2D_j) v_j + (H^2D_k) v_k + (H^2D_l) v_l, \ \forall i,j,k,l \in \{1,2,3,4\} \mathrm{ \ mutually \ distinct}. 
$$
After scaling $H$ such that all its coefficients are integer, we are done.
\end{proof}

Parallel to Main Conjecture \ref{mainconj} we conjecture:
\begin{conjecture} \label{(-1,-1)conj}
Let $Y = \mathrm{Tot}(\O_{\PP^1}(-1) \oplus \O_{\PP^1}(-1))$ and label the vertices and faces of $\Delta(Y)$ as before. Let $\ccR$ be any $T$-equivariant rank 2 reflexive sheaf on $Y$ described by toric data $(\mathbf{u},\mathbf{v},\mathbf{p})$ with $v_{\rho_4} = 0$. Let $\bslambda$ be a double square configuration along the compact edge. Then 
\begin{equation*} 
\sfW_{\ccR, \bslambda, \varnothing, \varnothing}(q) \Big|_{s_1+s_2+s_3=0} = (-1)^{|\bslambda| (v_{\rho_1} + v_{\rho_2} + v_{\rho_3} +1)} \sfZ_{\ccR, \bslambda, \varnothing, \varnothing}(q).
\end{equation*}
\end{conjecture}


\subsection{Expected obstructions} \label{good}

In this section, we explain where the sign in Main Conjecture \ref{mainconj} comes from. The following theorem discusses the contribution of connected components $\cC$ with ``expected obstructions''. The argument is a variation on \cite[Sect.~4.10, 4.11]{MNOP1}.  
\begin{theorem} \label{Tthm}
Let $Y$ be a smooth toric Calabi-Yau 3-fold with toric compactification $X$ with $H^0(K_{X}^{-1}) \neq 0$ and polarization $H$. Let $\ccR$ be a $T$-equivariant rank 2 $\mu$-stable reflexive sheaf on $X$ with toric data $({\bf{u}},{\bf{v}},{\bf{p}})$. For each $\alpha\beta \in E_c(Y)$, there are two faces $\rho_{1, \alpha\beta}$, $\rho_{2,\alpha\beta}$ which share $\alpha\beta$ as an edge and two disjoint faces $\rho_{3,\alpha\beta}$, $\rho_{4,\alpha\beta}$ connected by the edge $\alpha\beta$. Let $\hat{\bspi} \in \Pi(\ccR|_Y,\hat{\bslambda})$ for some $\hat{\bslambda} = \{\bslambda_{\alpha\beta} \in \Lambda(\ccR_{\alpha\beta})\}_{\alpha\beta \in E_c(Y)}$ and consider the connected component $\cC = \cC_{\hat{\bspi}}$. Assume the following:
\begin{enumerate}
\item[(i)] All closed points of $\cC$ have reflexive hull $\ccR$.
\item[(ii)] The $T_0$-fixed part of $\sum_{\alpha \in V(Y)} \sfV_\alpha + \sum_{\alpha\beta \in E_c(Y)} \sfE_{\alpha\beta}$ equals $T_{\cC} - \Omega_{\cC} \otimes (t_1 t_2 t_3)^{-1}$.
\end{enumerate}
Then 
\begin{align*}
\DT(\cC) = e(\cC) \prod_{\alpha\beta \in E_c(Y)} (-1)^{|\bslambda_{\alpha\beta}| (m_{\alpha\beta} (v_{\rho_{1,\alpha\beta}} + v_{\rho_{2,\alpha\beta}} + 1) +v_{\rho_{3,\alpha\beta}} + v_{\rho_{4,\alpha\beta}})}.
\end{align*}
\end{theorem}

\begin{remark}
In the case $\cC$ is in addition isolated, (ii) states that $\sum_{\alpha \in V(Y)} \sfV_\alpha + \sum_{\alpha\beta \in E_c(Y)} \sfE_{\alpha\beta}$ has no $T_0$-fixed terms. In \cite{MNOP1}, the $T$-fixed locus automatically consists of isolated reduced points and the absence of $T_0$-fixed terms always holds (this is the crucial technical point mentioned in \cite[Sect.~4.11]{MNOP1}). In our case, $\cC$ could be non-isolated. Moreover, even when $\cC$ is isolated there can be non-zero $T_0$-fixed terms as for $\bspi$ and $\bspi'$ in Example \ref{abnormal}. Examples \ref{normal}, \ref{normal2}, and \ref{withlegs} satisfy the assumptions of the theorem. 
\end{remark}

\begin{proof}[Proof of Theorem \ref{Tthm}]
From the definition and assumptions
\begin{align*}
\DT(\cC) = \int_{\cC} e(\Omega_{\cC} \otimes (t_1t_2t_3)^{-1}) \cdots \Big|_{s_1+s_2+s_3=0},
\end{align*}
where $\cdots$ stands for the contribution of all $T_0$-moving terms. This is equal to 
\begin{align*}
\DT(\cC) = \int_{\cC} c_{\mathrm{top}}(\Omega_{\cC}) \cdots \Big|_{s_1+s_2+s_3=0}.
\end{align*}
Because of degree reasons, the contribution of $\cdots$ is simply $\pm 1$. More precisely, it is $-1$ to the power the number of (Serre dual) $T_0$-moving \emph{pairs} in $\sum_{\alpha \in V(Y)} \sfV_\alpha + \sum_{\alpha\beta \in E_c(Y)} \sfE_{\alpha\beta}$ by Proposition \ref{SD}.  We conclude
\begin{align*}
\DT(\cC) = (-1)^{\dim(\cC) + \text{number of} \ T_0 \text{-moving pairs}} e(\cC).
\end{align*}
In order to determine the sign, we use the following claim. For all $\alpha\beta \in E_c(Y)$ and $\alpha \in V(Y)$, there exist splittings
\begin{align}
\sfE_{\alpha\beta} &= \sfE_{\alpha\beta}^{+} + \sfE_{\alpha\beta}^{-}, \label{splittingE} \\
\sfV_{\alpha} &= \sfV_{\alpha}^{+} + \sfV_{\alpha}^{-}, \label{splittingV}
\end{align}
such that\footnote{Note that \cite{MNOP1} only show these dualities after specialization $t_1 t_2 t_3=1$. Our dualities also hold before specialization.} 
\begin{align}
\sfE_{\alpha\beta}^{-} &= - (t_1 t_2 t_3)^{-1} \overline{\sfE}_{\alpha\beta}^{+}, \label{Edual} \\
\sfV_{\alpha}^{-} &= - (t_1 t_2 t_3)^{-1} \overline{\sfV}_{\alpha}^{+}. \label{Vdual}
\end{align}
If the claim is proved, then we conclude $\DT(\cC)$ equals
\begin{align*}
(-1)^{\dim(\cC) + \sum_{\alpha \in V(Y)} (\sfV^{+}_{\alpha}(1,1,1) - \mathbf{Constant}(\sfV_{\alpha}^{+})) + \sum_{\alpha\beta \in E_c(Y)} (\sfE^{+}_{\alpha\beta}(1,1,1) - \mathbf{Constant}(\sfE_{\alpha\beta}^{+})} e(\cC),
\end{align*}
where $\mathbf{Constant}(\cdots)$ denotes the rank of the $T_0$-fixed term. Moreover, $\sfV^{+}_{\alpha}(1,1,1)$ and $\sfE^{+}_{\alpha\beta}(1,1,1)$ should be interpreted as first restricting the vertex and edge expression to (any) closed point of $\cC$, thus getting rid of moduli, and then putting $(t_1,t_2,t_3) = (1,1,1)$. So we need to prove the above splittings and compute the terms in the power modulo 2. 

For \eqref{splittingE}, we define (similar to \cite[Sect.~4.10]{MNOP1})
$$
\sfG_{\alpha\beta}^{+} := - \sfQ_{\alpha\beta} \overline{P}(\ccR_{\alpha\beta}) - \sfQ_{\alpha\beta} \overline{\sfQ}_{\alpha\beta} \frac{1-t_2}{t_2}.
$$
In this definition, the coordinates are chosen such that $C_{\alpha\beta} \cap U_\alpha$ is given by $x_2=x_3=0$. We then define $\sfE_{\alpha\beta}^{+}$ as in expression \eqref{defVE}, but with $\sfG_{\alpha\beta}$ replaced by $\sfG_{\alpha\beta}^{+}$. Also define $\sfE_{\alpha\beta}^{-} := \sfE_{\alpha\beta} - \sfE_{\alpha\beta}^{+}$. The identity \eqref{Edual} can then be checked by direct calculation. This uses $m_{\alpha\beta} + m_{\alpha\beta}' = -2$ (valid because $\alpha\beta \in E_c(Y)$ and $Y$ is Calabi-Yau). From the definition of $\sfE_{\alpha\beta}^{+}$, we obtain 
\begin{align*}
\sfE_{\alpha\beta}^{+}(1,1,1) &= - \frac{\mathrm{d}}{\mathrm{d} \tau} \Big( \tau \sfG_{\alpha\beta}^{+}(t_2,t_3) - \sfG_{\alpha\beta}^{+}(t_2 \tau^{m_{\alpha\beta}}, t_3 \tau^{m_{\alpha\beta}'}) \Big) \Big|_{(\tau,t_2,t_3)=(1,1,1)} \\
&= -\sfG_{\alpha\beta}^{+}(1,1) + \Big( m_{\alpha\beta} t_2 \frac{\partial}{\partial t_2} + m_{\alpha\beta}' t_3 \frac{\partial}{\partial t_3} \Big) \sfG_{\alpha\beta}^{+}(t_2,t_3) \Big|_{(t_2,t_3) = (1,1)} \mod 2 \\
&= \Big( m_{\alpha\beta} t_2 \frac{\partial}{\partial t_2} + m_{\alpha\beta}' t_3 \frac{\partial}{\partial t_3} \Big) \sfG_{\alpha\beta}^{+}(t_2,t_3) \Big|_{(t_2,t_3) = (1,1)} \mod 2,
\end{align*}
where the second equality follows from $P(\ccR_{\alpha\beta})(1,1,1) = 0 \mod 2$ ($\ccR$ has rank 2). Using the fact that $P(\ccR_{\alpha\beta})(1,1,1) = 0 \mod 2$ and $m_{\alpha\beta}+m_{\alpha\beta}'=-2$, one obtains
\begin{equation} \label{E111}
\sfE_{\alpha\beta}^{+}(1,1,1) = m_{\alpha\beta} |\bslambda_{\alpha\beta}| (v_{\rho_{1,\alpha\beta}} + v_{\rho_{2,\alpha\beta}}  + 1) \mod 2,
\end{equation}
where $\rho_{1,\alpha\beta}, \rho_{2,\alpha\beta} \in F(X)$ are the two faces with common edge $\alpha\beta$. 

The splitting \eqref{splittingV} is more complicated. Let $\sfG_{\alpha\beta}^{+}$ be defined as above. Naively, one would like to define
$$
\sfV_{\alpha}^{+} := \sfQ_\alpha \overline{P}(\ccR_\alpha) -\sfQ_\alpha \overline{\sfQ}_{\alpha} \frac{(1-t_1)(1-t_2)}{t_1 t_2} + \sum_{i=1}^{3} \frac{\sfG_{\alpha\beta_i}^{+}(t_{i'},t_{i''})}{1-t_i}
$$
and $\sfV_{\alpha}^{-} := \sfV_\alpha - \sfV_{\alpha}^{+}$. Here the vertices neighbouring $\alpha \in V(Y)$ are denoted by $\beta_1, \beta_2, \beta_3$, where the labelling is such that $C_{\alpha \beta_i} \cap U_\alpha$ is given by $\{x_{i'} = x_{i''} = 0\}$ for all $i,i',i'' \in \{1,2,3\}$ mutually distinct. 
In the case $|\bslambda_{\alpha\beta}| = 0$ (no legs), this works and one obtains \eqref{splittingV}. The general case is more complicated. One introduces $\sfQ_{\alpha}'$ by the equation
$$
\sfQ_\alpha = \sfQ_{\alpha}' + \sum_{i=1}^{3} \frac{\sfQ_{\alpha\beta_i}}{1-t_i}
$$
and defines $\sfV_{\alpha}^{+}$ exactly analogous to \cite[Sect.~4.10]{MNOP1}. Setting $\sfV_{\alpha}^{-} := \sfV_\alpha - \sfV_{\alpha}^{+}$, equation \eqref{splittingV} can be deduced. The correct definition of $\sfV_{\alpha}^{+}$ leads to
$$
\sfV_{\alpha}^{+}(1,1,1) = \sfQ_{\alpha}'(1,1,1) \overline{P}(\ccR_\alpha)(1,1,1) + \sum_{i=1}^{3} \frac{\sfQ_{\alpha\beta_i} (\overline{P}(\ccR_\alpha) - \overline{P}(\ccR_{\alpha\beta_i}))}{1-t_i} \Big|_{(t_1,t_2,t_3) = (1,1,1)} \mod 2.
$$
This equality follows after noting that various terms in $\sfV_{\alpha}^{+}$ are zero or even for $(t_1,t_2,t_3)=(1,1,1)$. A similar (but easier) calculation as for $\sfE_{\alpha\beta}^{+}(1,1,1)$ shows
\begin{equation} \label{V111}
\sfV_{\alpha}^{+}(1,1,1) = \sum_{i=1}^{3} |\bslambda_{\alpha\beta_i}| v_{\rho_i} \mod 2,
\end{equation}
where $\rho_1,\rho_2,\rho_3$ are the three faces with vertex $\alpha$ labelled such that $\rho_{i'}, \rho_{i''}$ are the faces sharing edge $\alpha\beta_i$. 

Finally, we claim
\begin{align*}
\sum_{\alpha \in V(Y)} \mathbf{Constant}(\sfV_{\alpha}^{+}) + \sum_{\alpha\beta \in E_c(Y)} \mathbf{Constant}(\sfE_{\alpha\beta}^{+}) &= \dim(\cC) \mod 2.
\end{align*}
Define
$$
\mathsf{X}^{\pm} := \sum_{\alpha \in V(Y)} \sfV^{\pm}_{\alpha} + \sum_{\alpha\beta \in E_c(Y)} \sf{E}^{\pm}_{\alpha\beta}.
$$
Take any $T_0$-fixed term $f$ in $\mathsf{X}^{+}$ \emph{not} occurring in $T_{\cC} - \Omega_{\cC} \otimes (t_1 t_2 t_3)^{-1}$. This term must cancel after adding $\mathsf{X}^{-}$ (by assumption). Therefore, $\mathsf{X}^{-}$ contains the term $-f$. By duality \eqref{Edual} and \eqref{Vdual}, $\mathsf{X}^{+}$ must contain $f^* (t_1 t_2 t_3)^{-1}$. Therefore, nonzero terms in $\mathsf{X}^{+}$ not occurring in $T_{\cC} - \Omega_{\cC} \otimes (t_1 t_2 t_3)^{-1}$ come in pairs (and likewise for $\mathsf{X}^{-}$). Distributing the terms of $T_{\cC} - \Omega_{\cC} \otimes (t_1 t_2 t_3)^{-1}$ over $\mathsf{X}^{\pm}$ by using \eqref{Edual}, \eqref{Vdual} and calculating modulo 2 shows the claim.
\end{proof}

\subsection{$T_0$-localization} \label{T0loc}

In this section, we prove Main Conjecture \ref{mainconj} from a $T_0$-localization argument similar to \cite[Sect.~5.2]{PT2}. Pandharipande-Thomas need two conjectures for their argument \cite[Conj.~2, Conj.~3]{PT2}. We need completely similar analogs of their conjectures in our setting. The first is smoothness of the $T_0$-fixed locus (Conjecture \ref{T0conj1}). The second will be described below. 
\begin{theorem} \label{T0thm1}
Assume Conjecture \ref{T0conj1} is true. Let $Y$ be a smooth quasi-projective toric Calabi-Yau 3-fold with toric compactification $X$ satisfying $H^0(K_{X}^{-1}) \neq 0$. Fix a polarization $H$ on $X$ and let $\M_{X} := \M_{X}^{H}(2,c_1,c_2,c_3)$ and $\M_{Y \subset X} := \M_{Y \subset X}^{H}(2,c_1,c_2,c_3)$. Suppose $\cC \subset \M_{X}^{T_0}$ is a connected component satisfying Assumption \ref{as2} with reflexive hull $\ccR$. Let $N^{\vir,0}$ be the virtual normal bundle on $\cC$ w.r.t.~the $T_0$-action. Then the class
$$
N^{\vir,0} + \langle \ccR,\ccR \rangle_0 \otimes \O_{\cC} \in K_{0}^{T_0}(\cC)
$$ 
is of the form $N^+ + N^-$ with $\overline{N}^{-} = -N^{+} \in K_{0}^{T_0}(\cC)$. Moreover
\begin{align*}
\DT(\cC) &= (-1)^{\dim(\cC)+\rk(N^+)} e(\cC). 
\end{align*}
\end{theorem}
\begin{proof}
By symmetry of the $T_0$-equivariant obstruction theory (Theorem \ref{sym})
$$
E^{\mdot \vee} = T_{\cC} - \Omega_{\cC} + N^{\vir,0}.
$$
Therefore, the invariant is given by 
$$
\DT(\cC) = \int_{\cC} e(\Omega_{\cC} - (N^{\vir,0} + \langle \ccR,\ccR\rangle_0 \otimes \O_{\cC})).
$$
Note that $\langle \ccR,\ccR \rangle_0$ only has $T_0$-moving terms by Assumption \ref{as2}. Theorem \ref{sym} therefore gives 
$$
N^{\vir,0} + \langle \ccR,\ccR \rangle_0 = - \overline{\big( N^{\vir,0} + \langle \ccR,\ccR \rangle_0 \otimes \O_{\cC} \big)} \in K_{0}^{T_0}(\cC).
$$
Consequently, $N^{\vir,0} + \langle \ccR,\ccR \rangle_0$ is of the form $N^+ + N^-$ with $\overline{N}^{-} = -N^{+}$ and we conclude
\begin{equation*}
\DT(\cC) = \int_{\cC} c_{\mathrm{top}}(\Omega_{\cC}) (-1)^{\rk(N^+)} = (-1)^{\dim(\cC) + \rk(N^+)} e(\cC). \qedhere
\end{equation*}
\end{proof}

\begin{remark} \label{buddies2}
Combining Theorem \ref{T0thm1} and Proposition \ref{T0DT=TDT} gives
$$
\DT(\cC) = (-1)^{\dim(\cC)+\rk(N^+)} \sum_{\cC_i \subset \cC^{\C^*}} e(\cC_i) = \sum_{\cC_i \subset \cC^{\C^*}} \DT(\cC_i),
$$
where $T/T_0 \cong \C^*$. It is tempting to conclude 
\begin{equation} \label{tempt}
\DT(\cC_i) \stackrel{?}{=} (-1)^{\dim(\cC)+\rk(N^+)} e(\cC_i),
\end{equation}
for any $\C^*$-fixed component $\cC_i$ of $\cC$. This is certainly not true in general, because it would imply $\DT(\cC_i)$ is never zero contradicting Example \ref{abnormal}. This is very reminiscent of the discussion of buddies in Remark \ref{buddies1}. This leads us to expect that two elements $\bspi_1, \bspi_2 \in \Pi(\ccR,\hat{\bslambda})$ are buddies if and only if $\cC_{\bspi_1}$, $\cC_{\bspi_2}$ lie in the same connected component of the $T_0$-fixed locus. This would provide a geometric explanation for the (combinatorially mysterious) notion of buddies. So in general the classical count $\pm e(\cC_i)$ and virtual count $\DT(\cC_i)$ are different. Remarkably, they become equal again after summing over all connected components $\cC_i$ of the same connected component of the $T_0$-fixed locus. 
\end{remark}

The analog of the second conjecture of Pandharipande-Thomas \cite[Conj.~3]{PT2} is: 
\begin{conjecture} \label{T0conj2}
Let $Y$ be a smooth quasi-projective toric Calabi-Yau 3-fold with toric compactification $X$ satisfying $H^0(K_{X}^{-1}) \neq 0$. Fix a polarization $H$, let $\M_{X} := \M_{X}^{H}(2,c_1,c_2,c_3)$, and let $\M_{Y \subset X} := \M_{Y \subset X}^{H}(2,c_1,c_2,c_3)$. Suppose $\cC \subset \M_{X}^{T_0}$ is a connected component satisfying Assumption \ref{as2} and let $\cC_i$ be a connected component of $\cC^{\C^*}$ where $\C^* \cong T/T_0$. Since $\cC_i$ is a connected component of $\M_{X}^{T}$, we have defined vertices $\sfV_\alpha$, $\sfV_{\alpha}^{\pm}$ and edges $\sfE_{\alpha\beta}$, $\sfE_{\alpha\beta}^{\pm}$ on it. For any $[\F] \in \cC_i$, these give restrictions $\sfV_\alpha|_{[\F]}$, $\sfV_{\alpha}^{\pm}|_{[\F]}$, $\sfE_{\alpha\beta}|_{[\F]}$, $\sfE_{\alpha\beta}^{\pm}|_{[\F]}$. Then
\begin{align*}
\dim \Ext^1(\F,\F)^{T_0} = &\sum_{\alpha \in V(Y)} \mathbf{Constant}(\sfV_\alpha^{+}|_{[\F]}) + \sum_{\alpha\beta \in E_c(Y)} \mathbf{Constant}(\sfE_{\alpha\beta}^{+}|_{[\F]}) \mod 2,
\end{align*}
where $\mathbf{Constant}(\cdots)$ denotes the rank of the $T_0$-fixed term.
\end{conjecture}

\begin{remark}
We recall that the quantities $\sfV_{\alpha}^{\pm}|_{[\F]}$ and $\sfE_{\alpha\beta}^{\pm}|_{[\F]}$ were defined in the proof of Theorem \ref{Tthm}. Their definition and equations \eqref{E111}, \eqref{V111} hold under the assumptions of Conjecture \ref{T0conj2} and do not need the additional assumption (ii) of Theorem \ref{Tthm}. 
\end{remark}

\begin{remark} \label{evidence}
Suppose the setting is as in Conjecture \ref{T0conj2} for some $T_0$-fixed connected component $\cC$. If $\cC_i$ is a connected component of $\cC^{\C^*}$ satisfying (ii) of Theorem \ref{Tthm}, then it satisfies Conjecture \ref{T0conj2} by the proof of Theorem \ref{Tthm}. If $\cC = \cC^{\C^*}$ then $\cC$ is smooth (namely a product of $\PP^1$'s) so it satisfies Conjecture \ref{T0conj1}. See Examples \ref{normal}, \ref{normal2}, and \ref{withlegs}. If $\cC \neq \cC^{\C^*}$, we also have evidence supporting Conjecture \ref{T0conj1}, e.g.~the $T_0$-fixed terms of $\bspi$ and $\bspi'$ of Example \ref{abnormal} suggest that $\cC \cong \PP^1$ and $\cC^{\C^*}$ consists of its two $T$-fixed points.
\end{remark}

Conjectures \ref{T0conj1} and \ref{T0conj2} imply (a slightly weaker version of) Main Conjecture \ref{mainconj}. 
\begin{theorem} \label{T0thm}
Let $Y$ be a smooth quasi-projective toric Calabi-Yau 3-fold with toric compactification $X$ satisfying $H^0(K_{X}^{-1}) \neq 0$. Fix a polarization $H$ on $X$ and Chern classes $c_1$, $c_2$. Let $\ccR$ be a $T$-equivariant rank 2 $\mu$-stable reflexive sheaf on $X$ with toric data $({\bf{u}},{\bf{v}},{\bf{p}})$. We assume:
\begin{enumerate}
\item[(i)] Conjectures \ref{T0conj1} and \ref{T0conj2} hold. 
\item[(ii)] $\Ext^1(\ccR,\ccR)^{T_0} = \Ext^2(\ccR,\ccR)^{T_0} = 0$.
\item[(iii)] All the connected components $\cC \subset \bigsqcup_{c_3} \M_{Y \subset X}^{H}(2,c_1,c_2,c_3)^{T_0}$ are compact and have constant reflexive hulls $\ccR$.
\end{enumerate}
For each $\alpha\beta \in E_c(Y)$, there are two faces $\rho_{1,\alpha\beta}$, $\rho_{2,\alpha\beta}$ which share $\alpha\beta$ as an edge and two disjoint faces $\rho_{3,\alpha\beta}$, $\rho_{4,\alpha\beta}$ connected by the edge $\alpha\beta$. Then
\begin{align*}
&\sum_{c_3} \sum_{\cC \subset \M_{Y \subset X}^{H}(2,c_1,c_2,c_3)^T} \DT(\cC)\Big|_{s_1+s_2+s_3=0} q^{c_3} = \\
&\sum_{\hat{\bslambda}} \sfZ_{Y,\ccR|_Y, \hat{\bslambda}}(q^{-2}) q^{c_{3}(\ccR)} \prod_{\alpha\beta \in E_c(Y)} (-1)^{|\bslambda_{\alpha\beta}| (m_{\alpha\beta} (v_{\rho_{1,\alpha\beta}} + v_{\rho_{2,\alpha\beta}} + 1) +v_{\rho_{3,\alpha\beta}} + v_{\rho_{4,\alpha\beta}})} q^{|\bslambda_{\alpha\beta}| C_{\alpha\beta} (c_1(X) + c_1)} ,
\end{align*}
where $\cC$ runs over all connected components and $\hat{\bslambda} = \{\bslambda_{\alpha\beta} \in \Lambda(\ccR_{\alpha\beta})\}_{\alpha\beta \in E_c(Y)}$ runs over all double square configurations giving rise to second Chern class $c_2$.
\end{theorem}
\begin{proof}
The LHS can be written as
$$
\sum_{c_3} \sum_{\cC \subset  \M_{Y \subset X}^{H}(2,c_1,c_2,c_3)^{T_0}} \DT(\cC) q^{c_3},
$$
by Proposition \ref{T0DT=TDT} and where the sum is over all connected components of the $T_0$-fixed locus. 
By Theorem \ref{T0thm1}
$$
\DT(\cC) = (-1)^{\dim(\cC)+\rk(N^+)} e(\cC) = (-1)^{\dim(\cC)+\rk(N^+)} \sum_{\cC_i \subset \cC^{\C^*}} e(\cC_i),
$$
where the sum is over all connected components of $\cC^{\C^*}$. Such a component $\cC_i \cong \cC_{\hat{\bspi}}$ for some $\hat{\bspi} \in \Pi(\ccR|_Y,\hat{\bslambda})$ for some $\hat{\bslambda} = \{\bslambda_{\alpha\beta} \in \Lambda(\ccR_{\alpha\beta})\}_{\alpha\beta \in E_c(Y)}$ giving rise to second Chern class $c_2$. If we can show
$$
\dim(\cC)+\rk(N^+|_{\cC_{\hat{\bspi}}}) = \sum_{\alpha\beta \in E_c(Y)} |\bslambda_{\alpha\beta}| (m_{\alpha\beta} (v_{\rho_{1,\alpha\beta}} + v_{\rho_{2,\alpha\beta}} + 1) +v_{\rho_{3,\alpha\beta}} + v_{\rho_{4,\alpha\beta}}) \mod 2,
$$
then we obtain RHS.

We evaluate $\dim(\cC)+\rk(N^+)$ on an arbitrary closed point $[\F] \in \cC_{\hat{\bspi}}$. By definition of $N^{\pm}$, $\sfV^{\pm}_{\alpha}$, $\sfE^{\pm}_{\alpha\beta}$ 
\begin{align*}
\rk(N^+|_{[\F]}) =&\sum_{\alpha \in V(Y)} (\sfV_\alpha^{+}|_{[\F]}(1,1,1) - \mathbf{Constant}(\sfV_\alpha^{+}|_{[\F]})) \\
&+ \sum_{\alpha\beta \in E_c(Y)} (\sfE_{\alpha\beta}^{+}|_{[\F]}(1,1,1) - \mathbf{Constant}(\sfE_{\alpha\beta}^{+}|_{[\F]})). 
\end{align*}
In the proof of Theorem \ref{Tthm}, we computed
\begin{align*}
&\sum_{\alpha \in V(Y)} \sfV_\alpha^{+}|_{[\F]}(1,1,1) + \sum_{\alpha\beta \in E_c(Y)} \sfE_{\alpha\beta}^{+}|_{[\F]}(1,1,1) \\
&= \sum_{\alpha\beta \in E_c(Y)} |\bslambda_{\alpha\beta}| (m_{\alpha\beta} (v_{\rho_{1,\alpha\beta}} + v_{\rho_{2,\alpha\beta}} + 1) +v_{\rho_{3,\alpha\beta}} + v_{\rho_{4,\alpha\beta}}) \mod 2.
\end{align*}
The theorem follows from 
\begin{align*}
\dim(\cC) = \sum_{\alpha \in V(Y)} \mathbf{Constant}(\sfV_\alpha^{+}|_{[\F]}) + \sum_{\alpha\beta \in E_c(Y)} \mathbf{Constant}(\sfE_{\alpha\beta}^{+}|_{[\F]}) \mod 2,
\end{align*}
which is precisely Conjecture \ref{T0conj2}.
\end{proof}

\begin{remark} \label{finalrem}
Let $Y$ be a smooth toric Calabi-Yau 3-fold with toric compactification $X$ satisfying $H^0(K_{X}^{-1}) \neq 0$. Let $\ccR$ be a rank 2 $\mu$-stable reflexive sheaf on $X$ with respect to some polarization. Assume: 
\begin{enumerate}
\item[(i)] Conjectures \ref{T0conj1}, \ref{T0conj2} hold.
\item[(ii)] $\Ext^1(\ccR,\ccR)^{T_0} = \Ext^2(\ccR,\ccR)^{T_0} = 0$. Let $\cC \subset \M_{Y \subset X}^{T_0}$ be a compact connected component with constant reflexive hulls $\ccR$ \emph{and} all cokernels 0-dimensional.
\end{enumerate}
Then $\DT(\cC) = (-1)^{\dim(\cC)+\rk(N^+)} e(\cC)$.
By the calculation in the proof of the previous theorem, $(-1)^{\dim(\cC)+\rk(N^+)} = 1$ (no legs). We deduce that in the absence of legs
$$
\DT(\cC) = \sum_{\cC_i \subset \cC^{\C^*}} \DT(\cC) = \sum_{\cC_i \subset \cC^{\C^*}} e(\cC_i) = e(\cC).
$$
This formula holds even though in general $\DT(\cC_i) \neq e(\cC_i)$ (Example \ref{abnormal}).
\end{remark}

\noindent {\tt{amingh@umd.edu}}, {\tt{m.kool1@uu.nl}}, {\tt{bjy@uoregon.edu}}
\end{document}